\documentclass[a4paper,10pt]{amsart}

\usepackage[
  margin=30mm,
  marginparwidth=25mm,     % + <- Width of your marginpar
  marginparsep=2mm,       % + <- Gap between text block and marginpar
  bottom=25mm,
  ]{geometry}

\usepackage[bbgreekl]{mathbbol}
\usepackage{amsfonts}
\usepackage{latexsym,amssymb,amsthm,mathrsfs,amsmath,amscd,enumerate,enumitem, amscd,color}
\usepackage{mathrsfs}
\usepackage{tikz}

\DeclareSymbolFontAlphabet{\mathbb}{AMSb}
\DeclareSymbolFontAlphabet{\mathbbl}{bbold}
\DeclareFontEncoding{FMS}{}{}
\DeclareFontSubstitution{FMS}{futm}{m}{n}
\DeclareFontEncoding{FMX}{}{}
\DeclareFontSubstitution{FMX}{futm}{m}{n}
\DeclareSymbolFont{fouriersymbols}{FMS}{futm}{m}{n}
\DeclareSymbolFont{fourierlargesymbols}{FMX}{futm}{m}{n}
\DeclareMathDelimiter{\VERT}{\mathord}{fouriersymbols}{152}{fourierlargesymbols}{147}

\newtheorem{thm}{Theorem}[section]
 \newtheorem{cor}[thm]{Corollary}

\theoremstyle{definition}

 \theoremstyle{remark}
 \newtheorem{rem}[thm]{Remark}

\usepackage{hyperref}

\newcommand{\N}{\mathbb{N}}

\newcommand{\R}{\mathbb{R}}
\newcommand{\Rn}{\mathbb{R}^n}
\newcommand{\supp}{\mathop{\mathrm{supp}}}

\newcommand{\eps}{\varepsilon}

\newcommand{\CcR}{C_{c}^{\infty}(\Rn)}
\newcommand{\Cc}{C_{c}^{\infty}}
\newcommand{\LpO}{L^{p}(\Omega)}
\newcommand{\LpOB}{L^{p}(\Omega,B)}
\newcommand{\LpR}{L^{p}(\Rn)}
\newcommand{\LpRB}{L^{p}(\Rn,B)}
\newcommand{\co}{((0,\infty), \frac{dt}{t})}
\newcommand{\coN}{((\frac{1}{N},\infty), \frac{dt}{t})}

\numberwithin{equation}{section}
\allowdisplaybreaks
\begin{document}

\footnotetext{Last modification: \today.}

\title[]
 {Uniformly convex and smooth Banach spaces and $L^p$-boundedness properties of Littlewood-Paley and area functions associated with semigroups}

\author[J.J. Betancor]{Jorge J. Betancor}
\address{Jorge J. Betancor, Juan C. Fari\~na\newline
	Departamento de An\'alisis Matem\'atico, Universidad de La Laguna,\newline
	Campus de Anchieta, Avda. Astrof\'isico S\'anchez, s/n,\newline
	38721 La Laguna (Sta. Cruz de Tenerife), Spain}
\email{jbetanco@ull.es;
jcfarina@ull.es
}

\author[J.C. Fari\~na]{Juan C. Fari\~na}

\author[V. Galli]{Vanesa Galli}
\address{Vanesa Galli, Sandra M. Molina\newline
Departamento de Matem\'atica, Facultad de Ciencias Exactas y Naturales,\newline
Universidad Nacional de Mar del Plata, \newline
Funes 3350, 7600 Mar del Plata,
Argentina.}
\email{vggalli@mdp.edu.ar; smolina@mdp.edu.ar}

\author[S.M. Molina] {Sandra M. Molina}
%\author{L. Rodr\'{\i}guez-Mesa}

\thanks{The authors are partially supported by MTM2016-79436-P}

\subjclass[2010]{43A15, 46B20, 47D03}

\keywords{Uniformly convex, uniformly smooth, Littlewood-Paley, area integrals.}

\date{}

%%% ----------------------------------------------------------------------

\begin{abstract}
In this paper we obtain new characterizations of the uniformly convex and smooth Banach spaces. These characterizations are established by using $L^{p}$-boundedness properties of Littlewood-Paley functions and area integrals associated with heat semigroups and involving fractional derivatives. Our results apply for instance by considering the heat semigroups defined by Hermite and Laguerre operators that do not satisfy the Markovian property.
\end{abstract}

\dedicatory{Dedicated to Professor Fernando P\'erez Gonz\'alez on the occasion of his retirement}

\maketitle
%
%\begin{flushright}\emph{Dedicated to Professor Fernando P\'erez Gonz\'alez on the occasion of his retirement}
%\end{flushright}

\section{Introduction}
It is a well known fact that the geometry of Banach spaces is closely related to vector valued harmonic analysis.  We study in this paper  some questions concerning to this connection. Our problems are included in the following general context. Suppose that $(\Omega, \mu)$ is a measure space and $T$ is a bounded operator from $L^p(\Omega) \equiv L^{p}(\Omega, \mu)$ into itself, where $1<p<\infty$.  We can think that $T$ is, for instance, a bounded singular integral in $\LpR$. By $B$ we denote a Banach space. The question is if the operator $T$ can be extended to the B\"{o}chner Lebesgue space $L^{p}(\Omega,\mu;B)$ as a bounded operator from $L^p(\Omega,B) \equiv L^{p}(\Omega,\mu;B)$ into itself. If the answer to the last question is not affirmative for every Banach space $B$, then the objective is to describe (geometric) properties that characterize those Banach spaces $B$ for which $T$ can be boundly extended to $L^{p}(\Omega,\mu;B)$.

The Banach space $B$ such that the Hilbert transform $\mathcal{H}$ in $\R$ can be extended to $\LpRB$ as a bounded operator from $\LpRB$ into itself with $1<p<\infty$, were characterized by Bourgain (\cite{Bo}) and Burkholder (\cite{Bu}) (see also \cite{Rub}) as those ones that satisfy the UMD (Unconditional Martingale Difference) property or, equivalently, those ones that are $\zeta$-convex (\cite{Bu1}). Other characterizations of UMD Banach spaces by using $L^{p}$-boundedness properties of singular integrals and multipliers were established in \cite{AT1}, \cite{BCCR}, \cite{BCFR}, \cite{Gu}, and \cite{Xu98}.

Our results concern to uniformly convex and smooth Banach spaces. We now recall definitions (see \cite{LT} for details). Let $B$ be a Banach space. The modulus of convexity of $B$ is defined by

$$
\delta_{B}(\eps)=\inf
\left\{
1-\left\|\frac{a+b}{2}\right\|: \quad a,b\in B, \:\|a\|=\|b\|=1, \:\|a-b\|=\eps
\right\},
\quad 0<\eps<2.
$$
The modulus of smoothness of $B$ is given as follows
$$
\rho_{B}(t)=\sup
\left\{\frac{\|a+tb\|+\|a-tb\|}{2}-1: \quad a,b\in B, \:\|a\|=\|b\|=1
\right\},
\quad t>0.
$$
Here $\|\cdot\|$ denotes the norm in $B$.

We say that $B$ is uniformly convex when $\delta_{B}(\eps)>0$, for every $0<\eps<2$, and that $B$ is uniformly smooth when $\lim\limits_{t\to 0}\rho_{B}(t)/t=0$. If $q>1$, $B$ is said to be $q$-uniformly convex (respectively $q$-uniformly smooth) provided that there exists $C>0$ such that $\delta_{B}(\eps)\geq C\eps^{q}$, $0<\eps<2$ (respectively, $\rho_{B}(t)\leq Ct^{q}$, $t>0$).

The notions of martingale type and cotype of a Banach space were introduced by Pisier (\cite{Pi5} and \cite{Pi2}). We say that $B$ has martingale cotype (respectively, type) $q$, when there exists $C>0$ such that for every martingale $(M_{n})_{n\in\N}$ on a certain probability space with values in $B$
$$
\sum_{n=1}^{\infty} E\|M_{n}-M_{n-1}\|^{q}+E\|M_{0}\|^{q}
\leq C\sup_{n\in\N} E\|M_{n}\|^{q},
$$
(respectively, $\sup\limits_{n\in\N} E\|M_{n}\|^{q}\leq C\left(\sum\limits_{n=1}^{\infty} E\|M_{n}-M_{n-1}\|^{q}+E\|M_{0}\|^{q}\right)$).
Here $E$ denotes, the corresponding expectation. We recall that if $B$ has martingale cotype (respectively, type) $q$, then $2\leq q<\infty$ (respectively, $1<q\leq 2$).

If $M=(M_{n})_{n\in\N}$ is a martingale on some probability space with values in $B$ and $1<q<\infty$, the $q$-square function $S_{q}(M)$ of $M$ is defined by
\[S_{q}(M)=\left(\sum_{n=1}^{\infty} \|M_{n}-M_{n-1}\|^{q}+\|M_{0}\|^{q}\right)^{1/q}.\]

The square function $S_{q}$
 allows us to characterize the $q$-uniformly convex and smooth Banach space (\cite[Theorem 4.51 and 4.52]{Pi3}). The Banach space $B$ is of martingale cotype (respectively, type) $q$ if, and only if, for every $1<p<\infty$ (or, equivalently, for some $1<p<\infty$), there exists $C>0$ such that, for every $L^{p}$-martingale $M$ with values in $B$,
$$
 E[S_{q}(M)]^{p}\leq C \sup_{n\in\N} E\|M_{n}\|^{p}
$$
(respectively, $\sup\limits_{n\geq 0} E\|M_{n}\|^{p}\leq C E[S_{q}(M)]^{p}$).

We define, for every $f\in\LpR$, $1\leq p\leq \infty$, and $t>0$,
$$
W_{t}(f)(x)=\frac{1}{(4\pi)^{n/2}}\int_{\Rn}
\frac{e^{-|x-y|^{2}/4t}}{t^{n/2}} f(y)\:dy, \quad x\in\Rn,
$$
and
$$
P_{t}(f)(x)=\frac{\Gamma(\frac{n+1}{2})}{\pi^{(n+1)/2}}\int_{\Rn}
\frac{t}{(t^{2}+|x-y|^{2})^{\frac{n+1}{2}}} f(y)\:dy, \quad x\in\Rn.
$$
$\{W_{t}\}_{t>0}$ and $\{P_{t}\}_{t>0}$ represent the classical heat and Poisson semigroup, respectively.

$T_{t}$ represents now to $W_{t}$ and $P_{t}$, for every $t>0$. The (vertical) Littlewood-Paley function $g_{q,T_{t}}$, $1<q<\infty$, associated
with $\{T_{t}\}_{t>0}$ is defined by
$$
g_{q,T_{t}}(f)(x)=\left(
\int_{0}^{\infty}|t\partial_{t}T_{t}(f)(x)|^{q} \frac{dt}{t}
\right)^{1/q},\quad x\in\Rn.
$$

It is well known that, there exists $C>0$ such that, for every $1<p<\infty$,
$$
\frac{1}{C}\|f\|_{\LpR}\leq
\|g_{2, T_{t}}(f)\|_{\LpR}\leq
C\|f\|_{\LpR},\quad f\in \LpR.
$$

Since, for every $t>0$ and $1<p<\infty$, $T_{t}$ is a positive bounded operator from $\LpR$ into itself, the tensor extension $T_{t}\otimes I_{B}$ is bounded from $\LpRB$ into itself and $\|T_{t}\otimes I_{B}\|_{\LpRB\to\LpRB}=\|T_{t}\|_{\LpR\to\LpR}$. Thus, the family $\{T_{t}\}_{t>0}$ can be also seen as a bounded semigroup of operators in $\LpRB$, $1<p<\infty$.

For every $f\in\LpRB$, $1<p<\infty$, and $1<q<\infty$, we define
\begin{equation}\label{1.1}
g_{q,T_{t};B}(f)(x)=\left(
\int_{0}^{\infty}\|t\partial_{t}T_{t}(f)(x)\|^{q}\frac{dt}{t}
\right)^{1/q}, \quad x\in\Rn.
\end{equation}
Here $\|\cdot\|$ denotes the norm in $B$.

It is well known (see \cite{Kw}) that, for some $1<p<\infty$,
\begin{equation}\label{1.2}
\frac{1}{C}\|f\|_{\LpRB}\leq
\|g_{2, T_{t};B}(f)\|_{\LpR}\leq
C\|f\|_{\LpRB},\quad f\in \LpRB,
\end{equation}
if and only if $B$ is isomorphic to a Hilbert space. Then, it is an interesting question to characterize, for instance in a geometric way, those Banach space $B$ for which one of the two inequalities in \eqref{1.2} holds.

By taking as a starting point the Pisier's characterizations of uniformly convex and smooth Banach spaces by martingale square functions, Xu (\cite{Xu98}) obtained a version of the Pisier's results by using Littlewood-Paley functions $g_{q,P_{t};B}$ associated with the Poisson semigroup $\{P_{t}\}_{t>0}$. Actually Xu (\cite{Xu98}) considered Poisson semigroup in the unit disc but the corresponding results hold for $\{P_{t}\}_{t>0}$ (see also \cite{MTX}).

\begin{thm}[\cite{Xu98}]\label{thm1.1}
Let $B$ be a Banach space and $1<p<\infty$.
\begin{itemize}
\item [(a)] Assume that $2\leq q<\infty$. Then, there exists $C>0$ such that
\[
\|g_{q,P_{t};B}(f)\|_{\LpR}
\leq C \|f\|_{\LpRB}, \quad f\in\LpRB,
\]
if and only if, there exists a norm $\left\VERT\cdot\right\VERT$ on $B$ that is equivalent to $\|\cdot\|$ and such that $(B,\left\VERT\cdot\right\VERT)$ is $q$-uniformly convex.

\item [(b)] Assume that $1<q\leq 2$. Then, there exists $C>0$ such that
\[
\|f\|_{\LpRB}
\leq C \|g_{q,P_{t};B}(f)\|_{\LpR}, \quad f\in\LpRB,
\]
if and only if, there exists a norm $\left\VERT\cdot\right\VERT$ in $B$ that is equivalent to $\|\cdot\|$ and such that $(B, \left\VERT\cdot\right\VERT)$ is $q$-uniformly smooth.
\end{itemize}
\end{thm}

Note that Theorem \ref{thm1.1} has the same flavour than \cite[Theorem 4.51 and 4.52]{Pi3}.

Theorem \ref{thm1.1} was extended to subordinated Poisson symmetric diffusion semigroup by Mart\'inez, Torrea and Xu (\cite{MTX}).

Let $(\Omega,\mathcal{A},\mu)$ be a $\sigma$-finite measure space. A uniparametric family $\{T_{t}\}_{t>0}$ of linear mappings from $\LpO$ into itself, for every $1\leq p\leq\infty$, is a symmetric diffusion semigroup in the Stein's sense when the following properties are satisfied

\begin{itemize}
\item[(i)] $T_{t}$ is a contraction on $\LpO$, for every $1\leq p\leq\infty$ and $t>0$;
\item[(ii)] $T_{t+s}=T_{t}T_{s}$, on $\LpO$, for every $1\leq p\leq\infty$ and $t,s>0$;
\item[(iii)] $\lim\limits_{t\to 0^{}+}T_{t}f=f$, in $L^{2}(\Omega)$, for every $f\in L^{2}(\Omega)$;
\item[(iv)] $T_{t}$ is selfadjoint in $L^{2}(\Omega)$;
\item[(v)] $T_{t}$ is positive preserving, that is, $T_{t}f\geq 0$ when $f\geq 0$, for every $t>0$;
\item[(vi)] $T_{t}$ is markovian, that is, $T_{t}1=1$, for every $t>0$.
\end{itemize}

Property (v) allows us to extend $T_{t}$ to $\LpO\otimes B$ and then to $\LpOB$ as a bounded operator from $\LpOB$ into itself, for every $1\leq p<\infty$ and $t>0$. Thus, $\{T_{t}\}_{t>0}$ is a semigroup of contractions in $\LpOB$, for every $1\leq p<\infty$.

The Poisson subordinated semigroup $\{P_{t}\}_{t>0}$ to $\{T_{t}\}_{t>0}$ is defined as follows: for every $t>0$,
\[
P_{t}f=\frac{1}{\sqrt{\pi}}\int_{0}^{\infty}\frac{e^{-s}}{\sqrt{s}} T_{s^{2}/4s}f\;ds.
\]
Thus, $\{P_{t}\}_{t>0}$ is also a symmetric diffusion semigroup.

For every $1<q<\infty$, the Littlewood-Paley functions $g_{q,T_{t};B}$ and $g_{q,P_{t};B}$ for a symmetric diffusion semigroup $\{T_{t}\}_{t>0}$ and its Poisson subordinated semigroup is defined as in \eqref{1.1}

We denote by $\mathbb{F}$ the subspace of $L^{2}(\Omega)$ that consists of all the fixed points of $\{T_{t}\}_{t>0}$, i.e., of all those $f\in L^{2}(\Omega)$ such that $T_{t}f=f$, for every $t>0$. Note that $\mathbb{F}$  coincides with the  subspace of $L^{2}(\Omega)$ constituted by all those $f\in D(A)$, where $A$ is the infinitesimal generator of $\{T_{t}\}_{t>0}$ and $D(A)$ is the domain of $A$, such that $Af=0$. By $\mathcal{F}:L^{2}(\Omega)\to\mathbb{F}$ we define the orthogonal projection from $L^{2}(\Omega)$ onto $\mathbb{F}$. It is a classical fact that $\mathcal{F}$ can be extended to a contractive projection (that we will continue denoting by $\mathcal{F}$) from $\LpO$ onto $\mathcal{F}(\LpO)$, for every $1\leq p<\infty$. Here, $\mathcal{F}(\LpO)$ consists of all the fixed points of $\{T_{t}\}_{t>0}$ in $\LpO$, $1\leq p<\infty$. Also, for every Banach space $B$ and $1\leq p<\infty$, $\mathcal{F}$ extends to a contractive projection from $\LpOB$ onto $\mathcal{F}(\LpOB)$, where $\mathcal{F}(\LpOB)$ is the subspace of $\LpOB$ constituted by all the fixed points of $\{T_{t}\}_{t>0}$ in $\LpOB$.

The following extension of Theorem \ref{thm1.1} was established by \cite{MTX}.

\begin{thm}[\cite{MTX}]\label{thm1.2}
Let $B$ be a Banach space and $1<p<\infty$. Suppose that $\{P_{t}\}_{t>0}$ is a Poisson subordinated semigroup to a symmetric diffusion semigroup.
\begin{itemize}
\item [(a)] If $2\leq q<\infty$ and there exists a norm $\left\VERT\cdot\right\VERT$ on $B$ that is equivalent to $\|\cdot\|$ and such that $(B,\left\VERT\cdot\right\VERT)$ is $q$-uniformly convex, then there exists $C>0$ such that
\begin{equation}\label{1.3}
\|g_{q,P_{t};B}(f)\|_{\LpR}
\leq C \|f\|_{\LpRB}, \quad f\in\LpRB.
\end{equation}

 \item [(b)] If $1<q\leq 2$ and there exists a norm $\left\VERT\cdot\right\VERT$ on $B$ that is equivalent to $\|\cdot\|$ and such that $(B, \left\VERT\cdot\right\VERT)$ is $q$-uniformly smooth, then there exists $C>0$ such that
\begin{equation}\label{1.4}
\|f\|_{\LpRB}
\leq C \left(
\|\mathcal{F}(f)\|_{\LpRB}+\|g_{q,P_{t};B}(f)\|_{\LpR}
\right), \quad f\in\LpRB.
\end{equation}
\end{itemize}

\end{thm}

\bigskip
When \eqref{1.3} (respectively, \eqref{1.4}) is satisfied is usually said that $B$ has Lusin cotype (respectively, type) $q$ with respect to the semigroup $\{P_{t}\}_{t>0}$.

In order to prove Theorem \ref{thm1.2} in \cite{MTX} a result due to Rota \cite[Chapter 4]{SteinLp} is used where the markovian property (vi) above is needed. However, Theorems \ref{thm1.1} and \ref{thm1.2} can be established by using Poisson semigroups associated with Bessel operators (\cite{BFMT}) and Laguerre operators (\cite{BFRST}) that are not Markovian.

In order to prove the results in the Bessel and Laguerre setting an argument that take advantage of the fact that Bessel and Laguerre operators are nice perturbations of the Laplace operator is used. This fact allows us to establish some tricky estimates for the difference of the Littlewood-Paley functions defined for the Laguerre and Bessel Poisson semigroups and the classical Poisson semigroups close to the diagonal, in the so called local region.

Torrea and Zhang (\cite{TZ}) extended the results in \cite{MTX} by using Littlewood-Paley functions involving fractional derivatives.

Let $\alpha>0$. We choose $m\in\N$ such that $m-1\leq\alpha < m$. If $\phi\in C^{m}(0,\infty)$ the $\alpha$-derivative $\partial_t^{\alpha}\phi(t)$ of $\phi$ in $t\in (0,\infty)$ is given by
$$
\partial_{t}^{\alpha}\phi(t)=
\frac{1}{\Gamma(m-\alpha)}\int_{t}^{\infty}
(\partial_u^{m}\phi)(u)(u-t)^{m-\alpha-1}du,
$$
provided that the last integral exists. The fractional derivative is usually named Weyl derivative.

\bigskip
Suppose that $\{T_{t}\}_{t>0}$ is a symmetric diffusion semigroup in $(\Omega,\mathcal{A},\mu)$. Then, for every $\alpha>0$, $f\in\LpO$, $1\leq p<\infty$, and $t>0$, we have that $\int_{t}^{\infty}\|\partial_{u}^{m}T_{u}(f)\|_{\LpO}(u-t)^{m-\alpha-1} du<\infty$, where $m-1\leq\alpha < m$ (see Section 2), and we define

\begin{equation}\label{1.5}
\partial_{t}^{\alpha}T_{t}(f)=\frac{1}{\Gamma(m-\alpha)}
\int_{t}^{\infty}\partial_{u}^{m}T_{u}(f)(u-t)^{m-\alpha-1} du.
\end{equation}
Also $\partial_{t}^{\alpha}T_{t}(f)$ is defined by \eqref{1.5} when $f\in\LpO\otimes B$, for every $\alpha,t>0$.

For every $\alpha>0$ and  $1<q<\infty$, the fractional Littlewood-Paley function $g_{q,T_{t};B}^{\alpha}$ associated with $\{T_{t}\}_{t>0}$ is defined by

$$
g_{q, T_{t};B}^{\alpha}(f)(x)=
\left(
\int_{0}^{\infty}
\|t^{\alpha}\partial_{t}^{\alpha}T_{t}f(x)\|^{q}
\frac{dt}{t}
\right)^{1/q}, \quad f\in\LpO\otimes B,\: 1\leq p<\infty.
$$
\begin{thm}(\cite[Theorems 1.1 and 1.2]{TZ})\label{thm1.3}
Let $B$ be a Banach space, $1<p<\infty$ and $\alpha>0$. Suppose that $\{P_{t}\}_{t>0}$ is a Poisson subordinated semigroup to a symmetric diffusion semigroup.
\begin{itemize}
\item[(a)] If $2\leq q<\infty$ and there exists a norm $\left\VERT\cdot\right\VERT$ on $B$ that is equivalent to $\|\cdot\|$ and such that  $(B,\left\VERT\cdot\right\VERT)$ is $q$-uniformly convex, then there exists $C>0$ such that
\begin{equation}\label{1.6}
\|g_{q, P_{t};B}^{\alpha}(f)\|_{\LpO}
\leq C \|f\|_{\LpOB}, \quad f\in\LpOB.
\end{equation}
\item[(b)] If $1<q\leq 2$ and there exists a norm $\left\VERT\cdot\right\VERT$ on $B$  that is equivalent to $\|\cdot\|$ and such that $(B,\left\VERT\cdot\right\VERT)$ is $q$-uniformly smooth, then there exists $C>0$ such that
\begin{equation}\label{1.7}
\|f\|_{\LpOB}
\leq C \left(
\|\mathcal{F}(f)\|_{\LpOB}+\|g_{q,P_{t};B}^{\alpha}(f)\|_{\LpO}
\right), \quad f\in\LpOB.
\end{equation}
\item[(c)] If $2\leq q<\infty$ and \eqref{1.6} holds when $\{P_{t}\}_{t>0}$ is the classical Poisson semigroup in $\Rn$, then there exists a norm $\left\VERT\cdot\right\VERT$ on $B$ that is equivalent to $\|\cdot\|$ and such that $(B,\left\VERT\cdot\right\VERT)$ is $q$-uniformly convex.
\item[(d)] If $1< q\leq 2$ and \eqref{1.7} holds when $\{P_{t}\}_{t>0}$ is the classical Poisson semigroup in $\Rn$, then there exists a norm $\left\VERT\cdot\right\VERT$ on $B$ that is equivalent to $\|\cdot\|$ and such that $(B,\left\VERT\cdot\right\VERT)$ is $q$-uniformly smooth.
\end{itemize}
\end{thm}

If $\alpha>0$ and $1<q<\infty$, the area integral $\mathcal{A}_{q,T_{t};B}^{\alpha}$ associated with the symmetric diffusion semigroup $\{T_{t}\}_{t}$ on $\Rn$ is defined by
$$
\mathcal{A}_{q,T_{t};B}^{\alpha}(f)(x)=
\left(
\int_{\Gamma(x)}
\|s^{\alpha}\partial^{\alpha}T_{s}f(y)_{|s=t^{2}}\|^{q}
\frac{dy\:dt}{t^{n+1}}
\right)^{1/q}, \quad f\in\LpR\otimes B,\: 1\leq p<\infty,
$$
where $\Gamma(x)=\{(y,t) \in \mathbb R^n \times (0,\infty): |y-x| < t\}$, $x\in \mathbb R^n$.

Versions of Theorem \ref{thm1.1} where the Littlewood-Paley function $g_{q,P_{t};B}$ is replaced by the area integral $\mathcal{A}_{q,T_{t};B}^{\alpha}$ were proved in \cite[p. 474]{MTX} where $\alpha=1$ and in \cite[Theorems 5.3 and 5.4]{TZ} for every $\alpha>0$.

\bigskip
In \cite[p. 474]{MTX} the following problem is posed. Is $(a)$ and $(b)$ in Theorem \ref{thm1.2} true when Poisson subordinated semigroups are replaced for any symmetric diffusion semigroup ?. Recently, Hytonen and Naor
 (\cite{HN}) proved that the answer is affirmative for the classical heat semigroup in $\Rn$ and for $p=q$. After this, the problem was solved in \cite{Xu18} establishing the following result.

\begin{thm}{(\cite[Theorem 2]{Xu18})}\label{thm1.4}
Let $B$ be a Banach space, $k\in\N$, $k\geq 1$, and  $1<p<\infty$. Suppose that $\{T_{t}\}_{t>0}$ is  a symmetric diffusion semigroup.
\begin{itemize}
\item[(a)] If $2\leq q<\infty$ and there exists a norm $\left\VERT\cdot\right\VERT$ on $B$ that is equivalent to $\|\cdot\|$ and such that  $(B,\left\VERT\cdot\right\VERT)$ is $q$-uniformly convex, then there exists $C>0$ such that
\[
\|g_{q, T_{t};B}^{k}(f)\|_{\LpO}
\leq C \|f\|_{\LpOB}, \quad f\in\LpOB.
\]

\item[(b)] If $1<q\leq 2$ and there exists a norm $|||\cdot|||$ on $B$  that is equivalent to $\|\cdot\|$ and such that $(B,\left\VERT\cdot\right\VERT)$ is $q$-uniformly smooth, then there exists $C>0$ such that
\[
\|f\|_{\LpOB}
\leq C \left(
\|\mathcal{F}(f)\|_{\LpOB}+\|g_{q,T_{t};B}^{k}(f)\|_{\LpO}
\right), \quad f\in\LpOB.
\]

\end{itemize}
\end{thm}

It is remarkable that the constants $C$ in Theorems \ref{thm1.2} and \ref{1.4} do not depend on the semigroups. They only depend on $k,p,q$ and the geometric constants of $B$. This fact allows to obtain dimension free $L^{p}$-estimates for Littlewood-Paley functions associated with semigroups in $\Rn$.

Our objective in this paper is to complete the study in \cite{Xu98} in some directions.

Our first results are concerned with uniformly convex Banach spaces.

\begin{thm}\label{thm1.5}
Let B be a Banach space, $2\leq q<\infty$, $1<p<\infty$ and $\alpha>0$.
\begin{itemize}
\item[(a)] If there exists a norm $\left\VERT\cdot\right\VERT$  on $B$ that is equivalent to $\|\cdot\|$ and such that $(B,\left\VERT\cdot\right\VERT)$ is $q$-uniformly convex and $\{T_{t}\}_{t>0}$ is a symmetric diffusion semigroup, then there exists $C>0$ such that
\[
\|g_{q,T_{t};B}^{\alpha}(f)\|_{\LpO}
\leq C \|f\|_{\LpOB}, \quad f\in\LpOB.
\]

\item[(b)] If there exists $C>0$ such that
\[
\|g_{q,W_{t};B}^{\alpha}(f)\|_{\LpR}
\leq C \|f\|_{\LpRB}, \quad f\in\LpRB.
\]
then there exists a norm $\left\VERT\cdot\right\VERT$  on $B$ that is equivalent to $\|\cdot\|$ and such that $(B,\left\VERT\cdot\right\VERT)$ is $q$-uniformly convex.
\end{itemize}
\end{thm}

\begin{thm}\label{thm1.6}
Let B be a Banach space, $2\leq q<\infty$, $1<p<\infty$ and $\alpha>0$. The following assertions are equivalent.
\begin{itemize}
\item[(a)] There exists a norm $\left\VERT\cdot\right\VERT$  on $B$ that is equivalent to $\|\cdot\|$ and such that $(B,\left\VERT\cdot\right\VERT)$ is $q$-uniformly convex.

\item[(b)] There exists $C>0$ such that
\[
\|\mathcal{A}_{q,W_{t};B}^{\alpha}(f)\|_{\LpR}
\leq C \|f\|_{\LpRB}, \quad f\in\LpRB.
\]
\end{itemize}
\end{thm}
We now establish two results related to uniformly smooth Banach spaces.

\begin{thm}\label{thm1.7}
Let $B$ be a Banach space, $1<q\leq 2$, $1<p<\infty$ and $\alpha>0$.
\begin{itemize}
\item[(a)] If there exists a norm $\left\VERT\cdot\right\VERT$ on $B$  that is equivalent to $\|\cdot\|$ and such that $(B,\left\VERT\cdot\right\VERT)$ is $q$-uniformly smooth and $\{T_{t}\}_{t>0}$ is a symmetric diffusion semigroup, then there exists $C>0$ such that

\[
\|f\|_{\LpOB}
\leq C \left(
\|\mathcal{F}(f)\|_{\LpOB}+\|g_{q,T_{t};B}^{\alpha}(f)\|_{\LpO}
\right), \quad f\in\LpOB.
\]

\item[(b)] If there exists $C>0$ such that
\[
\|f\|_{\LpRB}
\leq C \|g_{q,W_{t};B}^{\alpha}(f)\|_{\LpR}, \quad f\in\LpRB,
\]
then there exists a norm $\left\VERT \cdot\right\VERT$ on $B$  that is equivalent to $\|\cdot\|$ and such that $(B,\left\VERT\cdot\right\VERT)$ is $q$-uniformly smooth.
\end{itemize}
\end{thm}

\begin{thm}\label{thm1.8}
Let B be a Banach space, $1<q\leq 2$, $1<p<\infty$ and $\alpha>0$. The following assertions are equivalent.
\begin{itemize}
\item[(a)] There exists a norm $\left\VERT\cdot\right\VERT$  on $B$ that is equivalent to $\|\cdot\|$ and such that $(B,\left\VERT\cdot\right\VERT)$ is $q$-uniformly smooth.

\item[(b)] There exists $C>0$ such that
\[
\|f\|_{\LpRB}
\leq C \|\mathcal{A}_{q,W_{t};B}^{\alpha}(f)\|_{\LpR}, \quad f\in\LpRB.
\]
\end{itemize}
\end{thm}

Markovian property (property (vi)) is needed in the proof of the results in \cite{Xu98} (Theorem \ref{thm1.4}). We now establish characterization of uniformly convex and smooth Banach spaces via Littlewood-Paley functions and area integrals associated to heat semigroups defined by Hermite and Laguerre operators. It is remarkable that these semigroups do not satisfy the markovian property and the results in \cite{Xu98} does not apply.

The Hermite operator (also called harmonic oscillator) in $\Rn$ is defined by
$$
H(f)=-\Delta+|x|^{2},
$$
where $\Delta$ denotes the Laplace operator. For every $k=(k_1,\ldots,k_n)\in\N^{n}$ we define
$$
h_{k}(x)=\prod_{i=1}^{n} H_{k_i}(x_i),
\quad x=(x_1,\ldots,x_n)\in\Rn,
$$
where, for every $l\in\N$, $H_{l}$ represents the $l$-th Hermite function defined by
$$
H_{l}(z)=(\sqrt{\pi}2^{l}l!)^{-1/2}e^{-x^{2}/2}p_{l}(z), \quad z\in\R,
$$
being $p_{l}$ the $l$-th Hermite polynomial (\cite{Sz3}). It is well-known  that $\{h_{k}\}_{k\in\N^{n}}$ is a complete orthonormal system in $L^{2}(\Rn)$. Also, we have that
$$
Hh_{k}=\lambda_{k}h_{k},
$$
where $\lambda_{k}=2(k_1+\cdots+k_n+n)$, for every $k=(k_1,\ldots,k_n)\in\N^{n}$.

For every $f\in L^{2}(\Rn)$ and $k\in\N^{n}$ we define
$$
c_{k}(f)=\int_{\Rn}f(y)h_{k}(y)\:dy.
$$
We consider the operator $\mathfrak{H}$ defined by
\[
\mathfrak{H}(f)=\sum\limits_{k\in\N^{n}} c_{k}(f)\lambda_{k}h_{k},
\]
provided that $f\in D(\mathfrak{H})=
\left\{f\in L^{2}(\Rn):\sum\limits_{k\in\N^{n}} |\lambda_{k}|^{2}\:|c_{k}(f)|^{2}<\infty\right\}$. Note that $\mathfrak{H}(f)=H(f)$, $f \in C_c(\mathbb R^n)$ the space of smooth functions in $\mathbb R^n$ with compact support.
$\mathfrak{H}$ generates a semigroup of contractions $\{W_t^{\mathfrak{H}}\}_{t>0}$ in $L^{2}(\Rn)$. Here, for every $t>0$,
$$
W_{t}^{\mathfrak{H}}(f)=
\sum\limits_{k\in\N^{n}}
e^{-\lambda_{k}t}c_{k}(f)h_{k},
\quad f\in L^{2}(\Rn).
$$
We can write, for every $t>0$ and $f\in L^{2}(\Rn)$,
\begin{equation}\label{1.8}
W_{t}^{\mathfrak{H}}(f)(x)=
\int_{\Rn}W_{t}^{\mathfrak{H}}(f)(x,y)f(y)\:dy,
\quad x\in\Rn,
\end{equation}
where, for every $x,y \in \mathbb R^n$ and $t>0$,
$$
W_{t}^{\mathfrak{H}}(x,y)= \frac{1}{\pi^{\frac{n}{2}}} \left( \frac{e^{-2t}}{1-e^{-4t}}\right)^{ n/2} \exp\left(-\frac{1}{4}\left(|x-y|^2\frac{1+e^{-2t}}{1-e^{-2t}}+ |x+y|^2\frac{1 -e^{-2t}}{1 + e^{-2t}}\right)\right).
$$
%(2\pi\sinh(2t))^{-n/2}
%\exp\left(-\frac{1}{4}
%\left(\tanh(t)|x+y|^{2}+
%\coth(t)|x-y|^{2}\right)\right), \quad x,y\in\Rn\quad\text{and}\quad t>0.

Integral representation \eqref{1.8} allows to define a semigroup of contractions $\{W_{t}^{\mathfrak{H}}\}_{t>0}$ in $L^{p}(\Rn)$, for every $1\leq p\leq\infty$.

According to \cite[Proposition 3.3]{ST1}, we have that
$$
W_{t}^{\mathfrak{H}}(1)(x)=
(2\pi\cosh(2t))^{-n/2}
\exp\left(-\frac{1}{2}
\tanh(2t)|x|^{2}\right), \quad x\in\Rn\quad\text{and}\quad t>0.
$$
Hence $\{W_{t}^{\mathfrak{H}}\}_{t>0}$ does not satisfy the markovian property.

Littlewood-Paley functions and area integrals in the Hermite setting were studied in \cite{ST2} and \cite{BMR}, respectively.

We establish the following characterization of uniformly convex and smooth Banach spaces via Littlewood-Paley functions and area integrals associated with the Hermite heat semigroup.

\begin{thm}\label{thm1.9}
Let $B$ be a Banach space, $2\leq q<\infty$, $1<p<\infty$, and $\alpha>0.$ The following assertions are equivalent.
\begin{itemize}
\item [(a)] There exists a norm $\left\VERT\cdot\right\VERT$ on $B$ that is equivalent to $\|\cdot\|$ and such that $(B,\left\VERT\cdot\right\VERT)$ is $q$-uniformly convex.

\item[(b)] There exists $C>0$ such that
$$
\|g_{q,W_{t}^{\mathfrak{H}};B}^{\alpha}(f)\|_{\LpR}
\leq C \|f\|_{\LpRB},\quad f\in\LpRB.
$$

\item[(c)] There exists $C>0$ such that
$$
\|\mathcal{A}_{q,W_{t}^{\mathfrak{H}};B}^{\alpha}(f)\|_{\LpR}
\leq C \|f\|_{\LpRB},\quad f\in\LpRB.
$$
\end{itemize}
\end{thm}

\begin{thm}\label{thm1.10}
Let $B$ be a Banach space, $1<q\leq 2$, $1<p<\infty$, and $\alpha>0.$ The following assertions are equivalent.
\begin{itemize}
\item [(a)] There exists a norm $\left\VERT\cdot\right\VERT$ on $B$ that is equivalent to $\|\cdot\|$ and such that $(B,\left\VERT\cdot\right\VERT)$ is $q$-uniformly smooth.

\item[(b)] There exists $C>0$ such that
\[
\|f\|_{\LpRB}\leq C
\|g_{q,W_{t}^{\mathfrak{H}};B}^{\alpha}(f)\|_{\LpR},
\quad f\in\LpRB.
\]

\item[(c)] There exists $C>0$ such that
\[
\|f\|_{\LpRB} \leq C
\|\mathcal{A}_{q,W_{t}^{\mathfrak{H}};B}^{\alpha}(f)\|_{\LpR}, \quad f\in\LpRB.
\]
\end{itemize}
\end{thm}

Let $\beta>-1/2$. By $L_{\beta}$ we denote the Laguerre operator on $(0,\infty)$
\[
L_{\beta}=
\frac{1}{2}
\left(
-\frac{d^{2}}{dx^{2}}+x^{2}+\frac{\beta^{2}-1/4}{x^{2}}
\right), \quad x\in (0,\infty).
\]

For every $k\in\N$ we have that
$$
L_{\beta}(\varphi_{k}^{\beta})=
\lambda_{k}^{\beta}\varphi_{k}^{\beta},
$$
where $\lambda_{k}^{\beta}=2k+\beta+1$, and
$$
\varphi_{k}^{\beta}(x)=
\left(
\frac{2\Gamma(k+1)}{\Gamma(k+1+\beta)}
\right)^{1/2}
e^{-x^{2}/2}x^{\beta+1/2}L_{k}^{\beta}(x^{2}),
\quad x\in (0,\infty),
$$
and $L_{k}^{\beta}$ in the $k$-th polynomial of type $\beta$ (\cite[p. 100]{Sz3}). The system $\{\varphi_{k}^{\beta}\}_{k\in\N}$ is complete and orthonormal in $L^{2}(0,\infty)$. We define, for every $k\in\N$,
$$
c_{k}^{\beta}(f)=
\int_{0}^{\infty}
f(y)\:\varphi_{k}^{\beta}(y)\: dy,
\quad f\in L^{2}(0,\infty).
$$
We consider the operator $\mathcal{L}_{\beta}$ defined by
\[
\mathcal{L}_{\beta}(f)=\sum_{k=0}^{\infty}
c_{k}(f)\lambda_{k}^{\beta}\varphi_{k}^{\beta},
\]
for every $f\in D(\mathcal{L}_{\beta})=
\left\{g\in L^{2}(0,\infty):
\sum\limits_{k=0}^{\infty}
|c_{k}^{\beta}(g)|^{2}(\lambda_{k}^{\beta})^{2}<\infty
\right\}$. Note that $L_\beta(f) = \mathcal L_\beta(f)$, $ f\in C^\infty_c(0,\infty)$, the space of smooth functions with compact support in $(0,\infty)$.

The operator $\mathcal{L}_{\beta}$ generates the semigroup $\{W_{t}^{\mathcal L_\beta}\}_{t>0}$ of operators, being for every $t>0$,
$$
W_{t}^{\mathcal L_\beta}(f)=
\sum_{k=0}^{\infty}
e^{-\lambda_{k}t}
c_{k}(f)
\varphi_{k}^{\beta},
\quad f\in L^{2}(0,\infty).
$$

We can write, for every $t>0$,
\begin{equation}\label{1.9}
W_{t}^{\mathcal L_\beta}(f)(x)=
\int_{0}^{\infty} W_{t}^{\beta}(x,y)\:f(y)\:dy,
\quad f\in L^{2}(0,\infty),
\end{equation}
where, for every $t,x,y \in (0,\infty)$,
$$
W_{t}^{\mathcal L_\beta}(x,y)=
\left(
\frac{2e^{-t}}{1-e^{-2t}}
\right)^{1/2}
\left(
\frac{2xye^{-t}}{1-e^{-2t}}
\right)^{1/2}
I_{\beta}
\left(
\frac{2xye^{-t}}{1-e^{-2t}}
\right)
\exp
\left(-\frac{1}{2}(x^{2}+y^{2})
\frac{1+e^{-2t}}{1-e^{-2t}}
\right),
$$
and $I_{\beta}$ denotes the modified Bessel function of the first kind and order $\beta$.

Integral representation \eqref{1.9} also defines a semigroup of operators $\{W_{t}^{\mathcal L_\beta}\}_{t>0}$ in $L^{p}(0,\infty)$, $1\leq p\leq \infty$. The semigroup $\{W_{t}^{\mathcal L_\beta}\}_{t>0}$ does not satisfies the markovian property. $\{W_{t}^{\mathcal L_\beta}\}_{t>0}$ is contractive in $L^{p}(0,\infty)$, $1\leq p\leq \infty$, if and only if, $\beta\in\{-1/2\}\cup[1/2,+\infty)$ (see \cite{NS}).

Littlewood-Paley functions and area integrals in $\mathcal{L}_{\beta}$-setting were studied by \cite{Wr} and \cite{BMR}, respectively.

We establish the following results connecting uniformly convex and smooth Banach spaces with Littlewood-Paley and area integrals associated to heat semigroups defined for Laguerre operators.

\begin{thm}\label{thm1.11}
Let $B$ be a Banach space, $2\leq q<\infty$, $1<p<\infty$, $\alpha>0$ and $\beta>-1/2$. The following assertions are equivalent.
\begin{itemize}
\item [(a)] There exists a norm $\left\VERT\cdot\right\VERT$ on $B$ that is equivalent to $\|\cdot\|$ and such that $(B,\left\VERT\cdot\right\VERT)$ is $q$-uniformly convex.

\item[(b)] There exists $C>0$ such that
\[
\|g_{q,W_{t}^{\mathcal{L}_{\beta}};B}^{\alpha}(f)\|_{L^{p}(0,\infty)}
\leq C \|f\|_{L^{p}((0,\infty),B)},\quad f\in L^{p}((0,\infty),B).
\]

\item[(c)] There exists $C>0$ such that
\[
\|\mathcal{A}_{q,W_{t}^{\mathcal{L}_{\beta}};B}^{\alpha}(f)\|_{L^{p}(0,\infty)}
\leq C \|f\|_{L^{p}((0,\infty),B)},\quad f\in L^{p}((0,\infty),B).
\]
\end{itemize}
\end{thm}

\begin{thm}\label{thm1.12}
Let $B$ be a Banach space, $1<q\leq 2$, $1<p<\infty$, $\alpha>0$ and $\beta>-1/2$. The following assertions are equivalent.
\begin{itemize}
\item [(a)] There exists a norm $\left\VERT\cdot\right\VERT$ on $B$ that is equivalent to $\|\cdot\|$ and such that $(B,\left\VERT\cdot\right\VERT)$ is $q$-uniformly smooth.

\item[(b)] There exists $C>0$ such that
\[
\|f\|_{L^{p}((0,\infty),B)}\leq C
\|g_{q,W_{t}^{\mathcal{L}_{\beta}};B}^{\alpha}(f)\|_{L^{p}(0,\infty)},
\quad f\in L^{p}((0,\infty),B).
\]

\item[(c)] There exists $C>0$ such that
\[
\|f\|_{L^{p}((0,\infty),B)} \leq C
\|\mathcal{A}_{q,W_{t}^{\mathcal{L}_{\beta}};B}^{\alpha}(f)\|_{L^{p}(0,\infty)}, \quad f\in L^{p}((0,\infty),B).
\]
\end{itemize}
\end{thm}

We remark that from our results by taking $q=2$ we can deduce new characterizations of Hilbert spaces.

In the remaining of this paper we present proofs of Theorems \ref{thm1.5}-\ref{thm1.12}. In our proofs Calder\'on-Zygmund theory for Banach-valued singular integrals plays a key role (see \cite{RRT}).
Throughout this paper by $C$ and $c$ we always denote positive constants that can change in each occurrence.

\section{Results concerning symmetric diffusion semigroups }
Let $(\Omega,\mu)$ be a $\sigma$-finite measure space and let $B$ be a Banach space. For every $1\leq p\leq\infty$, we denote by $\LpOB$ the $p$-th B\"{o}chner-Lebesgue space. It is well known that if  $1\leq p\leq\infty$ and $T$ is a positive bounded operator from $\LpO$ into itself, then $T\otimes I_{B}$ is bounded from $\LpOB$ into itself and
$$
\|T\otimes I_{B}\|_{\LpOB\to\LpOB}=\|T\|_{\LpO\to\LpO}.
$$

Suppose that $\{T_{t}\}_{t>0}$ is a semigroup of positive operators in $\LpO$, $1 \leq p\leq \infty$. For every $t>0$, we also denote by $T_{t}$ to the operator $T_{t}\otimes I_{B}$ defined in $L^p(\Omega)\otimes B$, $1\leq p\leq\infty$. Thus, $\{T_{t}\}_{t>0}$ is a semigroup of operators in $\LpOB$, $1\leq p\leq\infty$.

Let $\alpha>0$. We choose $m\in\N$ such that $m-1\leq\alpha < m$. If $f\in C^{m}(0,\infty)$ the $\alpha$-derivative $\partial_t^{\alpha}f(t)$ of $f$ in $t\in (0,\infty)$ is given by
$$
\partial_{t}^{\alpha}f(t)=\frac{1}{\Gamma(m-\alpha)}\int_{t}^{\infty}(\partial_u^{m}f)(u)(u-t)^{m-\alpha-1}du,
$$
provided that the last integral exists. The fractional derivative is usually named Weyl derivative.

Suppose that $\{T_{t}\}_{t>0}$ is a symmetric diffusion semigroup in $(\Omega,\mathcal{A},\mu)$. According to \cite[Theorem 1, p. 67]{SteinLp} $\{T_{t}\}_{t>0}$ defines a bounded analytic semigroup in  $\LpO$, $1<p<\infty$. Then for every $k\in\N$, the set $\{t^{k}\partial_{t}^{k}T_{t}\}_{t>0}$ is bounded in the space $\mathcal{L}(\LpO)$ of bounded operators from $\LpO$ into itself, for every $1<p<\infty$. It follows that, for every $1<p<\infty$, $\int_{t}^{\infty}\|\partial_{u}^{m}T_{u}\|_{\LpO\to\LpO}(u-t)^{m-\alpha-1} du<\infty$, $t>0$.

We define, for every $1<p<\infty$ and $f\in\LpO$,

\[\partial_{t}^{\alpha}T_{t}f=\frac{1}{\Gamma(m-\alpha)}
\int_{t}^{\infty}\partial_{u}^{m}T_{u}(f)(u-t)^{m-\alpha-1} du, \quad t>0.\]

Assume that the Banach space $B$ is of martingale cotype $q$ where $2\leq q<\infty$. According to \cite[p. 5]{Xu18} (also see \cite[Theorem 1.2]{Pi5}), for every $1\leq p<\infty$, $\{T_{t}\}_{t>0}$ is an analytic semigroup on $\LpOB$. We also define, for every $1< p<\infty$ and $f\in\LpOB$,

\[\partial_{t}^{\alpha}T_{t}f=\frac{1}{\Gamma(m-\alpha)}
\int_{t}^{\infty}\partial_{u}^{m}T_{u}(f)(u-t)^{m-\alpha-1} du, \quad t>0.\]

Thus, $\partial_{t}^{\alpha}T_{t}f\in\LpOB$ provided that $f\in\LpOB$, $1\leq p<\infty$ and $t>0$. Note also that if $f\in\LpO\otimes B$, $1\leq p<\infty$ and $t>0$, then $\partial_{t}^{\alpha}T_{t}f\in\LpO\otimes B$

Let $1\leq p<\infty$ and $1<r<\infty$. We define the Littlewood-Paley function $g^\alpha_{r,T_{t};B}(f)$ of $f\in\LpOB$ associated with the semigroup $\{T_{t}\}_{t>0}$ as follows

$$
g^\alpha_{r, T_{t};B}(f)(x)=
(
\int_{0}^{\infty}\|t^{\alpha}\partial_{t}^{\alpha}T_{t}(f)(x)\|^{r}
\frac{dt}{t})^{1/r}, \quad x\in\Omega.
$$
Here, $\|\cdot\|$ denotes the norm in $B$.
%Xu (\cite[Theorem 2]{Xu18}) established the following result.
%
%\begin{ThA}\label{thm1}
%Let $B$ be a Banach space and let $k$ be a positive integer. If $B$ is of martingale cotype $q$ with $2\leq q<\infty$ y $\{T_{t}\}_{t>0}$ is a symmetric diffusion semigroup, then
%$$
%\|g_{q, T_{t};B}^{k}(f)\|_{\LpO}\leq C\|f\|_{\LpOB},
%$$
%for every $f\in\LpOB$ with $1< p<\infty$. Here $C$ is a positive constant depending on $p,q,k$ and the martingale cotype $q$ constant of $B$
%\end{ThA}

By proceeding as in \cite[Proposition 3.1]{TZ} we can see that if $0<\alpha<\beta$ then
\begin{equation}\label{A1}
g_{r,T_{t};B}^{\alpha}(f)\leq g_{r,T_{t};B}^{\beta}(f), \quad f\in\LpOB,
\end{equation}
for every $1<r<\infty$ and $1 \leq p<\infty$.

From \eqref{A1} and Theorem \ref{thm1.4}, (a), we deduce immediately the following property.

\begin{cor}\label{cor1}
Let $B$ be a Banach space and $\alpha>0$. If $B$ is of martingale cotype $q$ with $2\leq q<\infty$ and $\{T_{t}\}_{t>0}$ is a symmetric diffusion semigroup, then
$$
\|g_{q, T_{t};B}^{\alpha}(f)\|_{\LpO}\leq C\|f\|_{\LpOB},
$$
for every $f\in\LpOB$ with $1<p<\infty$. Here $C$ is a positive constant depending on $p,q,\alpha$ and the martingale cotype $q$ of $B$.
\end{cor}

In order to  prove a converse of Corollary \ref{cor1} we consider the classical heat semigroup $\{W_{t}\}_{t>0}$ defined by
$$
W_{t}(f)(x)=\int_{\Rn}W_{t}(x-y)f(y)dy, \quad f\in\bigcup_{p\in [1,\infty]}\LpRB,
$$
where $W_{t}(z)=\frac{1}{(4\pi)^{\frac{n}{2}}}\dfrac{e^{-|z|^{2}/4t}}{t^{n/2}}$, $z\in\Rn$ and $t>0$.

\begin{thm}\label{thm2}
Let $B$ be a Banach space, $1<p<\infty$, $2\leq q<\infty$ and $\alpha>0$. Suppose that $g_{q,W_{t};B}^{\alpha}$ is bounded from $\LpRB$ into $\LpR$. Then, there exists a norm $\left\VERT\cdot\right\VERT$ on $B$ that is equivalent to $\|\cdot\|$ and such that $(B,\left\VERT\cdot\right\VERT)$ is $q$-uniformly convex.
\end{thm}

\begin{proof}
In this proof we combine some ideas developed in \cite[Sections 2 and 3]{TZ}. By proceeding as in \cite[Theorem 2.3]{TZ} we can see that, if $\gamma, \beta>0$,

\begin{equation}\label{A0}
\partial_{t}^{\gamma}(\partial_{t}^{\beta}W_{t}(f))=
\partial_{t}^{\gamma+\beta}W_{t}(f), \qquad f\in L^{r}(\Rn,B)\quad \text{and} \quad t>0,
\end{equation}
for every $1\leq r<\infty$.

Let $N\in\N$. We denote $H_{N}=L^{q}((\frac{1}{N}, \infty),\frac{dt}{t})$. Assume that $f\in \CcR$ that represents the space of smooth functions with compact support in $\Rn$.

Let $\beta>0$. We consider the function $F$ defined by
\[
[F(x)](t)=t^{\beta}\partial_{t}^{\beta}W_{t}(f)(x), \qquad x\in\Rn\quad\text{and}\quad t>0.
\]

We are going to see that $F(x)\in H_{N}$, for every $x\in\Rn$, and that the function $F$ is $H_{N}$-strongly measurable in $\Rn$. Actually we will prove that $F$ is a continuous function from $\Rn$ into $H_{N}$. Note firstly that $W_{t}(f)(x)\in C^{\infty}(\Rn\times(0,\infty))$.

Let $m\in\N$ such that $m-1\leq \beta<m$. We write
\begin{eqnarray*}
|\partial_{t}^{m}W_{t}(f)(x)|&=&
\left|\frac{1}{(4\pi)^{\frac{n}{2}}}\partial_{t}^{m}\int_{\Rn} \frac{e^{-|x-y|^{2}/4t}}{t^{n/2}}f(y)dy\right|\\
&\leq& C\int_{\Rn}\left|\partial_{t}^{m}\left(\frac{e^{-|x-y|^{2}/4t}}{t^{n/2}}
\right)\right|\,|f(y)|dy\\
&\leq& C\int_{\Rn}\frac{e^{-c|x-y|^{2}/t}}{t^{n/2+m}}|f(y)|dy\\
&\leq& Ct^{-n/2-m}, \quad x\in\Rn\quad\text{and}\quad t>0.
\end{eqnarray*}
Then,
\begin{eqnarray}
|\partial_{t}^{\beta}W_{t}(f)(x)|&\leq&
C\int_{t}^{\infty}|\partial_{s}^{m}W_{s}(f)(x)|(s-t)^{m-\beta-1}ds\nonumber\\
&\leq& C\int_{t}^{\infty}\frac{(s-t)^{m-\beta-1}}{s^{n/2+m}}
\leq Ct^{-n/2-\beta}, \quad x\in\Rn\quad\text{and}\quad t>0\label{A01}.
\end{eqnarray}

We obtain $|[F(x)](t)|\leq Ct^{-n/2}$, $x\in\Rn$ and $t>0$. Then,
\[
\|F(x)\|_{H_{N}}^{q}=\int_{1/N}^{\infty}|[F(x)](t)|^{q}\frac{dt}{t}<\infty, \quad x\in\Rn.
\]

Let $x_{0}\in\Rn$. Assume that $(x_{k})_{k=1}^{\infty}$ is a sequence of real numbers and that $x_{k}\to x_{0}$, as $k\to\infty$.
Since $|[F(x_{k})](t)-[F(x_{0})](t)|\leq Ct^{-n/2}$, $k\in\N$ and $t>0$, the dominated convergence theorem leads to

\[
\int_{1/N}^{\infty}|[F(x_{k})](t)-[F(x_{0})](t)|^{q}\frac{dt}{t}\to 0, \quad\text{as}\quad k\to\infty.
\]
Hence $F$ is continuous from $\mathbb R^n$ into $H_N$. It is clear that
$$
g_{q,W_{t};\mathbb C}^{\beta}(f)(x)=
\|F(x)\|_{L^{q}\co}=\lim\limits_{N\to\infty}\|F(x)\|_{H_{N}}, \quad x\in\Rn.
$$

We consider the operator $T^{\beta}$ defined, for every $f\in \CcR$, by $T^{\beta}(f)=F$, where $[F(x)](t)= t^{\beta}\partial_{t}^{\beta}W_{t}(f)(x)$, $x\in\Rn$ and $t>0$. Suppose that $g_{q,W_{t};\mathbb C}^{\beta}$ is bounded from $\LpR$ into itself. It follows that
$$
\|T^{\beta}f\|_{L^{p}(\Rn,L^{q}((0,\frac{1}{N}), \frac{dt}{t}))}
\leq C\|f\|_{\LpR}, \qquad f\in \CcR,
$$
 where $C>0$ does not depend on $N$.

 Then, $T^{\beta}$ can be extended to $\LpR$ as a bounded operator from $\LpR$ into $L^{p}(\Rn,L^{q}\co)$.

 Let $f\in \CcR$ and $x\in\Rn$. We define the function $G$ as follows

\begin{eqnarray*}
G: \Rn &\to& H_{N}\\
y  &\to& [G(y)](t)=t^{\beta}\partial_{t}^{\beta}W_{t}(x-y)f(y).
\end{eqnarray*}

By proceeding as above we can see that $G\in L^{1}(\Rn,H_{N})$. Also, we have that
\begin{equation}\label{A1.0}
\left(\int_{\Rn}G(y) dy\right)(t)=
\int_{\Rn} t^{\beta}\partial_{t}^{\beta}W_{t}(x-y)f(y) dy,
\quad a.e\quad t\in (\frac{1}{N}, \infty),
\end{equation}
where the first integral is understood in the $H_{N}$-B\"{o}chner sense.

Indeed, let $L\in H'_{N}$. There exists $h\in L^{q'}\coN$ such that
$$
L(g)=\int_{\frac{1}{N}}^{\infty}g(t)h(t)\frac{dt}{t}, \quad g\in L^{q}\coN .
$$
By using well known properties of the B\"{o}chner integral we deduce that
$$
L\left(\int_{\Rn}G(y) dy\right)=\int_{\Rn} L(G(y)) dy=
\int_{\Rn}\int_{\frac{1}{N}}^{\infty}
t^{\beta}\partial_{t}^{\beta}W_{t}(x-y)g(t)\frac{dt}{t}f(y) dy.
$$
Since
$$
\int_{\frac{1}{N}}^{\infty}\int_{\Rn}
|t^{\beta}\partial_{t}^{\beta}W_{t}(x-y)| |f(y)| dy |g(t)|\frac{dt}{t}\leq
C\left(\int_{\frac{1}{N}}^{\infty}\dfrac{1}{t^{nq/2+1}}dt\right)^{1/q}
\|g\|_{L^{q'}\left((\frac{1}{N},\infty),\frac{dt}{t}\right)}<\infty,
$$
Fubini theorem leads to
$$
L\left(\int_{\Rn}G(y) dy\right)=
\int_{\frac{1}{N}}^{\infty}\int_{\Rn}
t^{\beta}\partial_{t}^{\beta}W_{t}(x-y)f(y) dy g(t)\frac{dt}{t}.
$$
According to Hahn-Banach theorem we get \eqref{A1.0}.
Above estimates allows us to differentiate under the integral sign to obtain

\begin{equation}\label{A1.1}
\partial_{t}^{\beta}\int_{\Rn}W_{t}(x-y)f(y)dy=
\int_{\Rn}\partial_{t}^{\beta}W_{t}(x-y)f(y)dy, \quad t>0.
\end{equation}

Thus we have proved that, for every $f\in\CcR$,
\[(T^{\beta}f)(x)=\int_{\Rn}K(x-y)f(y)dy,\quad x\in\Rn,\]
where $[K(z)](t)=t^{\beta}\partial_{t}^{\beta}W_{t}(z)$, $z\in\Rn$ and $t>\frac{1}{N}$. The last integral is understood in $L^{q}\coN$-B\"{o}chner sense.

By using partial integration we can write

\begin{align}\label{A1.2}
\|K(z)\|_{H_{N}}&=\left(\int_{\frac{1}{N}}^{\infty}
|t^{\beta}\partial_{t}^{\beta}W_{t}(z)|^{q}\frac{dt}{t}
\right)^{1/q}
\leq
C\left(\int_{0}^{\infty}
\left|t^{\beta}\int_{t}^{\infty}
\partial_{s}^{m}W_{s}(z)(s-t)^{m-\beta-1}ds\right|^{q}\frac{dt}{t}
\right)^{1/q}\nonumber\\
&\leq C\left(
\int_{0}^{\infty}
\left|t^{\beta}\int_{t}^{\infty}
\frac{e^{-c|z|^{2}/s}}{s^{n/2+m}}
(s-t)^{m-\beta-1}ds\right|^{q}\frac{dt}{t}
\right)^{1/q}\nonumber\\
&\leq
C\left(
\left(
\int_{0}^{|z|^{2}}
\left|t^{\beta}\int_{t}^{\infty}
\frac{e^{-c|z|^{2}/s}}{s^{n/2+m}}
(s-t)^{m-\beta-1}ds\right|^{q}\frac{dt}{t}
\right.\right)^{1/q}\nonumber\\
&\hspace{5mm}
+\left(\left.
\int_{|z|^{2}}^{\infty}
\left|t^{\beta}\int_{t}^{\infty}
\frac{e^{-c|z|^{2}/s}}{s^{n/2+m}}
(s-t)^{m-\beta-1}ds\right|^{q}\frac{dt}{t}
\right)^{1/q}
\right)\nonumber\\
&\leq
C\left(
\left(
\int_{0}^{|z|^{2}}
\left|t^{\beta}
\left(
\left.
\frac{e^{-c|z|^{2}/s}}{s^{n/2+m}}\frac{(s-t)^{m-\beta}}{m-\beta}\right]_{t}^{\infty}\right.\right.\right.\right.\nonumber\\
&\hspace{5mm}
\left.\left.\left.-\int_{t}^{\infty}
\frac{\partial}{\partial s}
\left[
\frac{e^{-c|z|^{2}/s}}{s^{n/2+m}}
\right]
\frac{(s-t)^{m-\beta}}{m-\beta}ds\right)\right|^{q}\frac{dt}{t}
\right)^{1/q}\nonumber\\
&\hspace{5mm}\left.
+\left(
\int_{|z|^{2}}^{\infty}
\left|t^{\beta}\int_{t}^{\infty}
\frac{(s-t)^{m-\beta-1}}{s^{n/2+m}}ds\right|^{q}\frac{dt}{t}
\right)^{1/q}
\right)\nonumber\\
&\leq
C\left(
\left(
\int_{0}^{|z|^{2}}
\left|t^{\beta}
\int_{t}^{\infty}
\frac{e^{-c|z|^{2}/s}}{s^{n/2+m+1}}
(s-t)^{m-\beta}ds\right|^{q}\frac{dt}{t}
\right)^{1/q}\right.\nonumber\\
&\hspace{5mm}+
\left.\left(
\int_{|z|^{2}}^{\infty}
\left|t^{-n/2}\int_{1}^{\infty}
\frac{(v-1)^{m-\beta-1}}{v^{n/2+m}}dv\right|^{q}\frac{dt}{t}
\right)^{1/q}
\right)\nonumber\\
&\leq
C\left(\left(
\int_{0}^{|z|^{2}}
\left|t^{\beta}
\int_{t}^{\infty}
\frac{e^{-c|z|^{2}/s}}{s^{n/2+\beta+1}}ds\right|^{q}\frac{dt}{t}
\right)^{1/q}+
\left(
\int_{|z|^{2}}^{\infty}
\frac{1}{t^{qn/2+1}}dt
\right)^{1/q}
\right)\nonumber\\
&\leq
C\left(
\frac{1}{|z|^{n+2\beta}}
\left(\int_{0}^{|z|^{2}}
t^{\beta q-1}dt
\right)^{1/q}
+\frac{1}{|z|^{n}}
\right)\leq \frac{C}{|z|^{n}}, \qquad z\in\Rn\setminus\{0\}.
\end{align}
Note that $C>0$ does not depend on $N$.

We define
$[\mathscr{H}_{i}(z)](t)=t^{\beta}\partial_{t}^{\beta}\partial_{z_{i}}W_{t}(z)$, $z\in\Rn$, $i=1,\ldots,n$ and $t>\frac{1}{N}$. By proceeding as above we can prove that

\[\|\mathscr{H}_{i}(z)\|_{H_{N}}\leq\frac{C}{|z|^{n+1}}, \quad z\in\Rn\setminus\{0\}\quad\text{and}\quad i=1,\ldots,n\]
where $C>0$ does not depend on $N$.

\label{pA6}
By using Calder\'on-Zygmund theorem for vector valued singular integrals we conclude that, for every $r\in(1,\infty)$, $T^{\beta}$ can be extended to $L^{r}(\Rn)$ as a bounded operator from $L^r({\Rn})$ to $L^r({\Rn,H_{N}})$ and by denoting the extended operator also by $T^{\beta}$ we have that
$$
\|T^{\beta}(f)\|_{L^{r}(\Rn,H_{N})}\leq C\|f\|_{L^{r}(\Rn)}, \quad f\in L^{r}(\Rn),
$$
where $C>0$ does not depend on $N$.

Let $1<r<\infty$. We have that
$$
\left\|\left(
\int_{\frac{1}{N}}^{\infty}
|t^{\beta}\partial_{t}^{\beta}W_{t}(f)(x)|^{q}\frac{dt}{t}
\right)^{1/q}
\right\|_{L^{r}(\Rn)}
\leq C \|f\|_{L^{r}(\Rn)},\quad f\in\CcR.
$$
Since $C$ does not depend on $N$ by making $N\to\infty$ we get

\begin{equation}\label{A2}
\left\|\left(
\int_{0}^{\infty}
|t^{\beta}\partial_{t}^{\beta}W_{t}(f)(x)|^{q}\frac{dt}{t}
\right)^{1/q}
\right\|_{L^{r}(\Rn)}
\leq C \|f\|_{L^{r}(\Rn)},\quad f\in\CcR.
\end{equation}

Let $f\in L^{r}({\Rn})$. We choose a sequence $(f_{k})_{k=1}^{\infty}$ in $\CcR$ such that $f_{k}\to f$, as $k\to\infty$, in $L^{r}(\Rn)$. We have that
\begin{eqnarray}\label{A11}
|t^{\beta}\partial_{t}^{\beta}W_{t}(f-f_{k})(x)|
&\leq &\left(
\int_{\Rn} |t^{\beta}\partial_{t}^{\beta}W_{t}(x-y)|^{r'} dy
\right)^{1/r'}
\|f-f_{k}\|_{L^{r}(\Rn)}\nonumber\\
& \leq & C \|f-f_{k}\|_{L^{r}(\Rn)}
\left(
\int_{\Rn}
\left|t^{\beta}
\int_{t}^{\infty}\frac{e^{-c|x-y|^{2}/s}}{s^{n/2+m}} (s-t)^{m-\beta-1} ds
\right|^{r'} dy
\right)^{1/r'}\nonumber\\
& \leq & C \|f-f_{k}\|_{L^{r}(\Rn)}
t^{\beta}
\int_{t}^{\infty}\frac{(s-t)^{m-\beta-1}}{s^{n/2+m}}
\left(
\int_{\Rn} e^{-c|x-y|^{2}/s} dy
\right)^{1/r'}
ds\nonumber\\
& \leq &C \|f-f_{k}\|_{L^{r}(\Rn)}
t^{\beta}
\int_{t}^{\infty}\frac{(s-t)^{m-\beta-1}}{s^{rn/2+m}} ds\nonumber\\
& \leq & C \|f-f_{k}\|_{L^{r}(\Rn)} t^{-rn/2}, \quad x\in\Rn\quad\text{and}\quad t>0.
\end{eqnarray}
then, $t^{\beta}\partial_{t}^{\beta}W_{t}(f-f_{k})(x)\to 0$ as $k\to\infty$, uniformly in $(x,t)\in\Rn \times[\frac{1}{N},\infty)$.

By using Fatou Lemma and \eqref{A2} we obtain
\begin{align*}
\left\|
\left(
\int_{0}^{\infty}
|t^{\beta}\partial_{t}^{\beta}W_{t}(f)(x)|^{q} \frac{dt}{t}
\right)^{1/q}
\right\|_{L^{r}(\Rn)}
&\leq \liminf_{k\to \infty}
\left\|
\left(
\int_{0}^{\infty}
|t^{\beta}\partial_{t}^{\beta}W_{t}(f_{k})(x)|^{q} \frac{dt}{t}
\right)^{1/q}
\right\|_{L^{r}(\Rn)}\\ \vspace{3mm}
&\leq C \lim\limits_{k\to\infty} \|f_{k}\|_{L^{r}(\Rn)}=
C\|f\|_{L^{r}(\Rn)}.
\end{align*}

The arguments we have developed above by replacing $\CcR$ by $\CcR\otimes B$ and, for every $N\in\N$, $H_{N}$ by $H_{N,B}=L^{q}((\frac{1}{N},\infty), \frac{dt}{t}; B)$, allows us to prove that, for every $1<r<\infty$, there exists $C>0$ such that
\begin{equation}\label{A4}
\|g_{q, W_{t};B}^{\alpha}(f)\|_{L^{r}(\Rn)}\leq
C\|f\|_{L^{r}(\Rn, B)},\quad f\in L^{r}(\Rn, B).
\end{equation}
Our next objective is to see that, for every $k\in\N$ and $1<r<\infty$, there exists $C>0$ such that
$$
\|g^{k\alpha}_{q,W_t;B}(f)\|_{L^{r}(\Rn)}\leq
C\|f\|_{L^{r}(\Rn, B)},\quad f\in L^{r}(\Rn, B).
$$

In order to prove it we use an inductive argument. We have just to see that this property is true when $k=1$. Assume that for a certain $k\in\N$, $g^{k\alpha}_{q,W_t;B}$ defines a sublinear operator from $L^{r}(\Rn, B)$ into $L^{r}(\Rn)$ for every (equivalently, for some) $1<r<\infty$. In particular we have that the operator $T$ defined by
$$
[T(f)(x)](t)=t^{k\alpha}\partial_{t}^{k\alpha}W_{t}(f)(x),\quad
x\in \Rn\quad\text{and}\quad t>0,
$$
for every $f\in L^{q}(\Rn,B)$, is bounded from $L^{q}(\Rn,B)$ into $L^{q}(\Rn,L_{B}^{q}((0,\infty),\frac{dt}{t}))$.

We consider the operator
$$
\mathscr{H}: L^{q}(\Rn, L_{B}^{q}((0,\infty),\frac{dt}{t}))
\to L^{q}(\Rn, L_{B}^{q}((0,\infty)^{2},\frac{dt}{t}\frac{ds}{s})),
$$
defined, for every $h\in L^{q}(\Rn, L_{B}^{q}((0,\infty),\frac{dt}{t}))$, by
$$
[\mathscr{H}(h)(x)](s,t)=s^{\alpha}\partial_{s}^{\alpha}W_{s}([h(\cdot)](t))(x), \quad x\in\Rn\quad\text{and}\quad s,t>0.
$$

By using \eqref{A4} with $r=q$ we get
\begin{eqnarray*}
\|\mathscr{H}(h)\|_{L^{q}(\Rn, L_{B}^{q}((0,\infty)^{2},\frac{dt}{t}\frac{ds}{s}))}^{q}& = &
\int_{\Rn}\int_{0}^{\infty}\int_{0}^{\infty}
\|s^{\alpha}\partial_{s}^{\alpha}W_{s}([h(\cdot)](t))(x)\|^{q} \frac{dt}{t}\frac{ds}{s} dx\\
&=&\int_{0}^{\infty}
\|g_{q,W_{t};B}^{\alpha}([h(\cdot)](t))\|_{L^{q}(\Rn)}^{q}\frac{dt}{t} \\
& = & C \int_{0}^{\infty} \int_{\Rn}
\|[h(\cdot)](t)(x)\|^{q} dx\frac{dt}{t}\\
&= & C \|h\|^{q}_{L^{q}(\Rn, L_{B}^{q}((0,\infty),\frac{dt}{t}))}, \quad h\in L^{q}(\Rn, L_{B}^{q}((0,\infty),\frac{dt}{t})).
\end{eqnarray*}
Hence, the operator $\mathcal{L}=\mathscr{H}\circ T$ is bounded from $L^{q}(\Rn, B)$ into $L^{q}(\Rn, L_{B}^{q}((0,\infty)^{2},\frac{dt}{t}\frac{ds}{s}))$.

Let $f\in L^{q}(\Rn, B)$. We can write
\begin{eqnarray}
[\mathcal{L}(f)(x)](s,t) &=& [\mathscr{H}(T(f))(x)](s,t)\nonumber\\ \vspace{2mm}
&=& s^{\alpha}\partial_{s}^{\alpha}W_{s}([T(f)(\cdot)](t))(x)\nonumber\\ \vspace{2mm}
&=& s^{\alpha}\partial_{s}^{\alpha}W_{s}(t^{k\alpha}\partial_{t}^{k\alpha}W_{t}(f)(\cdot))(x)\nonumber\\ \vspace{2mm}
&=& s^{\alpha}t^{k\alpha}\partial_{s}^{\alpha}\partial_{t}^{k\alpha}
W_{s+t}(f)(x)\label{A5}\\ \vspace{2mm}
&=& s^{\alpha}t^{k\alpha}\partial_{u}^{(k+1)\alpha}
W_{u}(f)(x)|_{u=s+t}, \quad x\in\Rn\quad\text{and}\quad t,s>0.\label{A6}
\end{eqnarray}

We now justify \eqref{A5} and \eqref{A6}. We choose $l\in\N$ such that $l-1\leq k\alpha<l$. It follows that
\begin{eqnarray*}
\partial_{t}^{k\alpha}W_{s}(W_{t}(f))(x) &= & \partial_{t}^{k\alpha}W_{t+s}(f)(x)\\
&= &\int_{t}^{\infty}\partial_{u}^{l}W_{u+s}(f)(x)(u-t)^{l-k\alpha-1}du\\
&= &\int_{t}^{\infty}\partial_{u}^{l}W_{u}(W_{s}(f))(x)(u-t)^{l-k\alpha-1} du\\
&= &\int_{t}^{\infty}\Delta^{l}W_{u}(W_{s}(f))(x)(u-t)^{l-k\alpha-1} du\\
&= &\int_{t}^{\infty}W_{s}(\Delta^{l}W_{u}(f))(x)(u-t)^{l-k\alpha-1} du\\
&= &\int_{t}^{\infty}W_{s}(\partial_{u}^{l}W_{u}(f))(x)(u-t)^{l-k\alpha-1} du,\quad x\in\Rn\quad\text{and}\quad s,t>0.\\
\end{eqnarray*}
We can write

\begin{align*}
&\int_{t}^{\infty}\int_{\Rn} \frac{e^{-|x-y|^{2}/4s}}{s^{n/2}}
\int_{\Rn} \frac{e^{-c|y-z|^{2}/u}}{u^{n/2+l}}
\|f(z)\|_{B}\,dz dy\,(u-t)^{l-k\alpha-1} du\\
&\hspace{1cm}\leq \int_{t}^{\infty}\int_{\Rn} \frac{e^{-|x-y|^{2}/4s}}{s^{n/2}}
\left(\int_{\Rn} \frac{e^{-c|y-z|^{2}/u}}{u^{(n/2+l)q'}}dz
\right)^{1/q'} dy \: (u-t)^{l-k\alpha-1} du\:
\|f\|_{L^{q}(\Rn,B)}\\
&\hspace{1cm}\leq C \int_{t}^{\infty}\int_{\Rn} \frac{e^{-|x-y|^{2}/4s}}{s^{n/2}}
\frac{1}{u^{l+n/2q}} dy\:(u-t)^{l-k\alpha-1} du\:\|f\|_{L^{q}(\Rn,B)}\\
&\hspace{1cm}\leq C \int_{t}^{\infty} \frac{1}{u^{l+n/2q}}
 (u-t)^{l-k\alpha-1} du\:\|f\|_{L^{q}(\Rn,B)}\\
&\hspace{1cm}\leq C \|f\|_{L^{q}(\Rn,B)} t^{-k\alpha-n/2q},
\quad x\in\Rn\quad\text{and}\quad s,t>0.\\
\end{align*}
We  obtain
\begin{align*}
\partial_{t}^{k\alpha} W_{s}(W_{t}(f))(x)
&= W_{s}\left[\int_{t}^{\infty}
\partial_{u}^{l}W_{u}(f)(\cdot)(u-t)^{l-k\alpha-1} du
\right](x)\\
&= W_{s}(\partial_{t}^{k\alpha} W_{t}(f))(x),
\quad x\in\Rn\quad\text{and}\quad s,t>0.\\
\end{align*}
Then \eqref{A5} is established.

On the other hand, we have that
\begin{align*}
\partial_{t}^{k\alpha} W_{s+t}(f)(x)
&= \int_{t}^{\infty}
\partial_{u}^{l} W_{s+u}(f)(x)(u-t)^{l-k\alpha-1}du\\
&= \int_{t}^{\infty}
\partial_{v}^{l} W_{v}(f)(x)_{|v=s+u}(u-t)^{l-k\alpha-1}du\\
&= \int^{\infty}_{t+s}
\partial_{v}^{l} W_{v}(f)(x)(v-(t+s))^{l-k\alpha-1} du\\
&= \partial_{v}^{k\alpha} W_{v}(f)(x)_{|v=s+t}
\quad x\in\Rn\quad\text{and}\quad s,t>0,
\end{align*}
and \eqref{A6} can be deduced.

By proceeding as in \cite[(3.7)]{TZ} we can prove that
$$
\|g_{q, W_{t};B}^{(k+1)\alpha}(f)\|_{L^{q}(\Rn)}
\leq C\|f\|_{L^{q}(\Rn,B)}, \quad f\in L^{q}(\Rn,B).
$$

By using Calder\'on-Zygmund theory for vector valued integrals as above we obtain that, for every $1<r<\infty$,
$$
\|g_{q, W_{t};B}^{(k+1)\alpha}(f)\|_{L^{r}(\Rn)}
\leq C\|f\|_{L^{r}(\Rn,B)}, \quad f\in L^{r}(\Rn,B).
$$

We now choose $k\in\N$ such that $k\alpha\geq 1$. By writing the heat semigroup instead of the Poisson semigroup, the argument in \cite[Proposition 3.1]{TZ} allows us to prove that
$$
g_{q,W_{t};B}^{1}(f)\leq g_{q,W_{t};B}^{k\alpha}(f),
\quad f\in L^{r}(\Rn,B),\quad 1<r<\infty.
$$
Hence, $g_{q,W_{t};B}^{1}$ is bounded from $L^{r}(\Rn,B)$ into $L^{r}(\Rn)$,  for every $1<r<\infty$.

By using subordination formula, the Poisson semigroup $\{P_{t}\}_{t>0}$ can be written
\begin{align*}
P_{t}(f) &=\frac{t}{2\sqrt{\pi}}
\int_{0}^{\infty}\frac{e^{-t^{2}/4u}}{u^{3/2}} W_{u}(f) du\\
&= \frac{1}{\sqrt{\pi}}
\int_{0}^{\infty}\frac{e^{-v}}{\sqrt{v}} W_{t^{2}/4v}(f) du,
\quad f\in L^{r}(\Rn,B),\quad 1\leq r\leq\infty.\\
\end{align*}

Let $f\in L^{r}(\Rn)\otimes B$, $1<r<\infty$. We have that
\begin{align*}
&\int_{0}^{\infty}\frac{e^{-v}}{\sqrt{v}}
\int_{\Rn}
\left|\frac{\partial}{\partial t}
\left(\frac{e^{-|x-y|^{2}v/t^{2}}}{t^{n}} (4v)^{n/2}
\right)\right|
\|f(y)\|\: dy\:dv\\
&\hspace{1cm}\leq C
\int_{0}^{\infty} \frac{e^{-v}}{\sqrt{v}}
\int_{\Rn} \frac{e^{-c|x-y|^{2}v/t^{2}}}{t^{n+1}} v^{n/2}
\|f(y)\| dy dv\\
&\hspace{1cm}\leq C
\int_{0}^{\infty} \frac{e^{-v}}{\sqrt{v}}
\left(
\int_{\Rn} \left(\frac{e^{-c|x-y|^{2}v/t^{2}}}{t^{n+1}} v^{n/2}\right)^{r'}dy
\right)^{1/r'} dv \:
\|f\|_{L^{r}(\Rn,B)} \\
&\hspace{1cm}\leq C\int_{0}^{\infty}e^{-v} v^{\frac{n}{2r}-\frac{1}{2}}\: dv \:t^{-(\frac{n}{r}+1)}<\infty,
\quad x\in\Rn\quad\text{and}\quad t>0.
\end{align*}
Then,
\begin{align*}
\partial_{t}P_{t}(f)(x)
&= \frac{1}{\sqrt{\pi}}
\int_{0}^{\infty} \frac{e^{-v}}{\sqrt{v}} \partial_{t}(W_{t^{2}/4v}(f)(x)) \:dv\\
&= \frac{1}{2\sqrt{\pi}}
\int_{0}^{\infty} \frac{e^{-v}}{v^{3/2}} t\: \partial_{u} W_{u}(f)(x)_{|u=t^{2}/4v} \:dv,
\quad x\in\Rn\quad\text{and}\quad t>0.\\
\end{align*}
Minkowski inequality leads to
\begin{align*}
g_{q,P_{t};B}^{1}(f)(x)
&\leq C
\int_{0}^{\infty} \frac{e^{-v}}{v^{3/2}}
\left(\int_{0}^{\infty}
\|t^{2}\partial_{u}W_{u}(f)(x)_{|u=t^{2}/4v}\|^{q} \frac{dt}{t}
\right)^{1/q} dv\\
&\leq C
\int_{0}^{\infty} \frac{e^{-v}}{\sqrt{v}} dv \: g_{q,W_{t};B}^{1}(f)(x),
\quad x\in\Rn\quad\text{and}\quad t>0.\\
\end{align*}
Hence, for every $1<r<\infty$ there exists $C>0$ such that
$$
\|g_{q,P_{t};B}^{1}(f)\|_{L^{r}(\Rn)}
\leq C \|f\|_{L^{r}(\Rn,B)}, \quad f\in L^{r}(\Rn)\otimes B.
$$
Since $L^{r}(\Rn)\otimes B$ is dense in $L^{r}(\Rn,B)$, by proceeding as above we obtain that, for every $1<r<\infty$,
$$
\|g_{q,P_{t};B}^{1}(f)\|_{L^{r}(\Rn)}
\leq C \|f\|_{L^{r}(\Rn,B)}, \quad f\in L^{r}(\Rn,B).
$$

According to \cite[Theorem 5.2]{MTX} we conclude that there exists a norm $\left\VERT\cdot\right\VERT$ on $B$ that is equivalent to $\|\cdot\|$ and such that $(B,\left\VERT\cdot\right\VERT)$ is $q$-uniformly convex.
\end{proof}

We define the $(\alpha,q)$-area integrals associated with the Poisson and the heat semigroups as follows:
$$
\mathcal{A}_{q,P_{t};B}^{\alpha}(f)(x)=
\left(\int_{\Gamma(x)}
\|t^{\alpha}\partial_{t}^{\alpha}P_{t}(f)(y)\|^{q}\frac{dy\:dt}{t^{n+1}}
\right)^{1/q}, \quad x\in\Rn
$$
and
$$
\mathcal{A}_{q,W_{t};B}^{\alpha}(f)(x)=
\left(\int_{\Gamma(x)}
\|\left(
s^{\alpha}\partial_{s}^{\alpha}W_{s}(f)(y)
\right)_{|s=t^{2}}\|^{q}\frac{dy\:dt}{t^{n+1}}
\right)^{1/q}, \quad x\in\Rn,
$$
where $\Gamma(x)=\{(y,t)\in\Rn\times(0,\infty): \: |x-y|<t\}$, for every $x\in\Rn$.

The following result was established in \cite[Theorem 5.3]{TZ}.
\begin{thm}\label{thm3}(\cite[Theorem 5.3]{TZ})
Let $B$ be a Banach space and $2\leq q<\infty$. The following statements are
equivalent:
\begin{enumerate}
\item [(a)] There exists a norm $\left\VERT\cdot\right\VERT$ on $B$ that is equivalent to $\|\cdot\|$ and such that $(B,\left\VERT\cdot\right\VERT)$ is $q$-uniformly convex.
\item [(b)] For every (equivalently, for some) $p \in(1,\infty)$ and $\alpha>0$ there exists $C>0$ such that
$$
\|\mathcal{A}_{q,P_{t};B}^{\alpha}(f)\|_{\LpR}
\leq C \|f\|_{\LpRB}, \quad \LpRB.
$$
\end{enumerate}
\end{thm}

The same result can be proved by replacing the Poisson semigroup by the heat semigroup.

\begin{thm}\label{thm4}
Let $B$ be a Banach space and $2\leq q<\infty$. The following statements are
equivalent:
\begin{enumerate}
\item [(a)] There exists a norm $\left\VERT\cdot\right\VERT$ on $B$ that is equivalent to $\|\cdot\|$ and such that $(B,\left\VERT\cdot\right\VERT)$ is $q$-uniformly convex.
\item [(b)] For every (equivalently, for some) $p \in(1,\infty)$ and $\alpha>0$ there exists $C>0$ such that
$$
\|\mathcal{A}_{q,W_{t};B}^{\alpha}(f)\|_{\LpR}
\leq C \|f\|_{\LpRB}, \quad \LpRB.
$$
\end{enumerate}
\end{thm}
\begin{proof}
$(a)\Rightarrow (b)$
Assume that the property $(a)$ holds. Let $\alpha>0$. For every $f\in L^{q}(\Rn,B)$, we can write
\begin{align*}
\|\mathcal{A}_{q,W_{t};B}^{\alpha}(f)\|_{L^{q}(\Rn)}^{q}
&=
\int_{\Rn}\int_{\Gamma(x)}
\|\left(
s^{\alpha}\partial_{s}^{\alpha}W_{s}(f)(y)
\right)_{|s=t^{2}}\|^{q}\frac{dy\:dt}{t^{n+1}} dx\\
&= \int_{\Rn}\int_{0}^{\infty}\int_{|x-y|<t}
\|(s^{\alpha}\partial_{s}^{\alpha}W_{s}(f)(y))_{|s=t^{2}}\|^{q}
\frac{dy\:dt}{t^{n+1}} dx\\
&= \int_{\Rn}\int_{0}^{\infty}
\|(s^{\alpha}\partial_{s}^{\alpha}W_{s}(f)(y))_{|s=t^{2}}\|^{q}
\int_{|x-y|<t} dx
\frac{dt\:dy}{t^{n+1}} \\
&= c_{n}\int_{\Rn}\int_{0}^{\infty}
\|(s^{\alpha}\partial_{s}^{\alpha}W_{s}(f)(y))_{|s=t^{2}}\|^{q}
\frac{dt}{t} dy \\
&=\frac{c_{n}}{2} \int_{\Rn} \left( g_{q,W_{t};B}^{\alpha}(f)(y)\right)^{q} dy,
\end{align*}
where $c_{n}$ denotes the Lebesgue measure of the unit ball in $\Rn$.\\

According to Corollary \ref{cor1} we obtain
$$
\|\mathcal{A}_{q,W_{t};B}^{\alpha}(f)\|_{L^{q}(\Rn)}
\leq C\|f\|_{L^{q}(\Rn,B)},\quad f\in L^{q}(\Rn,B).
$$

Note that by proceeding as in the proof of \cite[Proposition 2.1 (a)]{AHM} we can see that
$$
\|\mathcal{A}_{q,W_{t};B}^{\alpha}(f)\|_{L^{r}(\Rn)}
\leq C\|g_{q,W_{t};B}^{\alpha}(f)\|_{L^{r}(\Rn)},\quad f\in L^{r}(\Rn),
$$
provided that $2\leq r<\infty$.
\smallskip

We can write
$$
\mathcal{A}_{q,W_{t};B}^{\alpha}(f)(x)=
\|(s^{\alpha}\partial_{s}^{\alpha}W_{s}(f)(x+y))_{|s=t^{2}}\|_{H}, \quad x\in\Rn,
$$
where $H=L_{B}^{q}(\Gamma(0), \frac{dy\:dt}{t^{n+1}})$.

Let $N\in\N$. We consider  $H_{N}=L_{B}^{q}(\Gamma_{N}(0), \frac{dy\:dt}{t^{n+1}})$, where
$\Gamma_{N}(0)=\{(y,t)\in\Rn\times (\frac{1}{N},\infty):\quad |x-y|<t\}$.\\

Let $f\in\CcR\otimes B$ and $x\in\Rn$. We define
$$
[F(x)](y,t)=(s^{\alpha}\partial_{s}^{\alpha}W_{t}(f)(x+y))_{|s=t^{2}}, \quad y\in\Rn\quad\text{and}\quad t>0.
$$
According to \eqref{A01} we get
$$
\int_{\Gamma_{N}(0)}
\|[F(x)](y,t)\|_{B}^{q} \frac{dy\:dt}{t^{n+1}}
\leq C
\int_{\frac{1}{N}}^{\infty} \int_{|y|<t} \frac{1}{t^{n(q+1)+1}} dy\:dt\leq CN^{nq}.
$$

Moreover, if $(x_{k})_{k=1}^{\infty}$ is a sequence in $\Rn$ such that $x_{k}\to x$, as $k\to\infty$, the dominated convergence theorem allows us to see that
$$
F(x_{k})\to F(x), \quad\text{as}\quad k\to\infty,\quad\text{in}\quad H_{N}.
$$
Hence, $F$ is a strongly measurable function from $\Rn$ into $H_{N}$.

For every $f\in\CcR\otimes B$ we define $T_{\alpha}f=F$ where $[F(x)](y,t)=(s^{\alpha}\partial_{s}^{\alpha}W_{s}(f)(x+y))_{|s=t^{2}}$, $x,y\in\Rn$ and $t>0$. We have that
$$
\|T_{\alpha}f\|_{L^{q}(\Rn,H_{N})}
\leq C \|f\|_{L^{q}(\Rn,B)},\quad
f\in\CcR\otimes B.
$$

Let $f\in\CcR\otimes B$ and $x\in\Rn$. We define the function $G$ as follows
\begin{eqnarray*}
G: \Rn &\to& H_{N}\\
z  &\to& [G(z)](y,t)=(s^{\alpha}\partial_{s}^{\alpha}W_{s}(x+y-z))_{|s=t^{2}}f(z).
\end{eqnarray*}
We can write (see \eqref{A01})
\begin{eqnarray*}
\int_{\Rn}\|G(z)\|_{H_{N}}&=&
\int_{\Rn}
\left(
\int_{\Gamma_{N}(0)}
\|(s^{\alpha}\partial_{s}^{\alpha}W_{s}(x+y-z))_{|s=t^{2}}f(z)\|_{B}^{q}\: \frac{dy\:dt}{t^{n+1}}
\right)^{1/q} dz\\
&\leq &C \int_{\Rn}\|f(z)\|_{B}
\left(
\int_{\Gamma_{N}(0)} \frac{dy\:dt}{t^{n(q+1)+1}}
\right)^{1/q} dz\\
&\leq& C \int_{\Rn}\|f(z)\|_{B}
\left(
\int_{\frac{1}{N}}^{\infty} \frac{dt}{t^{n(q+1)}}
\right)^{1/q} dz \leq C,
\end{eqnarray*}
and, for every $g\in L^{q'}(\Gamma_{N}(0),\frac{dy\:dt}{t^{n}})$,
\begin{multline*}
\int_{\Rn}\int_{\Gamma_{N}(0)}
(s^{\alpha}\partial_{s}^{\alpha}W_{s}(x+y-z))_{|s=t^{2}}\:
g(y,t)\frac{dy\:dt}{t^{n+1}}f(z)\:dz\\
=\int_{\Gamma_{N}(0)}\int_{\Rn}
(s^{\alpha}\partial_{s}^{\alpha}W_{s}(x+y-z))_{|s=t^{2}}
f(z)\:dz \: g(y,t)\frac{dy\:dt}{t^{n+1}}.
\end{multline*}

Then,
$$
\left(
\int_{\Rn}G(z)\:dz
\right)(y,t)=
\int_{\Rn}(s^{\alpha}\partial_{s}^{\alpha}W_{s}(x+y-z))_{|s=t^{2}}
\:f(z)\:dz, \quad a.e.\:(y,t)\in\Gamma_{N}(0),
$$
where the first integral is understood as $H_{N}$-B\"{o}chner integral.

As in \eqref{A1.1} we have that
$$
\partial_{s}^{\alpha}W_{s}(f)(x+y)=
\int_{\Rn}\partial_{s}^{\alpha}W_{s}(x+y-z)
f(z)\:dz, \quad x,y\in\Rn\quad\text{and}\quad s>0.
$$

We obtain
$$
(T_{\alpha}f)(x)=\int_{\Rn} K(x-z)\:f(z)\:dz, \quad x\in\Rn,
$$
where, for every $z\in\Rn$,
$$
[K(z)](y,t)=(s^{\alpha}\partial_{s}^{\alpha}W_{s}(y+z))_{|s=t^{2}},
\quad y\in\Rn\text{and}\quad t>0,
$$
and the integral is understood in $H_{N}$-B\"{o}chner sense.

Let $m\in\N$ such that $m-1\leq\alpha<m$. By arguing as in \eqref{A1.2} we get

\begin{align*}
&\|K(z)\|_{L^{q}(\Gamma_{N}(0),\frac{dy\:dt}{t^{n+1}})}=
\left(
\int_{\Gamma_{N}(0)}
\left|
\left(
s^{\alpha}\int_{s}^{\infty}\partial_{u}^{m} W_{u}(y+z)(u-s)^{m-\alpha-1} du
\right)_{|s=t^{2}}
\right|^{q} \frac{dy\:dt}{t^{n+1}}
\right)^{1/q}\\
&\hspace{8mm}\leq C
\left(
\int_{0}^{\infty} \int_{|y|<t}
\left|t^{2\alpha}\int_{t^{2}}^{\infty}
\frac{e^{-c|y+z|^{2}/u}}{u^{n/2+m}}(u-t^{2})^{m-\alpha-1} du
\right|^{q} \frac{dy\:dt}{t^{n+1}}
\right)^{1/q}\\
&\hspace{8mm}\leq C
\left[
\left(
\int_{0}^{\infty} \int_{|y|<t,\:|y|\leq|z|/2}\:
\left|t^{2\alpha}
\int_{t^{2}}^{\infty}
\frac{e^{-c|y+z|^{2}/u}}{u^{n/2+m}}(u-t^{2})^{m-\alpha-1} du
\right|^{q} \frac{dy\:dt}{t^{n+1}}
\right)^{1/q}
\right.\\
&\hspace{10mm}+
\left.
\left(
\int_{0}^{\infty} \int_{|y|<t,\:|z|/2\leq|y|}\:
\left|t^{2\alpha}
\int_{t^{2}}^{\infty}
\frac{e^{-c|y+z|^{2}/u}}{u^{n/2+m}}(u-t^{2})^{m-\alpha-1} du
\right|^{q} \frac{dy\:dt}{t^{n+1}}
\right)^{1/q}
\right]\\
&\hspace{8mm}\leq C
\left[
\left(
\int_{0}^{\infty} \int_{|y|<t}
\left|t^{2\alpha}
\int_{t^{2}}^{\infty}
\frac{e^{-c|z|^{2}/u}}{u^{n/2+m}}(u-t^{2})^{m-\alpha-1} du
\right|^{q} \frac{dy\:dt}{t^{n+1}}
\right)^{1/q}
\right.\\
&\hspace{10mm}+
\left.
\left(
\int_{|z|/2}^{\infty} \int_{|y|<t}
\left|t^{2\alpha}
\int_{t^{2}}^{\infty}
\frac{1}{u^{n/2+m}}(u-t^{2})^{m-\alpha-1} du
\right|^{q} \frac{dy\:dt}{t^{n+1}}
\right)^{1/q}
\right]\\
&\hspace{8mm}\leq C
\left[
\left(
\int_{0}^{\infty}
\left|t^{2\alpha}
\int_{t^{2}}^{\infty}
\frac{e^{-c|z|^{2}/u}}{u^{n/2+m}}(u-t^{2})^{m-\alpha-1} du
\right|^{q} \frac{dt}{t}
\right)^{1/q}
\right.\\
&\hspace{10mm}+
\left.
\left(
\int_{|z|/2}^{\infty}
\left|t^{2\alpha}
\int_{t^{2}}^{\infty}
\frac{1}{u^{n/2+m}}(u-t^{2})^{m-\alpha-1} du
\right|^{q} \frac{dt}{t}
\right)^{1/q}
\right]\\
&\hspace{8mm}\leq C
\left[
\left(
\int_{0}^{\infty}
\left|v^{\alpha}
\int_{v}^{\infty}
\frac{e^{-c|z|^{2}/u}}{u^{n/2+m}}(u-v)^{m-\alpha-1} du
\right|^{q} \frac{dv}{v}
\right)^{1/q}
\right.\\
&\hspace{10mm}+
\left.
\left(
\int_{|z|/2}^{\infty}
\frac{1}{t^{nq+1}}
\left(
\int_{1}^{\infty}
\frac{(u-1)^{m-\alpha-1}}{u^{n/2+m}} du
\right)^{q} dt
\right)^{1/q}
\right]\\
&\leq C\frac{1}{|z|^{n}},\quad z\in\Rn\setminus\{0\}.
\end{align*}
Here $C>0$ does not depend on $N$.

We define, for every $i=1,\ldots,n$, and $z\in\Rn$, $[\mathcal{H}_{i}(z)](y,t)=(s^{\alpha}\partial_{s}^{\alpha}\partial_{z_i}W_{s}(y+z))_{|s=t^{2}}$, $y\in\Rn$ and $t>0$. By proceeding as above we obtain, for every $i=1,\ldots,n$,
$$
\|\mathcal{H}_{i}(z)\|_{L^{q}(\Gamma_{N}(0),\frac{dy\:dt}{t^{n+1}})}
\leq C/|z|^{n+1}, \quad z\in\Rn\setminus\{0\},
$$
where $C>0$ does not depend on $N$.

Then, Calder\'on-Zygmund theorem for Banach-valued singular integrals leads to $T_{\alpha}$ can be extended, for every $1<p<\infty$, from $\CcR\otimes B$ to $\LpRB$ as a bounded operator from $\LpRB$ into $L^{p}(\Rn,H_{N})$. Moreover, for every $1<p<\infty$ there exists $C>0$ that is not depending on $N$ such that
$$
\|T_{\alpha}(f)\|_{L^{p}(\Rn,H_{N})}
\leq C\|f\|_{\LpRB}, \quad f\in \CcR\otimes B.
$$

Suppose that $1<p<\infty$. Let $f\in\LpRB$. There exists a sequence $(f_{k})_{k=1}^{\infty}$ in $\CcR\otimes B$ such that $f_{k}\to f$, as $k\to \infty$, in $\LpRB$. As in \eqref{A11}  we can see that
$$
\|s^{\alpha}\partial_{s}^{\alpha}W_{s}(f-f_{k})(x+y)_{|s=t^{2}}\|
\leq C \|f-f_{k}\|_{\LpRB}t^{-np},
\quad x,y\in\Rn\quad\text{and}\quad t>0.
$$
Then $s^{\alpha}\partial_{s}^{\alpha}W_{s}(f-f_{k})(x+y)_{|s=t^{2}}\to 0$, as $k\to \infty$, in $B$, uniformly in $(x,y,s)\in\Rn\times\Rn\times[\frac{1}{N},\infty)$. By using now Fatou lemma we obtain
\begin{align*}
&\hspace{-12mm}\left\|\left(
\int_{\Gamma_{N}(0)}
\|(s^{\alpha}\partial_{s}^{\alpha}W_{s}(f)(x+y))_{|s=t^{2}}\|_{B}^{q}
\frac{dy\:dt}{t^{n+1}}
\right)^{1/q}\right\|_{\LpR} \\
&\hspace{55mm}=\|(s^{\alpha}\partial_{s}^{\alpha}W_{s}(f)(\cdot,y))_{|s=t^{2}}\|_{L^p(\Rn,H_{N})}\\
&\hspace{55mm}\leq \liminf\limits_{k\to\infty}
\|(s^{\alpha}\partial_{s}^{\alpha}W_{s}(f_{k})(\cdot,y))_{|s=t^{2}}\|_{L^p(\Rn,H_{N})}\\
&\hspace{55mm}=\liminf\limits_{k\to\infty} \|T_{\alpha}(f_{k})\|_{L^p(\Rn,H_{N})}\\
&\hspace{55mm}\leq  C \lim\limits_{k\to\infty}\|f_{k}\|_{L^p(\Rn,H_{N})}\\
&\hspace{55mm}= C \|f\|_{L^p(\Rn,H_{N})}.
\end{align*}
By taking $N\to\infty$, we conclude that
$$
\|\mathcal{A}_{q,W_{t};B}^{\alpha}(f)\|_{\LpR}
\leq C \|f\|_{\LpRB}.
$$
Thus, the proof of $(a)\Rightarrow(b)$ is complete.

$(b)\Rightarrow (a)$. Suppose that $1<p<\infty$ and $\alpha>0$ and that
$$
\|\mathcal{A}_{q,W_{t};B}^{\alpha}(f)\|_{\LpR}
\leq C \|f\|_{\LpRB}, \quad f\in\LpRB,
$$
for certain $C>0$.

By using Calder\'on-Zygmund theorem for Banach-valued singular integrals as in the first part of this proof, we deduce that, for every $r\in(1,\infty)$, there exists $C>0$ for which
$$
\|\mathcal{A}_{q,W_{t};B}^{\alpha}(f)\|_{L^{r}(\Rn)}
\leq C \|f\|_{L^{r}(\Rn,B)}, \quad f\in L^{r}(\Rn,B).
$$
Then, for every $f\in L^{q}(\Rn,B)$,
$$
\|g_{q,W_{t};B}^{\alpha}(f)\|_{L^{q}(\Rn)}=
\left(\frac{2}{c_{n}}
\right)^{1/q}
\|\mathcal{A}_{q,W_{t};B}^{\alpha}(f)\|_{L^{q}(\Rn)}\leq
C \|f\|_{L^{q}(\Rn,B)},
$$
where $c_n$ denotes the Lebesgue measure of the unit ball in $\mathbb R^n$

According to Theorem \ref{thm2}, there exists a norm $\left\VERT\cdot\right\VERT$ on $B$ defining the original topology of $B$ and such that $(B,\left\VERT\cdot\right\VERT)$ is $q$-uniformly convex.
\end{proof}

The following result was established by \cite[Theorem 2 (ii)]{Xu18}.
\begin{thm}{\cite[Theorem 2 (ii)]{Xu18}}\label{thm5}
Let $B$ be a Banach space and let $k$ be a positive integer.
Suppose that there exists a norm $\left\VERT \cdot \right\VERT$ on $B$, that is equivalent to $\|\cdot\|$ and such that $(B, \left\VERT\cdot\right\VERT)$ is $q$-uniformly smooth, where $1<q \leq 2$. If $\{T_{t}\}_{t>0}$ is a symmetric diffusion semigroup, then
$$
\|f\|_{\LpOB}\leq C
(\|\mathcal{F}(f)\|_{\LpOB}+\|g_{q,T_{t};B}^{k}(f)\|_{\LpO}),
$$
for every $f\in\LpOB$ with $1<p<\infty$. Here $C$ is a positive constant depending on $p,q,k$ and the martingale type $q$ constant of $B$.
\end{thm}
This result can be extended to $g$-functions involving fractional derivatives.
\begin{thm}\label{thm6}
Let $B$ be a Banach space and let $\alpha$ be a positive real number. Suppose that there exists a norm $\left\VERT\cdot \right\VERT$ on $B$, that is equivalent to $\|\cdot\|$ and such that $(B, \left\VERT\cdot\right\VERT)$ is $q$-uniformly smooth, where $1<q \leq 2$. If $\{T_{t}\}_{t>0}$ is a symmetric diffusion semigroup, then
$$
\|f\|_{\LpOB}\leq C
(\|\mathcal{F}(f)\|_{\LpOB}+
\|g_{q,T_{t};B}^{\alpha}(f)\|_{\LpO}),
$$
for every $f\in(\LpO\cap L^{2}(\Omega))\otimes B$ with $1<p<\infty$. Here $C$ is a positive constant depending on $p,q,\alpha$ and the martingale type $q$ constant of $B$.
\end{thm}
\begin{proof}
According to \eqref{A1} and Theorem \ref{thm5} we have that
$$
\|f\|_{\LpOB}\leq C
(\|\mathcal{F}(f)\|_{\LpOB}+
\|g_{\alpha,W_{t}}^{q}(f)\|_{\LpO}),
$$
for every $f\in\LpOB$ with $1<p<\infty$, provided that $\alpha\geq 1$.

Suppose that $\alpha\in(0,1)$ and $p\in(1,\infty)$. We have that, for every $f\in\LpOB$,
$$
\partial_{t}^{\alpha}T_{t}f=\frac{1}{\Gamma(1-\alpha)}
\int_{t}^{\infty}\partial_{u}T_{u}(f)(u-t)^{-\alpha} du, \quad t>0.
$$

Let $f\in L^{2}(\Omega)$. Spectral calculus leads to (see \cite[p. 4]{Xu18})
$$
\partial_{u}T_{u}(f)=
-\int_{[0,\infty)}\lambda e^{-\lambda u} E_{A}(d\lambda)f=
-\int_{(0,\infty)}\lambda e^{-\lambda u} E_{A}(d\lambda)f, \quad u>0,
$$
where $E_{A}$ denotes the spectral measure for the infinitesimal generator $A$ of the semigroup $\{T_{t}\}_{t>0}$. Let $g\in L^{2}(\Omega)$. We have that
\begin{align*}
\int_{\Omega}&\int_{t}^{\infty}
\partial_{u}T_{u}(f)(u-t)^{-\alpha}du\:g\:d\mu=
\int_{t}^{\infty}\int_{\Omega}
\partial_{u}T_{u}(f)\:g\:d\mu\:(u-t)^{-\alpha}du\\
&=-\int_{t}^{\infty}\int_{\Omega}\int_{(0,\infty)}
\lambda e^{-\lambda u} E_{A}(d\lambda)f\:g\:d\mu\:(u-t)^{-\alpha}du\\
&=-\int_{t}^{\infty}\int_{(0,\infty)}
\lambda e^{-\lambda u} \langle E_{A}(d\lambda)f,g\rangle_{L^{2}(\Omega,\mu)} (u-t)^{-\alpha}du, \quad t>0.
\end{align*}

We consider the complex measure $\nu$ defined by
$$
\nu(B)=\langle E_{A}(B)f,g\rangle,
$$
for every Borel measurable set $B$. We can write
$$
\int_{(0,\infty)}
\lambda e^{-\lambda u} \langle E_{A}(d\lambda)f,g\rangle=
\int_{(0,\infty)}
\lambda e^{-\lambda u} d\nu(\lambda), \quad u>0,
$$
and
$$
\int_{(0,\infty)}
\lambda e^{-\lambda u} d|\nu|(\lambda)\leq \frac{C}{u}|\nu|(0,\infty)\leq \frac{C}{u} \|f\|_{2}\|g\|_{2}, \quad u>0.
$$

Then
$$
\int_{t}^{\infty}(u-t)^{-\alpha}\int_{(0,\infty)}
\lambda e^{-\lambda u} d|\nu|(\lambda)\:du
\leq C\|f\|_{2}\|g\|_{2} \int_{t}^{\infty}\frac{1}{u(u-t)^{\alpha}} du
\leq \frac{C}{t^{\alpha}} \|f\|_{2}\|g\|_{2}, \quad t>0.
$$

We get
\begin{eqnarray*}
\int_{\Omega}\int_{t}^{\infty}
\partial_{u}T_{u}(f)(u-t)^{-\alpha}  du\:g\:d\mu
&=&-\int_{(0,\infty)}\lambda\int_{t}^{\infty}e^{-\lambda u}
(u-t)^{-\alpha} du\:
\langle E_{A}(d\lambda)f,g\rangle_{L^{2}(\Omega)}\\
&=&-\Gamma(1-\alpha)\int_{(0,\infty)} \lambda^{\alpha}e^{-\lambda t} \langle E_{A}(d\lambda)f,g\rangle_{L^{2}(\Omega)}\\
&=& -\Gamma(1-\alpha)\int_{\Omega}\int_{(0,\infty)} \lambda^{\alpha}
e^{-\lambda t} E_{A}(d\lambda)f\:g\:d\mu, \quad t>0.\\
\end{eqnarray*}

From the arbitrariness of $g$ it follows that
$$
\partial_{t}^{\alpha}T_{t}f=-\int_{(0,\infty)}
\lambda^{\alpha}e^{-\lambda t} E_{A}(d\lambda)f, \quad t>0.
$$
For every $t>0$, $\partial_{t}^{\alpha}T_{t}$ is a selfadjoint operator in $L^{2}(\Omega)$. Then, we can write
\begin{align*}
\langle
\partial_{t}^{\alpha}T_{t}f,\partial_{t}^{\alpha}T_{t}g
\rangle_{L^{2}(\Omega)}
&=\langle(\partial_{t}^{\alpha}T_{t})(\partial_{t}^{\alpha}T_{t})f,g\rangle_{L^{2}(\Omega)}\\
&=\langle
\int_{(0,\infty)}\lambda^{2\alpha} e^{-2\lambda t} E_{A}(d\lambda)f,g
\rangle_{L^{2}(\Omega)}\\
&=\int_{(0,\infty)}\lambda^{2\alpha} e^{-2\lambda t}
\langle E_{A}(d\lambda)f,g
\rangle_{L^{2}(\Omega)}, \quad t>0.
\end{align*}
We have that
$$
\int_{0}^{\infty} t^{2\alpha}
\int_{(0,\infty)} \lambda^{2\alpha} e^{-2\lambda t} d|\nu|(\lambda)
\frac{dt}{t}=\frac{\Gamma(2\alpha)}{2^{2\alpha}}|\nu|(0,\infty)<\infty.
$$
We obtain
\begin{align*}
&\int_{0}^{\infty}
\langle
t^{\alpha}\partial_{t}^{\alpha}T_{t}f,t^{\alpha}\partial_{t}^{\alpha}T_{t}g
\rangle_{L^{2}(\Omega)}
\frac{dt}{t}
=\int_{0}^{\infty} t^{2\alpha}
\int_{(0,\infty)} \lambda^{2\alpha} e^{-2\lambda t}
\langle E_{A}(d\lambda)f,g
\rangle_{L^{2}(\Omega)}
\frac{dt}{t}\\
&=\int_{(0,\infty)} \lambda^{2\alpha}
\int_{0}^{\infty} t^{2\alpha-1} e^{-2\lambda t} dt \:
\langle E_{A}(d\lambda)f,g
\rangle_{L^{2}(\Omega)}\\
&=\frac{\Gamma(2\alpha)}{2^{2\alpha}} \int_{(0,\infty)}
\langle E_{A}(d\lambda)f,g
\rangle_{L^{2}(\Omega)}=
\frac{\Gamma(2\alpha)}{2^{2\alpha}}
\langle E_{A}((0,\infty))f,g
\rangle_{L^{2}(0,\infty)}\\
&=\frac{\Gamma(2\alpha)}{2^{2\alpha}}
\langle f-E_{A}(\{0\})f,g
\rangle_{L^{2}(\Omega)}
=\frac{\Gamma(2\alpha)}{2^{2\alpha}}
\langle f-E_{A}(\{0\})f,g-E_{A}(\{0\})g
\rangle_{L^{2}(\Omega)}.
\end{align*}
We conclude that
$$
\int_{0}^{\infty}
\langle
t^{\alpha}\partial_{t}^{\alpha}T_{t}f,t^{\alpha}\partial_{t}^{\alpha}T_{t}g
\rangle_{L^{2}(\Omega)}
\frac{dt}{t}
=\frac{\Gamma(2\alpha)}{2^{2\alpha}}
\langle f-\mathcal{F}(f),g-\mathcal{F}(g)
\rangle_{L^{2}(\Omega)}.
$$

We denote by $\langle\:,\:\rangle_{B,B^{*}}$ the duality bracket between $B$ and the dual $B^{*}$ of $B$.
Let $f\in(\LpO\cap L^{2}(\Omega))\otimes B$ and $g\in(L^{p'}(\Omega)\cap L^{2}(\Omega))\otimes B^{*}$. We have that
$$
\int_{\Omega}
\langle f-\mathcal{F}(f),g-\mathcal{F}(g)
\rangle_{B,B^{*}} d\mu=
\frac{2^{2\alpha}}{\Gamma(2\alpha)}
\int_{0}^{\infty}\int_{\Omega}
\langle t^{\alpha}\partial_{t}^{\alpha}T_{t}f,t^{\alpha}\partial_{t}^{\alpha}T_{t}g
\rangle_{B,B^{*}} \:d\mu\frac{dt}{t}
$$
H\"{o}lder inequality leads to
\begin{align*}
\left|
\int_{\Omega}
\langle f-\mathcal{F}(f),g-\mathcal{F}(g)
\rangle_{B,B^{*}} d\mu
\right|
&\leq C
\int_{0}^{\infty}\int_{\Omega}
\|t^{\alpha}\partial_{t}^{\alpha}T_{t}f\|_{B}
\|t^{\alpha}\partial_{t}^{\alpha}T_{t}g\|_{B^{*}}
\:d\mu \frac{dt}{t}\\
&\leq C
\left\|\left(
\int_{0}^{\infty}
\|t^{\alpha}\partial_{t}^{\alpha}T_{t}f\|_{B}^{q}\frac{dt}{t}
\right)^{1/2}\right\|_{L^{p}(\Omega)}
\left\|\left(
\int_{0}^{\infty}
\|t^{\alpha}\partial_{t}^{\alpha}T_{t}g\|_{B^{*}}^{q'}\frac{dt}{t}
\right)^{1/2}\right\|_{L^{p'}(\Omega)}.
\end{align*}
Here by $\|\;\|_B$ and $\|\;\|_{B^*}$ we denote the norm in $B$ and in $B^*$, respectively.

Since $B$ has martingale type $q$, $B^{*}$ has martingale cotype $q'$ (\cite[lemma 3.2]{OX}).
Corollary \ref{cor1} implies that
$$
\left|\int_{\Omega}\langle f-\mathcal{F}(f),g\rangle_{B,B^{*}} d\mu\right|
=\left|\int_{\Omega}\langle f-\mathcal{F}(f),g-\mathcal{F}(g)\rangle_{B,B^{*}} d\mu\right|
\leq C
\left\|\left(
\int_{0}^{\infty}
\|t^{\alpha}\partial_{t}^{\alpha}T_{t}f\|_{B}^{q} \frac{dt}{t}
\right)^{1/q}\right\|_{\LpO} \|g\|_{L^{p'}(\Omega)}.
$$
The arbitrariness of $g$ allows us to conclude that
$$
\|f-\mathcal{F}(f)\|_{\LpOB}\leq C \|g_{\alpha,T_{t};B}^{q}(f)\|_{\LpO}
$$
\end{proof}

We now prove a converse of Theorem \ref{thm6} by using the classical heat semigroup $\{W_{t}\}_{t>0}$ on $\Rn$. In this case, the projection $\mathcal{F}$ is the null operator.

\begin{thm}\label{thm7}
Let $B$ be a Banach space, $1<q\leq 2$, $1<p<\infty$ and $\alpha>0$. Suppose that there exists $C>0$ such that
$$
\|f\|_{\LpRB}\leq C
\|g_{q,W_{t};B}^{\alpha}(f)\|_{\LpR},
\quad f\in \LpR\otimes B.
$$
Then, there exists a norm $\left\VERT\cdot\right\VERT$ on $B$ that is equivalent to $\|\cdot\|$ and such that $(B,\left\VERT\cdot\right\VERT)$ is $q$-uniformly smooth.
\end{thm}

\begin{proof}
Let $f\in\CcR\otimes B^{*}$. Fix $\eps>0$. We choose $h\in\CcR\otimes \Cc(0,\infty)\otimes B$ such that
$\|h\|_{L^{p}(\Rn, L^{q}_{B}((0,\infty),\frac{dt}{t}))}=1$ and
$$
\|g_{q',W_{t};B}^{\alpha}(f)\|_{L^{p'}(\Rn)}
\leq\left|
\int_{\Rn}\int_{0}^{\infty}
\langle
t^{\alpha}\partial_{t}^{\alpha}W_{t}(f)(x), h(x,t)
\rangle_{B^{*},B} \frac{dt}{t} dx
\right|+\eps.
$$
We have that
$$
\partial_{t}^{\alpha}W_{t}(x-y)=
\frac{1}{\Gamma(m-\alpha)}
\int_{t}^{\infty} \partial_{u}^{m}W_{u}(x-y)
(u-t)^{ m-\alpha-1} du,
\quad t>0 \quad \text{and}\quad x,y\in\Rn,
$$
where $m-1\leq\alpha<m$ and $m\in\N$. Then,
\begin{align}\label{A.2}
|\partial_{t}^{\alpha}W_{t}(x-y)|
&\leq C\:\int_{t}^{\infty}
\frac{e^{-c|x-y|^{2}/u}}{u^{n/2+m}} (u-t)^{m-\alpha-1} du \nonumber\\
&\leq C \frac{1}{t^{n/2+\alpha}},
\quad t>0\quad\text{and}\quad x,y\in\Rn,
\end{align}
and
\begin{align*}
&\int_{\Rn}\int_{0}^{\infty}
|\langle
t^{\alpha}\partial_{t}^{\alpha}W_{t}(f)(x), h(x,t)
\rangle_{B^{*},B}|
\frac{dt}{t} dx\\
&\hspace{5mm}\leq C \int_{\Rn}\int_{0}^{\infty}\int_{\Rn}
|t^{\alpha}\partial_{t}^{\alpha}W_{t}(x-y)| \:
\|f(y)\|_{B^{*}} dy \:\|h(x,t)\|_{B}
\frac{dt}{t} dx\\
&\hspace{5mm}\leq C \int_{\Rn}\int_{0}^{\infty}\int_{\Rn}
\frac{1}{t^{n/2}}
\|f(y)\|_{B^{*}} dy \: \|h(x,t)\|_{B}
\frac{dt}{t} dx\\
&\hspace{5mm}\leq C \|f(y)\|_{L^{\infty}(\Rn,B^{*})}
\int_{\Rn}\int_{0}^{\infty} \|h(x,t)\|_{B}
\frac{dt}{t} dx<\infty.\\
\end{align*}
We can write
\begin{align*}
&\int_{\Rn}\int_{0}^{\infty}
\langle
t^{\alpha}\partial_{t}^{\alpha}W_{t}(f)(x), h(x,t)
\rangle_{B^{*},B}
\frac{dt}{t} dx\\
&\hspace{5mm}=\int_{\Rn}\int_{0}^{\infty}
\langle
f(y),\int_{\Rn} t^{\alpha}\partial_{t}^{\alpha}W_{t}(x-y)\:h(x,t)\:dx
\rangle_{B^{*},B}
\frac{dt}{t} dy\\
&\hspace{5mm}=\int_{\Rn}
\langle
f(y),\int_{0}^{\infty}\int_{\Rn}
t^{\alpha}\partial_{t}^{\alpha}W_{t}(x-y)\:h(x,t)\:dx\:\frac{dt}{t}
\rangle_{B^{*},B} \:
 dy.\\
\end{align*}

We consider the operator $S_{\alpha}$ defined by
$$
S_{\alpha}(H)(y)=  \int_{0}^{\infty}\int_{\Rn}
t^{\alpha}\partial_{t}^{\alpha}W_{t}(x,y)\:H(x,t)\:dx\:\frac{dt}{t},
\quad H\in\CcR\otimes \Cc(0,\infty)\otimes B.
$$

Let $H\in\CcR\otimes \Cc(0,\infty)\otimes B$. By using our hypothesis we deduce that
\begin{eqnarray}\label{7.1}
\left|
\int_{\Rn}\int_{0}^{\infty}
\langle
t^{\alpha}\partial_{t}^{\alpha}W_{t}(f)(x),\:H(x,t)
\rangle_{B^{*},B}\:\frac{dt}{t} dx
\right|
&\leq&
\|f\|_{L^{p'}(\Rn,B^{*})}
\|S_{\alpha}H\|_{L^{p}(\Rn,B)}\nonumber\\
&\leq & C\|f\|_{L^{p'}(\Rn,B^{*})}
\|g_{q,W_{t};B}^{\alpha}(S_{\alpha}H)\|_{L^{p}(\Rn,B)}\nonumber\\
&=& C\|f\|_{L^{p'}(\Rn,B^{*})}
\|s^{\alpha}\partial_{s}^{\alpha}W_{s}[S_{\alpha}H](x)\|_{L^{p}(\Rn,L^{q}((0,\infty),\frac{dt}{t};B))}.
\end{eqnarray}
By proceeding as above we can see that $S_{\alpha}(H)\in L^{\infty}(\Rn,B)$. We get
$$
\partial_{s}^{\alpha}W_{s}[S_{\alpha}H](x)=
\int_{\Rn}\partial_{s}^{\alpha}W_{s}(x-y)(S_{\alpha}H)(y) \:dy,
\quad x\in\Rn\quad\text{and}\quad s\in(0,\infty).
$$
Also, since
$$
\int_{\Rn}\int_{0}^{\infty}\int_{\Rn}
|\partial_{s}^{\alpha}W_{s}(x-y)|\:t^{\alpha}\:
|\partial_{t}^{\alpha}W_{t}(y-z)|\:\|H(z,t)\|\:dz\frac{dt}{t}dy<\infty,
\quad x\in\Rn\quad\text{and}\quad s\in(0,\infty),
$$
we can write
$$
s^{\alpha}\partial_{s}^{\alpha}W_{s}[S_{\alpha}H](x)=
\int_{\Rn}\int_{0}^{\infty}
K_{\alpha}(x,s;y,t) H(y,t)\frac{dt}{t} dy,
\quad x\in\Rn\quad\text{and}\quad s\in(0,\infty),
$$
where
$$
K_{\alpha}(x,s;y,t)=
\int_{\Rn}
s^{\alpha}\partial_{s}^{\alpha}W_{s}(x-z)\:
t^{\alpha}\partial_{t}^{\alpha}W_{t}(z-y)\:dz,
\quad x,y\in\Rn\quad\text{and}\quad s,t\in(0,\infty).
$$
By using the semigroup property we obtain
\begin{align*}
K_{\alpha}(x,s;y,t)&=
s^{\alpha}\partial_{s}^{\alpha}t^{\alpha}\partial_{t}^{\alpha}
\int_{\Rn} W_{s}(x-z)W_{t}(z-y)dz\\
&=(st)^{\alpha}\partial_{s}^{\alpha}\partial_{t}^{\alpha}W_{s+t}(x-y),
\quad x,y\in\Rn\quad\text{and}\quad s,t\in(0,\infty).
\end{align*}

We consider the operator $T_{\alpha}$ defined as follows
$$
T_{\alpha}(F)(x,s)=
\int_{\Rn}\int_{0}^{\infty} K_{\alpha}(x,s;y,t)F(y,t)
\frac{dt}{t} dy,
\quad F\in L^{p,q}(\Rn\times(0,\infty), dx\frac{dt}{t},B).
$$
Note that the mixed B\"{o}chner-Lebesgue space $L^{p,q}(\Rn\times(0,\infty), dx\frac{dt}{t},B)$ can be identified with the space $L^{p}(\Rn,L^{q}((0,\infty),\frac{dt}{t},B))$.

We are going to see that $T_{\alpha}$ defines a bounded operator from
$L^{q}(\Rn,L^{q}((0,\infty),\frac{dt}{t},B))$ into itself.

Let $F\in L^{q}(\Rn,L_{B}^{q}((0,\infty),\frac{dt}{t}))$. H\"older's inequality leads to
$$
\|T_{\alpha}(F)(x,s)\|_{B}
\leq
\left(
\int_{\Rn}\int_{0}^{\infty}|K_{\alpha}(x,s;y,t)|\frac{dt}{t}dy
\right)^{1/q'}
\left(
\int_{\Rn}\int_{0}^{\infty}|K_{\alpha}(x,s;y,t)| \|f(y,t)\|_{B}^{q}
\frac{dt}{t}dy\right)^{1/q},
$$
for every $x\in\Rn$ and $s\in(0,\infty)$.

By choosing $m\in\N$ such that $m-1\leq\alpha<m$, we have that
\begin{align}
\partial_{s}^{\alpha}\partial_{t}^{\alpha} W_{s+t}(x,y)
&=\frac{1}{\Gamma(m-\alpha)}
\int_{s}^{\infty}\partial_{u}^{m}(\partial_{t}^{\alpha}W_{u+t}(x-y))
(u-s)^{m-\alpha-1} du\nonumber\\
&=\frac{1}{(\Gamma(m-\alpha))^{2}}
\int_{s}^{\infty} (u-s)^{m-\alpha-1} \partial_{u}^{m}
\int_{t}^{\infty}\partial_{v}^{m} W_{u+v}(x-y) (v-t)^{m-\alpha-1}
dv\: du\nonumber\\
&=\frac{1}{(\Gamma(m-\alpha))^{2}}
\int_{s}^{\infty} (u-s)^{m-\alpha-1}
\int_{t}^{\infty} \partial_{z}^{2m} W_{z}(x-y)_{|z=u+v}
(v-t)^{m-\alpha-1}
dv\: du,\label{7.2}
\end{align}
for every $x,y\in\Rn$ and $s,t\in(0,\infty)$.Then,
\begin{align*}
&\int_{\Rn}\int_{0}^{\infty}
|\partial_{s}^{\alpha} \partial_{t}^{\alpha}W_{s+t}(x-y)|
(st)^{\alpha} \frac{dt}{t} dy\\
&\hspace{8mm}\leq C
\int_{\Rn}\int_{0}^{\infty}
(st)^{\alpha}\int_{s}^{\infty} (u-s)^{m-\alpha-1}
\int_{t}^{\infty} |\partial_{z}^{2m} W_{z}(x-y)_{|z=u+v}|
(v-t)^{m-\alpha-1} dv\:du\:\frac{dt}{t}\:dy\\
&\hspace{8mm}\leq C
\int_{\Rn}\int_{0}^{\infty}
(st)^{\alpha}\int_{s}^{\infty} (u-s)^{m-\alpha-1}
\int_{t}^{\infty}
\frac{e^{-c\frac{|x-y|^{2}}{u+v}}}{(u+v)^{n/2+2m}}
(v-t)^{m-\alpha-1} dv\:du\:\frac{dt}{t}\:dy\\
&\hspace{8mm} \leq C
\int_{0}^{\infty} (st)^{\alpha}
\int_{s}^{\infty}(u-s)^{m-\alpha-1}
\int_{t}^{\infty} \frac{1}{(u+v)^{2m}}
(v-t)^{m-\alpha-1} dv\:du\:\frac{dt}{t}\\
&\hspace{8mm}\leq C
\int_{0}^{\infty} (st)^{\alpha}
\int_{s}^{\infty}(u-s)^{m-\alpha-1}
\int_{0}^{\infty} \frac{1}{(u+t+w)^{2m}} \:w^{m-\alpha-1}
dw\:du\:\frac{dt}{t}\\
&\hspace{8mm} \leq C
\int_{0}^{\infty} (st)^{\alpha}
\int_{s}^{\infty}(u-s)^{m-\alpha-1}
(u+t)^{-m-\alpha}\:du\:\frac{dt}{t}\\
&\hspace{8mm} \leq C
\int_{0}^{\infty} (st)^{\alpha}
\int_{0}^{\infty} w^{m-\alpha-1} (w+s+t)^{-m-\alpha}
dw\:\frac{dt}{t}\\
& \hspace{8mm}\leq C
\int_{0}^{\infty} \frac{(st)^{\alpha}}{(t+s)^{2\alpha}\:t} \:dt\\
& \hspace{8mm}\leq C
\int_{0}^{\infty} \frac{t^{\alpha-1}}{(1+t)^{2\alpha}} \:dt<\infty,
\quad x\in\Rn\quad\text{and}\quad s\in (0,\infty).
\end{align*}
We obtain
\begin{align*}
\|T_{\alpha}(F)\|^{q}_{L^{q}(\Rn,L^{q}((0,\infty),\frac{dt}{t};B))}
&\leq C
\int_{\Rn}\int_{\Rn}\int_{0}^{\infty}\int_{0}^{\infty}
|K_{\alpha}(x,s;y,t)|\:\|f(y,t)\|_{B}^{q}
\:\frac{dt}{t}\:dy\:\frac{ds}{s}\:dx\\
&\leq C
\|f\|^{q}_{L^{q}(\Rn,L^{q}((0,\infty),\frac{dt}{t};B))}.
\end{align*}

Let $x,y\in\Rn$, $x\neq y$. We define the operator $\mathbb{K}_{\alpha}(x,y)$ by
\begin{align*}
\mathbb{K}_{\alpha}(x,y):L^{q}((0,\infty),&\frac{dt}{t};B)\rightarrow
L^{q}((0,\infty),\frac{ds}{s};B)\\
& G \rightarrow [\mathbb{K}_{\alpha}(x,y)(G)](s)=
\int_{0}^{\infty} K_{\alpha}(x,s;y,t) G(t) \frac{dt}{t}.
\end{align*}
$\mathbb{K}_{\alpha}(x,y)$ is bounded from $L^{q}((0,\infty),\frac{dt}{t};B)$ into itself. Indeed, let $G\in L^{q}((0,\infty),\frac{ds}{s};B)$. We have that
$$
\|[\mathbb{K}_{\alpha}(x,y)(G)](s)\|_{B}
\leq
\left(
\int_{0}^{\infty}
|K_{\alpha}(x,s;y,t)| \:\frac{dt}{t}
\right)^{1/q'}
\left(\int_{0}^{\infty}
|K_{\alpha}(x,s;y,t)| \:\|G(t)\|_{B}^{q}\:\frac{dt}{t}\right)^{1/q}, \quad s\in (0,\infty).
$$

As above we get
\begin{eqnarray*}
\int_{0}^{\infty}
|K_{\alpha}(x,s;y,t)| \:\frac{dt}{t}
&=& \int_{0}^{\infty}
|\partial_{s}^{\alpha}\partial_{t}^{\alpha} W_{s+t}(x-y)|
(st)^{\alpha} \frac{dt}{t}\\
& \leq & C\int_{0}^{\infty} (st)^{\alpha}
\int_{s}^{\infty} (u-s)^{m-\alpha-1}
\int_{t}^{\infty} \frac{e^{-c\frac{|x-y|^{2}}{u+v}}}{(u+v)^{n/2+2m}}
(v-t)^{m-\alpha-1} \:dv\:du\:\frac{dt}{t}\\
& \leq &\frac{C}{|x-y|^{n}}
\int_{0}^{\infty} (st)^{\alpha}
\int_{s}^{\infty} (u-s)^{m-\alpha-1}
\int_{t}^{\infty} \frac{(v-t)^{m-\alpha-1}}{(u+v)^{2m}}
\:dv\:du\:\frac{dt}{t}\\
& \leq & \frac{C}{|x-y|^{n}}, \quad s\in (0,\infty).
\end{eqnarray*}
Then,
\begin{align}\label{7.1'}
\|\mathbb{K}_{\alpha}(x,y)(G)\|_{L^{q}((0,\infty),\frac{ds}{s};B)}
&\leq C \frac{1}{|x-y|^{n/q'}}
\left(
\int_{0}^{\infty}\int_{0}^{\infty}
|K_{\alpha}(x,s:y,t)|\:\|G(t)\|_{B}^{q} \frac{dt}{t}\:\frac{ds}{s}
\right)^{1/q}\\
&\leq C \frac{1}{|x-y|^{n}} \: \|G\|_{L^{q}((0,\infty),\frac{dt}{t};B)}.\nonumber
\end{align}

Let $G\in \Cc(\Rn,L^{q}((0,\infty),\frac{ds}{s})\otimes B)$. Suppose that $x\notin\supp (G)$ and $H\in \Cc(0,\infty)\otimes B^{*}$. By taking into account \eqref{7.1'} we obtain
\begin{align*}
\int_{0}^{\infty}&
\left\langle
\left(
\int_{\Rn} \mathbb{K}_{\alpha}(x,y) (G(y))
dy \right)(s), H(s)
\right\rangle_{B,B^{*}}
\frac{ds}{s}
= \int_{\Rn} \int_{0}^{\infty}
\langle [\mathbb{K}_{\alpha}(x,y)(G(y))](s), H(s)
\rangle_{B,B^{*}}
\frac{ds}{s}\: dy\\
& = \int_{0}^{\infty}
\left\langle
\int_{\Rn} \int_{0}^{\infty}
K_{\alpha}(x,s;y,t) [G(y)](t)\:\frac{dt}{t}\:dy, H(s)
\right\rangle_{B,B^{*}}\: \frac{ds}{s}.
\end{align*}
Then,
$$
\left(
\int_{\Rn}
\mathbb{K}_{\alpha}(x,y) \:[G(y)]\:dy
\right)(s)
= \int_{\Rn} \int_{0}^{\infty}
K_{\alpha}(x,s;y,t) \:[G(y)](t)\:\frac{dt}{t}\:dy,
\quad \text{a.e}\quad s\in(0,\infty).
$$

We denote by  $\mathcal{L}(L^{q}((0,\infty),\frac{ds}{s};B))$ the space of bounded and linear operators from $L^{q}((0,\infty),\frac{ds}{s};B)$ into itself. By \eqref{7.1'} we have that
$$
\|\mathbb{K}_{\alpha}(x,y)\|_{\mathcal{L}(L^{q}((0,\infty),\frac{ds}{s};B))}
\leq \frac{C}{|x-y|^{n}}, \quad x,y\in\Rn,\: x\neq y.
$$

Let $y,x,x_{0}\in\Rn$ and $G\in L^{q}((0,\infty),\frac{dt}{t};B)$. We have that
\begin{multline*}
[(K_{\alpha}(x,y)-K_{\alpha}(x,x_0))(G)](s)=
\int_{0}^{\infty}
(K_{\alpha}(x,s;y,t)-K_{\alpha}(x,s;x_0,t))\:G(t)\:\frac{dt}{t}\\
=\int_{0}^{\infty}\int_{0}^{1}
\nabla_{z} K_{\alpha}(x,s;\eta y+(1-\eta)x_0,t)(y-x_0)\:
d\eta\:G(t)\:\frac{dt}{t},
\quad s\in (0,\infty),
\end{multline*}
where $\nabla_{z}$ is understood with respect to the third variable, and
\begin{align*}
&\|[(K_{\alpha}(x,y)-K_{\alpha}(x,x_0))(G)](s)\|_{B}
\leq
\int_{0}^{1}\left(\int_{0}^{\infty}
|\nabla_{z} K_{\alpha}(x,s;\eta y+(1-\eta)x_0,t)|\:\frac{dt}{t}
\right)^{1/q'}\\
&\hspace{4cm}\cdot\left(
\int_{0}^{\infty}
|\nabla_{z} K_{\alpha}(x,s;\eta y+(1-\eta)x_0,t)|\: \|G(t)\|_{B}^{q}\frac{dt}{t}
\right)^{1/q}\:d\eta\:|y-x_0|.
\end{align*}

By proceeding as above we get
\begin{align*}
&\int_{0}^{\infty}
|\partial_{z_{i}}K_{\alpha}(x,s;z,t)|\:\frac{dt}{t}=
\int_{0}^{\infty}
|\partial_{s}^{\alpha}\partial_{t}^{\alpha}\partial_{z_{i}}W_{s+t}(x-z)|(st)^{\alpha} \frac{dt}{t}\\
&\leq C
\int_{0}^{\infty}(st)^{\alpha}
\int_{s}^{\infty} (u-s)^{m-\alpha-1}
\int_{t}^{\infty} \frac{e^{-c\frac{|x-y|^{2}}{u+v}}}{(u+v)^{n/2+2m+1/2}} (v-t)^{m-\alpha-1}\:dv\:du\:\frac{dt}{t}\\
& \leq
\frac{C}{|x-z|^{n+1}},\quad z\in \mathbb R^n,
\quad s\in (0,\infty)\quad\text{and}\quad i=1,\ldots,n.
\end{align*}

Then, if $|x-x_0|>2|x_0-y|$, by using Minkowski integral inequality we obtain

\begin{align*}
&\|(K_{\alpha}(x,y)-K_{\alpha}(x,x_0))(G)\|_{L^{q}((0,\infty),\frac{ds}{s};B)}\\
& \leq C \:
|y-x_0|\int_{0}^{1} \frac{1}{|x-(\eta y+(1-\eta)x_0)|^{n+1}}\:d\eta\:
\|G\|_{L^{q}((0,\infty),\frac{ds}{s};B)}\\
& \leq C
\frac{|y-x_0|}{(|x-x_0|-|y-x_0|)^{n+1}}
\|G\|_{L^{q}((0,\infty),\frac{ds}{s};B)}\\
& \leq C
\frac{|y-x_0|}{|x-x_0|^{n+1}}
\|G\|_{L^{q}((0,\infty),\frac{ds}{s};B)}.
\end{align*}

Hence, we get
$$
\|(K_{\alpha}(x,y)-K_{\alpha}(x,x_0))\|_
{\mathcal{L}(L^{q}((0,\infty),\frac{ds}{s};B))}
\leq C\: \frac{|y-x_0|}{|x-x_0|^{n+1}},
$$
provided that $|x-x_0|>2|x_0-y|$.

Calder\'on-Zygmund theorem for Banach valued singular integrals (\cite{RRT}) allows us to prove that the operator $T_{\alpha}$ can be extended to $L^{r}(\Rn, L^{q}((0,\infty),\frac{dt}{t};B))$ as a bounded operator from $L^{r}(\Rn, L^{q}((0,\infty),\frac{dt}{t};B))$ into itself, for every $1<r<\infty$.

By \eqref{7.1} we have that
\begin{eqnarray*}
\left|
\int_{\Rn}\int_{0}^{\infty}
\langle
t^{\alpha}\partial_{t}^{\alpha}W_{t}(f)(x), h(x,t)
\rangle_{B^{*},B}
\:\frac{dt}{t}\:dx
\right|
&\leq & C\|f\|_{L^{p'}(\Rn,B^{*})}
\|h\|_{L^{p}(\Rn, L^{q}((0,\infty),\frac{dt}{t};B))}\\
& \leq &C\:\|f\|_{L^{p'}(\Rn,B^{*})}.
\end{eqnarray*}

The arbitrariness of $\eps$ allows us to obtain
$$
\|g_{q',W_{t};B}^{\alpha}(f)\|_{L^{p'}(\Rn)}
\leq C\:\|f\|_{L^{p'}(\Rn,B^{*})}.
$$
According to Theorem \ref{thm2}, $B^{*}$ is of martingale cotype $q'$. Then, $B$ is of martingale type $q$ (see \cite[Lemma 3.2]{OX}).
Thus the proof is complete.
\end{proof}

\begin{rem}
Note that according to \eqref{A1} it is sufficient to prove Theorem \ref{thm7} when $\alpha\in\N$, $\alpha\geq1$.
\end{rem}

We now characterize smoothness property for a Banach space by using area integrals.
In \cite{TZ} the following result was established.

\begin{thm}(\cite[Theorem 5.4]{TZ})\label{thm8}
Let $B$ be a Banach space, $1<q\leq 2$, $1<p<\infty$, and $\alpha >0$. The following assertions are equivalents:
\begin{enumerate}
\item[(i)] There exists a norm $\left\VERT\cdot\right\VERT$ on $B$ that is equivalent to $\|\cdot\|$ and such that $(B,\left\VERT\cdot\right\VERT)$ is $q$-uniformly smooth.
\item[(ii)] There exists $C>0$ such that
$$
\|f\|_{\LpRB}\leq C \|\mathcal{A}_{q,P_{t};B}^{\alpha}(f)\|_{\LpR},
\quad f\in\LpR\otimes B.
$$
\end{enumerate}
\end{thm}

We now establish a heat version of Theorem \ref{thm8}.
\begin{thm}\label{thm9}
Let $B$ be a Banach space, $1<q\leq 2$, $1<p<\infty$, and  $\alpha >0$. The following assertions are equivalents:
\begin{enumerate}
\item[(a)] There exists a norm $\left\VERT\cdot\right\VERT$ on $B$ that is equivalent to $\|\cdot\|$ and such that $(B,\left\VERT\cdot\right\VERT)$ is $q$-uniformly smooth.
\item[(b)] There exists $C>0$ such that
$$
\|f\|_{\LpRB}\leq C \|\mathcal{A}_{q,W_{t};B}^{\alpha}(f)\|_{\LpR},
\quad f\in \CcR\otimes B.
$$
\end{enumerate}
\end{thm}

\begin{proof}
$(a)\Rightarrow (b)$. Suppose that there exists a norm $\left\VERT\cdot\right\VERT$ on $B$ that defines the original topology of $B$ and such that $(B,\left\VERT\cdot\right\VERT)$ is $q$-uniformly smooth.

By proceeding as in the proof of \cite[Proposition 2.1, $(a)$]{AHM} we can prove that
$$
\|g_{q,W_{t};B}^{\alpha}(f)\|_{\LpR}\leq
C \|\mathcal{A}_{q,W_{t};B}^{\alpha}(f)\|_{\LpR},
\quad f\in\LpR\otimes B,
$$
provided that $p\leq q$. Then, according to Theorem \ref{thm6} we deduce that
$$
\|f\|_{\LpRB}\leq
C \|\mathcal{A}_{q,W_{t};B}^{\alpha}(f)\|_{\LpR},
\quad f\in\LpR\otimes B,
$$
when $p\leq q$.

Suppose that $f,g\in \CcR$. We can write
$$
W_{t}(f)(x)=(e^{-|y|^{2}t}\hat{f})^{\vee}(x), \quad x\in\Rn\quad\text{and}\quad t>0,
$$
where, for every $h\in L^{1}(\Rn)$,
$$
\hat h (y)=\frac{1}{(2\pi)^{n/2}} \int_{\Rn} e^{-iz\cdot y} h(z)dz,
\quad y\in\Rn,
$$
and $\breve{h}(y)=\hat h(-y)$, $y\in\Rn$.

We choose $m\in\N$ such that $m-1\leq \alpha<m$. Since $ \hat f \in S(\mathbb R^n)$ (the Schwartz class)
$$
\int_{\Rn} e^{-|y|^{2}t}\:|y|^{2m}\:|\hat f(y)|\:dy
\leq C/t^{m}, \quad t>0,
$$
and we have that
$$
\partial_{t}^{m} W_{t}(f)(x)=
\left(
e^{-|y|^{2}t}\:(-1)^{m}\:|y|^{2m}\:\hat f(y)
\right)^{\vee}(x),
\quad x\in\Rn\quad\text{and}\quad t>0.
$$
Also, we get
$$
\int_{t}^{\infty} (u-t)^{m-\alpha-1}
\int_{\Rn} e^{-|y|^{2}u}\:|y|^{2m}\:|\hat f(y)|\:dy\:du
\leq C
\int_{t}^{\infty} (u-t)^{m-\alpha-1}\:u^{-m}\:du<\infty, \quad t>0,
$$
and we can write
$$
\partial_{t}^{\alpha} W_{t}(f)(x)=
\left(\partial_{t}^{\alpha}
\left(e^{-|y|^{2}t}\right)
\hat f\right)(x)=(-1)^{m}
\left(e^{-|y|^{2}t}|y|^{2\alpha}\hat f\right)(x), \quad x\in\Rn\quad\text{and}\quad t>0.
$$
Plancherel equality leads to
\begin{align*}
&\int_{0}^{\infty}\int_{\Rn}
t^{\alpha}\partial_{t}^{\alpha} W_{t}(f)(x)
t^{\alpha}\partial_{t}^{\alpha} W_{t}(g)(x) \:dx\:\frac{dt}{t}
=\int_{0}^{\infty}\int_{\Rn}
e^{-2|y|^{2}t}|y|^{4\alpha}\hat f(y)\hat g(y)\:dy\:t^{2\alpha-1}\:dt\\
&=\int_{\Rn} \hat f(y)\hat g(y)
\int_{0}^{\infty} e^{-2|y|^{2}t}|y|^{4\alpha}t^{2\alpha-1}\:dt\:dy
=\frac{\Gamma(2\alpha)}{2^{2\alpha}}\int_{\Rn} \hat f(y)\hat g(y)\:dy
=\frac{\Gamma(2\alpha)}{2^{2\alpha}}\int_{\Rn} f(x)\:g(x)\:dx.
\end{align*}
\bigskip
Note that all the above integrals are absolutely convergent.

An averaging trick appearing in \cite[p. 316]{CMS} leads to

\begin{multline*}
\int_{\Rn\times(0,\infty)}
t^{\alpha}\partial_{t}^{\alpha}W_{t}(f)(x)\:
t^{\alpha}\partial_{t}^{\alpha}W_{t}(g)(x)\:dx\:\frac{dt}{t}\\
=\frac{1}{c_{n}}\int_{\Rn}
\left(
\int_{\Rn\times(0,\infty)}
t^{\alpha}\partial_{t}^{\alpha}W_{t}(f)(x)\:
t^{\alpha}\partial_{t}^{\alpha}W_{t}(g)(x)\:\chi_{B(0,1)}
\left(\frac{x-y}{\sqrt{t}}\right)
\frac{dy\:dt}{t^{n/2+1}}
\right)\:dx,
\end{multline*}
where $\chi_{B(0,1)}$ denotes the characteristic function of the unit ball in $\Rn$ and $c_{n}$ means the Lebesgue measure of $B(0,1)$.

Let now $f\in \CcR\otimes B$ and $g\in \CcR\otimes B^{*}$. We have that

\begin{align*}
&\frac{\Gamma(2\alpha)}{2^{2\alpha}}\int_{\Rn}
\langle f(x),g(x)\rangle_{B,B^{*}}\:dx=
\int_{\Rn\times(0,\infty)}
\langle
t^{\alpha}\partial_{t}^{\alpha}W_{t}(f)(x),\:
t^{\alpha}\partial_{t}^{\alpha}W_{t}(g)(x)
\rangle_{B,B^{*}}\:dx\:\frac{dt}{t}\\
&\hspace{5mm}= \frac{1}{c_{n}}\int_{\Rn}
\left(
\int_{\Rn\times(0,\infty)}
\langle
t^{\alpha}\partial_{t}^{\alpha}W_{t}(f)(x),\:
t^{\alpha}\partial_{t}^{\alpha}W_{t}(g)(x)
\rangle_{B,B^{*}}
\:\chi_{B(0,1)}
\left(\frac{x-y}{\sqrt{t}}\right)
\frac{dy\:dt}{t^{n/2+1}}
\right)\:dx.
\end{align*}
By H\"older's inequality we get
$$
\left|\int_{\Rn}
\langle f(x), g(x)\rangle_{B,B^{*}}\:dx
\right|
\leq C \:\|\mathcal{A}_{q,W_{t};B}^{\alpha}(f)\|_{\LpR}
         \|\mathcal{A}_{q',W_{t};B}^{\alpha}(g)\|_{L^{p'}(\Rn)}.
$$
Since $B^{*}$ is of $q'$-cotype of martingale (\cite[Lemma 3.2]{OX}) Theorem \ref{thm4} implies that
$$
\left|\int_{\Rn}
\langle f(x), g(x)\rangle_{B,B^{*}}\:dx
\right|
\leq C \:\|\mathcal{A}_{q,W_{t};B}^{\alpha}(f)\|_{\LpR}
         \|g\|_{L^{p}(\Rn,B^{*})}.
$$
We conclude that
$$\|f\|_{L^{p}(\Rn,B^{*})}
\leq C\:\|\mathcal{A}_{q,W_{t};B}^{\alpha}(f)\|_{\LpR}.
$$

$(b)\Rightarrow (a)$. Suppose that $(b)$ holds. If $q\leq p<\infty$, by proceeding as in the proof of \cite[Proposition 2.1, $(a)$]{AHM} we can see that
$$
\|\mathcal{A}_{q,W_{t};B}^{\alpha}(f)\|_{\LpR}
\leq C \:
\|g_{q,W_{t};B}^{\alpha}(f)\|_{\LpR},
\quad f\in \CcR\otimes B.
$$
Then, by Theorem \ref{thm7}, $(a)$ is proved.
In other cases we need to work harder. We are going to prove that $B^{*}$ is of $q'$-martingale cotype. When we show this property, by using again \cite[Lemma 3.2]{OX} we will prove that $(a)$ holds.

Let $f\in \CcR\otimes B^{*}$. Fix $\eps>0$. We choose $h\in \CcR\otimes \Cc(\Gamma(0))\otimes B$ such that
$$
\|h\|_{L^{p}(\Rn, L^{q}(\Gamma(0),\frac{dy\:dt}{t^{n+1}}; B))}=1
$$
and
$$
\|\mathcal{A}_{q',W_{t};B}^{\alpha}(f)\|_{L^{p'}(\Rn)}
\leq
\left|\int_{\Rn}\int_{\Gamma(0)}
\langle s^{\alpha}\partial_{s}^{\alpha}W_{s}(f)(x+y)_{|s=t^{2}},
h(x,y,t)\rangle_{B^{*},B}\frac{dy\:dt\:dx}{t^{n+1}}
\right|+\eps
$$

The ideas in the sequel are the same as in the proof of Theorem \ref{thm7} but now the manipulation are more involved. By \eqref{A.2} we have that
\begin{align*}
&\int_{\Rn}\int_{\Gamma(0)}
\left|
\langle (s^{\alpha}\partial_{s}^{\alpha}W_{s}(f)(x+y))_{|s=t^{2}},
h(x,y,t)
\rangle_{B^{*},B}
\right|
\frac{dy\:dt\:dx}{t^{n+1}}\\
& \hspace{1cm}\leq C
\int_{\Rn}\int_{\Gamma(0)}\int_{\Rn}
| (s^{\alpha}\partial_{s}^{\alpha}W_{s}(x+y-z))_{|s=t^{2}}|\:
\|f(z)\|_{B^{*}} \: \|h(x,y,t)\|_{B} \:
dz\:\frac{dy\:dt}{t^{n+1}}\:dx\\
&\hspace{1cm} \leq C
\int_{\Rn}\int_{\Gamma(0)}\int_{\Rn}
\frac{1}{t^{2n+1}} \|f(z)\|_{B^{*}} \: \|h(x,y,t)\|_{B} \:
dz\:dy\;dt\:dx<\infty.
\end{align*}
We can write
\begin{multline*}
\int_{\Rn}\int_{\Gamma(0)}
\langle (s^{\alpha}\partial_{s}^{\alpha}W_{s}(f)(x+y))_{|s=t^{2}},
h(x,y,t)
\rangle_{B^{*},B}
\frac{dy\:dt\:dx}{t^{n+1}}\\
=\int_{\Rn}
\langle f(z),
\int_{\Gamma(0)}\int_{\Rn}
(s^{\alpha}\partial_{s}^{\alpha}W_{s}(x+y-z))_{|s=t^{2}}\:
h(x,y,t)\:dx\:\frac{dy\:dt}{t^{n+1}}
\rangle_{B^{*},B} \:dz.
\end{multline*}

We define
$$
S_{\alpha}(H)(z)=
\int_{\Gamma(0)}\int_{\Rn}
s^{\alpha}\partial_{s}^{\alpha}W_{s}(x+y-z)_{|s=t^{2}}\:
H(x,y,t)\:dx\:\frac{dy\:dt}{t^{n+1}},
\quad H\in \CcR\otimes \Cc(\Gamma(0))\otimes B.
$$

According to $(b)$ we have that
\begin{align}
&\left|\int_{\Gamma(0)}\int_{\Rn}
\langle (s^{\alpha}\partial_{s}^{\alpha}W_{s}(f)(x+y))_{|s=t^{2}},
H(x,y,t)
\rangle_{B^{*},B}
\frac{dy\:dt\:dx}{t^{n+1}}
\right|\nonumber\\
&\hspace{2cm}\leq \|f\|_{L^{p'}(\Rn,B^{*})}
\|S_{\alpha}(h)\|_{\LpRB}\nonumber\\
& \hspace{2cm}\leq C\:
\|f\|_{L^{p'}(\Rn,B^{*})}
\|\mathcal{A}_{g,W_{t},B}^{\alpha}(S_{\alpha}(h))\|_{\LpR}\nonumber\\
&\hspace{2cm}= C \:
\|f\|_{L^{p'}(\Rn,B^{*})}
\|(s^{\alpha}\partial_{s}^{\alpha}W_{s}(S_{\alpha}(h))(x+y))_{|s=t^{2}}\|_{L^{p}(\Rn,L^{q}(\Gamma(0),\frac{dy\:dt}{t^{n+1}};B))}.\label{8.1}
\end{align}
We can write
$$
s^{\alpha}\partial_{s}^{\alpha}W_{s}(S_{\alpha}(h))(x+y)_{|s=t^{2}}
=\int_{\Rn}\int_{\Gamma(0)}
K_{\alpha}(x,y,t;x_{1},y_{1},t_{1})\:
h(x_{1},y_{1},t_{1})\:
\frac{dy_{1}\:dt_{1}}{t_{1}^{n+1}}\:dx_{1},
$$
where
\begin{eqnarray*}
K_{\alpha}(x,y,t;x_{1},y_{1},t_{1})&=&
\int_{\Rn}
s^{\alpha}\partial_{s}^{\alpha}W_{s}(x_{1}-y_{1}-z)_{|s=t_{1}^{2}}
s^{\alpha}\partial_{s}^{\alpha}W_{s}(x+y-z)_{|s=t^{2}} \:dz\\
&=& t^{2\alpha}\:t_{1}^{2\alpha}
\partial_{u}^{\alpha}\partial_{v}^{\alpha}
W_{u+v}((x_{1}+y_{1})-(x+y))_{|u=t^{2}\:,v=t_{1}^{2}},
\end{eqnarray*}
$x,y,x_{1},y_{1}\in\Rn$ and $t,t_{1}\in(0,\infty)$.

We define the operator $T_{\alpha}$ by
$$
T_{\alpha}(F)(x,y,t)=
\int_{\Rn}\int_{\Gamma(0)}
K_{\alpha}(x,y,t;x_{1},y_{1},t_{1})\:
F(x_{1},y_{1},t_{1})
\frac{dy_{1}\:dt_{1}}{t_{1}^{n+1}}\;dx_{1},
$$
for every $F\in L^{p,q}(\Rn\times\Gamma(0),dx_{1}\:\frac{dy_{1}\:dt_{1}}{t_{1}^{n+1}};B)$
Here $L^{p,q}(\Rn\times\Gamma(0),dx_{1}\:\frac{dy_{1}\:dt_{1}}{t_{1}^{n+1}};B)$ can be identified by $L^{p}(\Rn,L^{q}(\Gamma(0),\frac{dy_{1}\:dt_{1}}{t_{1}^{n+1}};B))$.

The operator $T_{\alpha}$ defines a bounded operator from $L^{q}(\Rn,L^{q}(\Gamma(0),\frac{dy\:dt}{t^{n+1}};B))$ into itself. Indeed, let $F\in L^{q}(\Rn,L^{q}(\Gamma(0),\frac{dy\:dt}{t^{n+1}};B))$. We have that
\begin{multline*}
\|T_{\alpha}(F)(x,y,t)\|_{B}
\leq
\left(
\int_{\Rn}\int_{\Gamma(0)}
|K_{\alpha}(x,y,t;x_{1},y_{1},t_{1})|\:
dx_{1}\:\frac{dy_{1}\:dt_{1}}{t_{1}^{n+1}}
\right)^{1/q'}\\
\times
\left(
\int_{\Rn}\int_{\Gamma(0)}
|K_{\alpha}(x,y,t;x_{1},y_{1},t_{1})|\:
\|F(x_{1},y_{1},t_{1})\|_{B}^{q}\:
dx_{1}\:\frac{dy_{1}\:dt_{1}}{t_{1}^{n+1}}
\right)^{1/q},
 x\in\Rn,\:(t,y)\in\Gamma(0).
\end{multline*}

We take $m\in\N$ such that $m-1\leq\alpha<m$. According to \eqref{7.2} we get
\begin{align*}
&\int_{\Rn}\int_{\Gamma(0)}
|K_{\alpha}(x,y,t;x_{1},y_{1},t_{1})|\:
dx_{1}\:\frac{dy_{1}\:dt_{1}}{t_{1}^{n+1}}\\
&\hspace{8mm}\leq C
\int_{\Rn}\int_{\Gamma(0)}
t^{2\alpha}\:t_{1}^{2\alpha}
\int_{t^{2}}^{\infty} (u-t^{2})^{m-\alpha-1}\\
&\hspace{1cm}\times \int_{t_{1}^{2}}^{\infty}
|\partial_{z}^{2m}W_{z}((x_{1}+y_{1})-(x+y))_{z=u+v}|
(v-t_{1}^{2})^{m-\alpha-1}
dv\:du\:\frac{dy_{1}\:dt_{1}}{t_{1}^{n+1}}\:dx_{1}\\
&\hspace{8mm}\leq C
\int_{\Rn}\int_{\Gamma(0)}
t^{2\alpha}\:t_{1}^{2\alpha}
\int_{t^{2}}^{\infty} (u-t^{2})^{m-\alpha-1}
\int_{t_{1}^{2}}^{\infty}
\frac{e^{-c\frac{|x+y-x_{1}-y_{1}|^{2}}{u+v}}}{(u+v)^{n/2+2m}}
(v-t_{1}^{2})^{m-\alpha-1}
dv\:du\:\frac{dy_{1}\:dt_{1}}{t_{1}^{n+1}}\:dx_{1}\\
&\hspace{8mm}\leq C
\int_{\Gamma(0)} t^{2\alpha}\:t_{1}^{2\alpha}
\int_{t^{2}}^{\infty} (u-t^{2})^{m-\alpha-1}
\int_{t_{1}^{2}}^{\infty} \frac{1}{(u+v)^{2m}}
(v-t_{1}^{2})^{m-\alpha-1}
dv\:du\:\frac{dy_{1}\:dt_{1}}{t_{1}^{n+1}}\\
&\hspace{8mm}\leq C
\int_{\Gamma(0)}
\frac{t^{2\alpha}\:t_{1}^{2\alpha}}{(t^{2}+t_{1}^{2})^{2\alpha}}
\:\frac{dy_{1}\:dt_{1}}{t_{1}^{n+1}}\\
&\hspace{8mm}\leq C
\int_{0}^{\infty}
\frac{t^{2\alpha}\:t_{1}^{2\alpha-1}}{(t^{2}+t_{1}^{2})^{2\alpha}}
dt_{1}\\
&\hspace{8mm}= C\int_{0}^{\infty} \frac{z^{2\alpha-1}}{(z^{2}+1)^{2\alpha}}<\infty,
\quad x\in\Rn,\:(y,t)\in\Gamma(0).
\end{align*}

Then, we obtain
$$
\|T_{\alpha}(F)\|_{L^{q}(\Rn,\:L^{q}(\Gamma(0)),\frac{dy\:dt}{t^{n+1}};B))}
\leq
C\:\|F\|_{L^{q}(\Rn,L^{q}(\Gamma(0)),\frac{dy\:dt}{t^{n+1}};B))}.
$$

Let $x,x_{1}\in\Rn$, $x\neq x_{1}$. We consider the operator $\mathbb{K}_{\alpha}(x,x_{1})$ defined by
\begin{align*}
\mathbb{K}_{\alpha}(x,x_{1}):
L^{q}&(\Gamma(0),\frac{dx\:dt}{t^{n+1}};B)\to
L^{q}(\Gamma(0),\frac{dx\:dt}{t^{n+1}};B)\\
& G \to [\mathbb{K}_{\alpha}(x,x_{1})(G)](y,s)=
\int_{\Gamma(0)}K_{\alpha}(x,y,s;x_{1},y_{1},t_{1})\:
G(y_{1},t_{1})\:\frac{dy_{1}\:dt_{1}}{t_{1}^{n+1}}
\end{align*}
$\mathbb{K}_{\alpha}(x,x_{1})$ is a bounded operator from
$L^{q}(\Gamma(0),\frac{dx\:dt}{t^{n+1}};B)$ into itself. Indeed, let $G\in L^{q}(\Gamma(0),\frac{dx\:dt}{t^{n+1}};B)$. We can write
\begin{align*}
\|[\mathbb{K}_{\alpha}(x,x_{1})(G)](y,s)\|_{B}
& \leq
\left(
\int_{\Gamma(0)}|K_{\alpha}(x,y,s;x_{1},y_{1},t_{1})|
\frac{dy_{1}\:dt_{1}}{t_{1}^{n+1}}
\right)^{1/q'}\\
&\cdot
\left(
\int_{\Gamma(0)}|K_{\alpha}(x,y,s;x_{1},y_{1},t_{1})|\:
\|G(y_{1},t_{1})\|_{B}^{q}
\frac{dy_{1}\:dt_{1}}{t_{1}^{n+1}}
\right)^{1/q},
\quad (y,s)\in\Gamma(0).
\end{align*}

By proceeding as above, distinguising the cases  $|x-x_{1}|\geq 2(s+t_{1})$ and $|x-x_{1}|\leq 2(s+t_{1})$, we get
$$
\int_{\Gamma(0)}|K_{\alpha}(x,y,s;x_{1},y_{1},t_{1})|
\frac{dy_{1}\:dt_{1}}{t_{1}^{n+1}}
\leq \frac{C}{|x-x_{1}|^{n}},
\quad (y,s)\in\Gamma(0),
$$
and then,
\begin{equation}\label{8.2}
\|\mathbb{K}_{\alpha}(x,x_{1})(G)\|_
{L^{q}(\Gamma(0),\frac{dy\:dt}{t^{n+1}};B)}
\leq \frac{C}{|x-x_{1}|^{n}}
\|G\|_
{L^{q}(\Gamma(0),\frac{dy\:dt}{t^{n+1}};B)}.
\end{equation}
We denote by $\mathcal{L}(L^{q}(\Gamma(0),\frac{dx\:dt}{t^{n+1}};B))$ the space of bounded and linear operators from $L^{q}(\Gamma(0),\frac{dx\:dt}{t^{n+1}};B)$ into itself. By \eqref{8.2} we have that
$$
\|\mathbb{K}_{\alpha}(x,x_{1})\|_
{\mathcal{L}(L^{q}(\Gamma(0),\frac{dx\:dt}{t^{n+1}};B))}
\leq \frac{C}{|x-x_{1}|^{n}}.
$$

We now can prove as above that, for every $i=1,\ldots,n$,
$$
\int_{\Gamma(0)}|\partial_{x_{1}}K_{\alpha}(x,y,s;x_{1},y_{1},t_{1})|
\frac{dy_{1}\:dt_{1}}{t_{1}^{n+1}}
\leq \frac{C}{|x-x_{1}|^{n+1}},
\quad (y,s)\in\Gamma(0),
$$
and we deduce that
$$
\|\mathbb{K}_{\alpha}(x,x_{1})-\mathbb{K}_{\alpha}(x,x_{2})\|_
{\mathcal{L}^{q}(L^{q}(\Gamma(0),\frac{dx\:dt}{t^{n+1}};B))}
\leq \frac{|x_{1}-x_{2}|}{|x-x_{2}|^{n+1}}.
$$
provided that $|x-x_{2}|>2|x_{1}-x_{2}|$.

Let now $G\in \Cc(\Rn,L^{q}(\Gamma(0),\frac{dy\:dt}{t^{n+1}})\otimes B)$. Suposse that $x\notin\sup(G)$ and $H\in \Cc(\Gamma(0))\otimes B^{*}$. By using \eqref{8.2} we get
\begin{align*}
\int_{\Gamma(0)}
&\left\langle
\left(\int_{\Rn}
\mathbb{K}_{\alpha}(x,x_{1})(G(x_{1}))\:dx_{1}
\right)(y,s),
H(y,s)
\right\rangle_{B,B^{*}}
\frac{dy\:ds}{s^{n+1}}\\
& =\int_{\Rn} \int_{\Gamma(0)}
\langle
[\mathbb{K}_{\alpha}(x,x_{1})(G(x_{1}))](y,s),
H(y,s)
\rangle_{B,B^{*}}
\frac{dy\:ds}{s^{n+1}}\:dx_{1}\\
& =\int_{\Gamma(0)}
\left\langle
\int_{\Rn} \int_{\Gamma(0)}
K_{\alpha}(x,y,s;x_{1},y_{1},s_{1})\:[G(x_{1})](y,s_{1})\:
\frac{dy_{1}\:ds_{1}}{s_{1}^{n+1}}\:dx_{1},
H(y,s)
\right\rangle_{B,B^{*}}\frac{dy\:ds}{s^{n+1}}.
\end{align*}
Then
\begin{multline*}
\left( \int_{\Rn}
\mathbb{K}_{\alpha}(x,x_{1})(G(x_{1}))\:dx_{1}
\right)(y,s)\\
=\int_{\Rn} \int_{\Gamma(0)}
K_{\alpha}(x,y,s;x_{1},y_{1},s_{1})\:[G(x_{1})](y_{1},s_{1})\:
\frac{dy_{1}\:ds_{1}}{s_{1}^{n+1}}\:dx_{1},
\quad\text{a.e.}\:\:(y,s)\in\Gamma(0).
\end{multline*}

By using Calder\'on-Zygmund theorem for Banach valued singular integrals (\cite{RRT}) we deduce that the operator $T_{\alpha}$ can be extended to $L^{r}(\Rn,L^{q}(\Gamma(0),\frac{dy_{1}\:dt_{1}}{t_{1}^{n+1}};B))$  as a bounded operator from  $L^{r}(\Rn,L^{q}(\Gamma(0),\frac{dy_{1}\:dt_{1}}{t_{1}^{n+1}};B))$ into itself, for every $1<r<\infty$.

From \eqref{8.1}  it follows that
\begin{align*}
&\left|
\int_{\Gamma(0)}\int_{\Rn}
\langle
(s^{\alpha}\partial_{s}^{\alpha}W_{s}(f)(x+y))_{|s=t^{2}},
h(x,y,t)
\rangle_{B^{*},B}
\frac{dy\:dt\:dx}{t^{n+1}}
\right|\\
& \hspace{25mm}\leq C\: \|f\|_{L^{p'}(\Rn,B^{*})}
\|h\|_{L^{p}(\Rn,\:L^{q}(\Gamma(0),\frac{dy\:dt}{t^{n+1}};B))}\\
& \hspace{25mm}\leq C\: \|f\|_{L^{p'}(\Rn,B^{*})}.
\end{align*}

The arbitrariness of $\eps$ leads to
$$
\|\mathcal{A}_{q',W_{t};B}^{\alpha}(f)\|_{L^{p'}(\Rn)}
\leq C\: \|f\|_{L^{p'}(\Rn,B^{*})}
$$

By using Theorem \ref{thm4} we prove that $B^{*}$ is of martingale cotype $q'$.

Thus the proof is finished.
\end{proof}

\section{Results involving the heat semigroup for the Hermite operator.}

We now obtain characterization of martingale cotype and martingale type Banach spaces by using $g$-functions and area integrals defined by the heat semigroup associated with the Hermite operator (Theorems \ref{thm1.9} and \ref{thm1.10}). We recall that the semigroup $\{W_t^\mathfrak{H}\}_{t>0}$ is not Markovian.
\begin{thm}\label{thm2.1}
Let $B$ be a Banach space, $2 \leq q < \infty$, $1<p<\infty$, and $\alpha >0$. The following assertions are equivalent.
\begin{enumerate}
\item[(a)] There exists a norm $\left\VERT\cdot\right\VERT$ on $B$ that is equivalent to $\|\cdot\|$ and such that $(B;\left\VERT\cdot\right\VERT)$ is q-uniformly convex.
\item[(b)] There exists $C>0$ such that
$$
\left\|g^\alpha_{q,W_t^\mathfrak{H};B}(f) \right\|_{L^p(\mathbb R^n)} \leq C\|f\|_{L^p(\mathbb R^n,B)},\;\; f \in L^p(\mathbb R^n,B).
$$
\end{enumerate}
\end{thm}
\begin{proof}
Our first objective is to prove that the property (b) is equivalent to the following one

(b') There exists $C>0$ such that
$$
\left\|g^1_{q,W_t^\mathfrak{H};B}(f) \right\|_{L^p(\mathbb R^n)} \leq C\|f\|_{L^p(\mathbb R^n,B)},\;\; f \in L^p(\mathbb R^n,B).
$$
We have that, for every $f \in L^p(\mathbb R^n)$, $1 \leq p \leq \infty$,
$$
W^\mathfrak{H}_t(f)(x)= \int_{\mathbb R^n} W^\mathfrak{H}_t(x,y)f(y) dy,\;\; t>0,
$$
where, for every $x,y \in \mathbb R^n$ and $t>0$, $W_t^H(x,y)$ can be written
$$
W^\mathfrak{H}_t(x,y) = \frac{1}{\pi^{\frac{n}{2}}} \left( \frac{e^{-2t}}{1-e^{-4t}}\right)^{n/2} \exp\left(-\frac{1}{4}\left(|x-y|^2\frac{1+e^{-2t}}{1-e^{-2t}}+ |x+y|^2\frac{1 -e^{-2t}}{1 + e^{-2t}}\right)\right).
$$

For every $t>0$, $W^\mathfrak{H}_t$ is a positive operator. Then, for every $t>0$ and $1 \leq p < \infty$, $W_t^\mathfrak{H}$ has a (tensor) extension to $L^p(\mathbb R^n,B)$ satisfying the same $L^p$-boundedness properties than the corresponding scalar operator.

In the sequel we will use that, for every $k\in \mathbb N$,
\begin{equation}\label{H1}
\left|\frac{\partial^k}{\partial t^k} W_t^\mathfrak{H}(x,y)\right| \leq C\frac{e^{-c\frac{|x-y|^2}{t}}}{t^{\frac{n}{2}+k}},\;\; x,y \in \mathbb R^n\;\;\mbox{and}\;\;t>0.
\end{equation}
(See \cite[Lemma 2.5]{CD} for a proof of this property in a general situation).
Let $N \in \mathbb N$. We consider the space $H_N=L^q\left(\left(\frac{1}{N},\infty\right), \frac{dt}{t},B\right)$. Suppose that $f \in C_c^\infty(\mathbb R^n) \otimes B$. We now proceed as in the proof of the Theorem \ref{thm2}. We can see that, for every $x \in \mathbb R^n$, the function
\begin{equation}\label{H2}
\left[F_f(x)\right](t) = t^\alpha \partial_t^\alpha W_t^\mathfrak{H}(f)(x),\;\; t>0,
\end{equation}
is in $H_N$. Moreover, $F_f$ is $H_N$-strongly measurable in $\mathbb R^n$.

Assume that (b) holds. We consider the operator $T_\alpha$ defined by $T_\alpha(f)=F_f$, for every $f \in C_c^\infty(\mathbb R^n)\otimes B$ where $F_f$ is defined by \eqref{H2}. (b) says us that
$$
\|T_\alpha f\|_{L^p\big(\mathbb R^n,L^q\big(\big(\frac{1}{N},\infty\big),\frac{dt}{t},B\big)\big)}\leq C\|f\|_{L^p(\mathbb R^n,B)},\;\; f \in C_c^\infty(\mathbb R^n,B).
$$
We have that, for every $f \in C_c^\infty(\mathbb R^n)\otimes B$,
$$
(T_\alpha f)(x) = \int_{\mathbb R^d}K_\alpha(x,y)f(y)dy,\;\; x \in \mathbb R^d,
$$
where $\left[K_\alpha(x,y)\right](t)= t^\alpha\partial^\alpha_tW_t^\mathfrak{H}(x,y)$, $x,y \in \mathbb R^n$ and $t>\frac{1}{N}$. Here the integral is understood in $L^q\left(\left(\frac{1}{N},\infty\right), \frac{dt}{t},B\right)$-B\"{o}chner sense.

According to \eqref{H1} and by taking $m\in \mathbb N$ such that $m-1 \leq \alpha < m$, we get
\begin{eqnarray}\label{H3}
\|K_\alpha(x,y)\|_{H_N} &\leq & C\left(\int^\infty_0\left|t^\alpha\int_t^\infty \partial_s^m W_s^\mathfrak{H}(x,y)(s-t)^{m-\alpha-1} ds \right|^q \frac{dt}{t}\right)^{\frac{1}{q}} \nonumber\\
&\leq & C \left( \int_0^\infty \left|t^\alpha\int_t^\infty \frac{e^{-c \frac{|x-y|^2}{s}}}{s^{\frac{n}{2}+m}}(s-t)^{m-\alpha-1}ds\right|^q\frac{dt}{t}\right)^{\frac{1}{q}}\nonumber\\
&\leq&\frac{C}{|x-y|^n}\;\;x,y \in \mathbb R^n\;\;\mbox{and}\;\;x \not=y .
\end{eqnarray}

Also, by defining $[H_{\alpha,i}(x,y)](t)=t^\beta\partial^\beta_t\partial_{z_i}W_t^\mathfrak{H}(x,y)$, $x,y \in\mathbb R^n$, $i=1,\cdots,n$, and $t\in (\frac{1}{N},\infty)$, we obtain
\begin{equation}\label{H4}
\|H_{\alpha,i}(x,y)\|_{H_N} \leq \frac{C}{|x-y|^{n+1}},\;\;x,y \in \mathbb R^d\;\;,x \not=y\;\;\mbox{and}\;\; i=1,\cdots,n.
\end{equation}
The constant $C>0$ in \eqref{H3} and \eqref{H4} does not depend on $N\in \mathbb N$.

Calder\'on-Zygmund theorem for Banach-valued singular integrals implies that, for every, $1<r<\infty$, $T_\alpha$ can be extended to $L^r(\mathbb R^n,B)$ as a bounded operator from $L^r(\mathbb R^n, B)$ into $L^r(\mathbb R^n,H_N)$ and
$$
\|T_\alpha f\|_{L^r(\mathbb R^n,H_N)} \leq C\|f\|_{L^r(\mathbb R^n,B)},\;\; f \in L^r(\mathbb R^n,B),
$$
where $C$ is independent of $N$.

Fatou lemma leads to
\begin{equation}\label{H5}
\left\|g_{q,W_t^\mathfrak{H};B}^\alpha(f)\right\|_{L^r(\mathbb R^n)} \leq C\|f\|_{L^r(\mathbb R^n,B)},\;\; f \in L^r(\mathbb R^n,B),\;\mbox{and}\;1<r<\infty.
\end{equation}
We are going to see that, for every $k\in \mathbb N$ and $1<r<\infty$, there exists $C>0$ such that
\begin{equation}\label{H6}
\left\|g_{q,W_t^\mathfrak{H};B}^{k\alpha}(f)\right\|_{L^r(\mathbb R^n)} \leq C\|f\|_{L^r(\mathbb R^n,B)},\;\; f \in L^r(\mathbb R^n,B).
\end{equation}
We use an inductive procedure. \eqref{H5} says us that \eqref{H6} holds for $k=1$. Suppose that, for a certain $\ell \in \mathbb N$, $g_{q,W_t^\mathfrak{H};B}^{\ell\alpha}$ defines a bounded operator from $L^r(\mathbb R^n, B)$ into $L^r(\mathbb R^n)$ for every (equivalently, for some $r \in \mathbb N$). We define the operator $T$ as follows: for every $ f \in L^q(\mathbb R^n,B)$,
$$
[T(f)(x)](t) = t^{\ell\alpha}\partial_t^{\ell\alpha}W_t^\mathfrak{H}(f)(x),\;\; x \in \mathbb R^n\;\;\rm{and}\;\;t>0.
$$
The operator $T$ is bounded from $L^q(\mathbb R^n,B)$ into $L^q(\mathbb R^n,L^q((0,\infty),\frac{dt}{t}))$.

We also consider the following operator
$$
\mathcal{H}: L^q(\mathbb R^n,L^q((0,\infty),\frac{dt}{t},B)) \rightarrow L^q(\mathbb R^n,L^q((0,\infty)^2,\frac{dtds}{ts},B))
$$
defined by
$$
[\mathcal{H}(h)(x)](s,t) = s^\alpha \partial_s^\alpha W_s^\mathfrak{H}([h(\cdot)](t))(x),\;\; x \in \mathbb R^n,\;\;{\rm{and}}\;\; s,t >0.
$$
From \eqref{H5} with $r=q$ we deduce that, for every $h \in L^q(\mathbb R^n,L^q((0,\infty), \frac{dt}{t},B))$,
$$
\|\mathcal{H}(h)\|_{L^q(\mathbb R^n,L^q((0,\infty)^2, \frac{dtds}{ts},B))}\leq C\|h\|_{L^q(\mathbb R^n,L^q((0,\infty), \frac{dt}{t},B))}.
$$
Hence, the operator $\mathfrak{L}=\mathcal{H} \circ T$ is bounded from $L^q(\mathbb R^n,B)$ into $L^q(\mathbb R^n,L^q((0,\infty)^2, \frac{dtds}{ts},B))$.

By using \eqref{H1} we deduce that, for every $f \in L^q(\mathbb R^n,B)$,
$$
[\mathfrak{L}(f)(x)](s,t)= s^\alpha t^{\ell\alpha}\partial_u^{(\ell+1)\alpha}W_u^\mathfrak{H}(f)(x)_{|u=s+t}, \;x\in \mathbb R^n\;\;\rm{and}\;\; s,t>0.
$$
By proceeding as in \cite[(3.7)]{TZ} we can see that
$$
\left\|g^{(\ell+1)\alpha}_{q,W_t^\mathfrak{H};B}(f)\right\|_{L^q(\mathbb R^n)} \leq C\|f\|_{L^q(\mathbb R^n,B)},\;\;f\in L^q(\mathbb R^n,B).
$$
By applying the Calder\'on-Zygmund theorem for Banach-valued singular integrals (\cite{RRT}) we deduce that $g^{(\ell+1)\alpha}_{q,W_t^\mathfrak{H};B}$ defines a bounded sublinear operator from $L^r(\mathbb R^n,B)$ into $L^r(\mathbb R^n)$, for every $1<r<\infty$.

Thus \eqref{H6} is proved for every $k \in \mathbb N$ and $1<r<\infty$.

By \eqref{A1} and by choosing $k\in \mathbb N$ such that $k\alpha \geq 1$ we see that there exists $C>0$ such that
$$
\left\|g^1_{q,W_t^\mathfrak{H};B}(f)\right\|_{L^p(\mathbb R^n)} \leq C \|f\|_{L^p(\mathbb R^n,B)},\;\; f\in L^p(\mathbb R^n,B).
$$
Thus (b') is established.

The same argument allows us to prove that (b') implies (b).

Our objective is to see that the properties (a) and (b') are equivalent.

$(b') \Longrightarrow (a)$. We denote by $\{P_t^\mathfrak{H}\}_{t>0}$ the Poisson semigroup associated with Hermite operator. By using subordination formula we can write
$$
P_t^\mathfrak{H}(f)=\frac{1}{\sqrt \pi}\int_0^\infty \frac{e^{-v}}{{\sqrt v}}W^{\mathfrak{H}}_{\frac{t^2}{4v}}(f) du,\;\; f \in L^r(\mathbb R^n,B),\;1\leq r\leq \infty.
$$
By tacking in account \eqref{H1} we can see that, for every $f\in L^r(\mathbb R^n,B)$, $1 \leq r \leq \infty$,
$$
g^1_{q,P_t^\mathfrak{H};B}(f) \leq Cg^1_{q,W_t^\mathfrak{H};B}(f).
$$
Assume that $(b')$ holds. Then, for every $f\in L^{p}(\mathbb R^n,B)$,
$$
\left\|g^1_{q,P_{t}^{\mathfrak{H}};B}(f)\right\|_{L^p(\mathbb R^n)}\leq C\|f\|_{L^p(\mathbb R^n,B)}.
$$
According to now \cite[Theorem A]{AST}, (a) is true.

$(a) \Longrightarrow (b')$. Suppose that $(a)$ holds. By Theorem \ref{thm1.4} there exists $C>0$ such that
\begin{equation}\label{H7}
\left\|g^1_{q,W_t;B}(f)\right\|_{L^p(\mathbb R^n)} \leq C\|f\|_{L^p(\mathbb R^n,B)},\;\; f \in L^p(\mathbb R^n,B).
\end{equation}

We define
$$
\rho(x)=\left\{\begin{array}{ll}
                \frac{1}{2},& |x|\leq 1,\\
                \frac{1}{1+|x|}, & |x| > 1
                \end{array}\right..
$$
For every $x \in \mathbb R^n$, $\rho(x)$ is called the critical radius in $x$ (\cite[p. 516]{Sh}).

If $\{T_t\}_{t>0}$ is a semigroup of operators in $L^p(\mathbb R^n,B)$ we consider, for every $f \in L^p(\mathbb R^n,B)$, $G_{T_t}(f)(x,t)=t\partial_tT_t(f)(x)$, $x\in \mathbb R^n$ and $t>0$,
$$
G_{T_t,loc}(f)(x,t)=t\partial_tT_t(f\chi_{B(x,\rho(x))})(x), \;\;x\in \mathbb R^n\;\; \rm{and}\;\; t>0,
$$
and
$$
G_{T_t,glob}(f)(x,t) = t \partial_t T_t(f\chi_{(B(x,\rho(x)))^c})(x),\;\; x \in \mathbb R^n\;\;\rm{and}\;\; t>0.
$$
Let $f \in L^p(\mathbb R^n,B)$. We can write
\begin{eqnarray*}
\left|g^1_{q,W_t^\mathfrak{H};B}(f)(x) - g^1_{q,W_t;B}(f)(x) \right| &\leq &\left( \int_0^\infty \left\|t\partial_tW_t^\mathfrak{H}(f)(x) - t\partial_tW_t(f)(x)\right\|^q_B \frac{dt}{t}\right)^{\frac{1}{q}}\\
& \leq& \left(\int_0^\infty\left\|G_{W^\mathfrak{H}_t,loc}(f)(x,t) - G_{W_t,loc}(f)(x,t)\right\|_B^q \frac{dt}{t}\right)^{\frac{1}{q}} \\
&& + \left(\int_0^\infty \left\|G_{W^\mathfrak{H}_t,glob}(f)(x,t)\right\|^q\frac{dt}{t}\right)^{\frac{1}{q}}\\
&&+\left(\int_0^\infty \left\|G_{W_t,glob}(f)(x,t)\right\|^q\frac{dt}{t}\right)^{\frac{1}{q}},\;x\in \mathbb R^n.
\end{eqnarray*}
Our objective is to prove that the operators $G_{W^\mathfrak{H}_t,loc}-G_{{W_t, loc}}\:$, $\quad G_{W^\mathfrak{H}_t,glob}$ and $G_{W_t,glob}$ are bounded from $L^p(\mathbb R^n,B)$ into $L^p(\mathbb R^n, L^q((0,\infty),\frac{dt}{t},B))$. When we establish this objective, by tacking into account \eqref{H7}, we have that there exists $C>0$ such that
$$
\left\|g^1_{q,W^\mathfrak{H}_t;B}(f)\right\|_{L^p(\mathbb R^n)} \leq C\|f\|_{L^p(\mathbb R^n,B)},\;\; f \in L^p(\mathbb R^n,B),
$$
and $(b')$ is proved.

We use some ideas developed in \cite{BCFR}.

According to \cite[Proposition 5]{DGMTZ}, for every $M>0$ there exists $C>0$ such that
\begin{equation}\label{H8}
\frac{1}{C} \leq \frac{\rho(x)}{\rho(y)} \leq C,\;\; x \in B(y,M\rho(y)).
\end{equation}

Also there exists a sequence $(x_k)_{k=1}^\infty$ in $\mathbb R^n$ such that
\begin{enumerate}
\item[(i)] $\cup^{\infty}_{k=1} B(x_k,\rho(x_k))= \mathbb R^n$;
\item[(ii)] For every $M>0$ there exists $m\in \mathbb N$ such that, for each $j \in \mathbb N$,
$$
card\{k\in \mathbb N: B(x_k,M\rho(x_k)) \cap B(x_j,M\rho(x_j)) \not= \emptyset\} \leq m.
$$
\end{enumerate}
Let $k\in \mathbb N$. If $x \in B(x_k,\rho(x_k))$ and $y \in B(x,\rho(x))$, then

$$
|y-x_k| \leq \rho(x) + \rho(x_k) \leq C_0\rho(x_k).
$$
Note that, according to \eqref{H8}, $C_0$ does not depending on $k$. For every $x\in B(x_k,\rho(x_k))$  and $t>0$ we write
\begin{eqnarray*}
G_{W_t,loc}(f)(x,t) &=& G_{W_t}\left(\chi_{B(x_k,C_0\rho(x_k))}f\right)(x,t)
+ G_{W_t}\left(\chi_{B(x,\rho(x))} - \chi_{B(x_k,C_0\rho(x_k))}f\right)(x,t)\\
&=& G_{W_t}\left(\chi_{B(x_k,C_0\rho(x_k))}f\right)(x,t) - G_{W_t}\left(\chi_{B(x_k,C_0\rho(x_k))\setminus B(x,\rho(x))}f\right)(x,t).
\end{eqnarray*}
We have that
$$
\left\| t\partial_tW_t(x-y)\right\|_{L^q((0,\infty),\frac{dt}{t})} \leq C\left( \int_0^\infty\frac{e^{-c\frac{|x-y|^2}{t}}}{t^{\frac{n}{2}q+1}}dt\right)^{\frac{1}{q}} \leq \frac{C}{|x-y|^n},\;\;x,y\in \mathbb R^n,\;x\not=y.
$$
Integral Minkowski inequality and \eqref{H8} leads, for every $f\in L^p(\mathbb R^n,B)$, to
\begin{eqnarray*}
&&\left\|G_{W_t}\left(\chi_{B(x_k,C_0\rho(x_k))\setminus B(x,\rho(x))}f\right)(x,\cdot)\right\|_{L^q((0,\infty),\frac{dt}{t},B)} \\
&&\hspace{7mm} \leq \int_{B(x_k,C_0\rho(x_k))\setminus B(x,\rho(x))}\|t\partial_tW_t(x-y)\|_{L^q((0,\infty),\frac{dt}{t})}\|f(y)\|_Bdy\\
&&\hspace{7mm} \leq\frac{C}{\rho(x)^n} \int_{B(x_k,C_0\rho(x_k))}\|f(y)\|_Bdy\\
&&\hspace{7mm} \leq\frac{C}{\rho(x_k)^n} \int_{B(x_k,C_0\rho(x_k))}\|f(y)\|_Bdy \leq C\mathcal{M}(\|f\|)(x),\;\; x\in B(x_k,\rho(x_k)).
\end{eqnarray*}
Here $\mathcal{M}$ denotes the Hardy-Littlewood maximal function.

Then,
\begin{eqnarray*}
&&\left\|G_{W_t,loc}(f)\right\|^p_{L^p(\mathbb R^n,L^q((0,\infty);\frac{dt}{t},B))}\\
&&\hspace{1cm}\leq \sum_{k=1}^\infty\int_{B(x_k,\rho(x_k))}\|G_{W_t,loc}(f)(x,\cdot)\|^p_{L^q((0,\infty),\frac{dt}{t},B)}dx\\
&&\hspace{1cm} \leq C \sum_{k=1}^\infty\left(\int_{B(x_k,\rho(x_k))}\|G_{W_t}(\chi_{B(x_k,C_0\rho(x_k))}f)(x,\cdot)\|^p_{L^q((0,\infty)\frac{dt}{t},B)}\right.\\
&&\hspace{14mm}+\left.\int_{B(x_k,\rho(x_k))}\|G_{W_t}(\chi_{B(x_k,C_0\rho(x_k))\setminus B(x,\rho(x))}f)(x\cdot)\|^p_{L^q((0,\infty),\frac{dt}{t},B)}\right)dx \\
&&\hspace{1cm}\leq C\sum_{k=1}^\infty\left(\|g^1_{q,W_t}(\chi_{B(x_k,C_0\rho(x_k))}f)\|^p_{L^p(\mathbb R^n)} + \int_{B(x_k,\rho(x_k))} [\mathcal{M}(\|f\|)(x)]^p dx\right).
\end{eqnarray*}

By using \eqref{H7}, the property (ii) of $(x_k)^\infty_{k=1}$ and the maximal theorem, we obtain
\begin{eqnarray*}
\|G_{W_t,loc}(f)\|^p_{L^p(\mathbb R^n)} &\leq& C \left(\sum_{k=1}^\infty\int_{B(x_k,C_0\rho(x_k))}\|f(y)\|^pdy + \int_{\mathbb R^n}[\mathcal{M}(\|f\|)(x)]^p dx\right)\\
&\leq& C \|f\|^p_{L^p(\mathbb R^n,B)},\;\; f \in L^p(\mathbb R^n,B).
\end{eqnarray*}
By using \eqref{H7} again we conclude that
$$
\left\|G_{W_t,glob}(f)\right\|_{L^p(\mathbb R^n, L^q((0,\infty),\frac{dt}{t},B))}\leq C\|f\|_{L^p(\mathbb R^n,B)},\;\; f \in L^p(\mathbb R^n,B).
$$

We now study $G_{W^{\mathfrak{H}}_t,glob}$. According to \cite[(4.4) and (4.5)]{BCFST} we have that
\begin{equation}\label{H9}
\exp\left(-\frac{1}{4}\left(|x-y|^2 \frac{1+e^{-2t}}{1-e^{-2t}}+|x+y|^2 \frac{1-e^{-2t}}{1+e^{-2t}}\right)\right)\leq Ce^{-c(|x|+|y|)|x-y|},\;\; x,y\in \mathbb R^n\;\;\text{and}\;\;t>0.
\end{equation}
Also, we can write
\begin{multline}\label{H10}
\partial_tW_t^\mathfrak{H}(x,y)= \frac{1}{\pi^{\frac{\pi}{2}}}\left(\frac{e^{-2t}}{1-e^{-4t}}\right)^{\frac{n}{2}} \exp\left[-\frac{1}{4}\left(|x-y|^2 \frac{1+e^{-2t}}{1-e^{-2t}}+|x+y|^2 \frac{1-e^{-2t}}{1+e^{-2t}}\right)\right]\\
\times\left[-n\frac{1+e^{-4t}}{1-e^{-4t}} + |x-y|^2
\frac{e^{-2t}}{(1-e^{-2t})^{2}}-|x+y|^{2}\frac{e^{-2t}}{(1+e^{-2t})^{2}}\right], \;\; x,y\in \mathbb R^n\;\;\text{and} \;\;t>0.
\end{multline}
By combining \eqref{H9} and \eqref{H10} we obtain
$$
|t\partial_tW_t^\mathfrak{H}(x,y)| \leq C \frac{e^{-c(|x|+|y|)|x-y|}e^{-c\frac{|x-y|^2}{t}}}{t^{\frac{n}{2}}},\;\;x,y\in \mathbb R^n\;\;\text{and} \;\;t>0.
$$
Then, Minkowski integral inequality leads to
\begin{eqnarray*}
\left\| G_{W^\mathfrak{H}_t,glob}(f)(x,\cdot)\right\|_{L^q((0,\infty),\frac{dt}{t},B)}&\leq &\int_{|x-y| >\rho(x)}\|f(y)\|_B\left(\int_0^\infty |t\partial_tW_t^\mathfrak{H}(x,y)|^q\frac{dt}{t}\right)^{\frac{1}{q}}dy \\
&\leq & C\int_{|x-y|>\rho(x)} \|f(y)\|_B \frac{e^{-c(|x|+|y|)|x-y|}}{|x-y|^n}dy.
\end{eqnarray*}
If $|x-y|>\rho(x)$, we have that
$$
(|x|+|y|)\rho(x)\geq |x-y|\rho(x) \geq \rho(x)^2=\frac{1}{4},\;|x| \leq 1,
$$
and
$$
(|x|+|y|)\rho(x) \geq \frac{|x|}{1+|x|} > \frac{1}{2},\;|x| > 1.
$$
Then,
\begin{multline*}
\left\| G_{W^\mathfrak{H}_t,glob}(f)(x,\cdot)\right\|_{L^q((0,\infty),\frac{dt}{t},B)}\leq C\sum_{k=0}^\infty\int_{2^k\rho(x) <|x-y|\leq 2^{k+1}\rho(x)} \|f(y)\|_B \frac{e^{-c(|x|+|y|)|x-y|}}{|x-y|^n}dy \\
\leq C\sum^\infty_{k=0} \frac{e^{-c2^m}}{\rho(x)^n2^{kn}}\int_{|x-y|\leq 2^{k+1}\rho(x)} \|f(y)\|_B dy \leq C \mathcal{M}(\|f\|)(x),\;x\in \mathbb R^n.
\end{multline*}
It follows that
$$
\left\| G_{W^\mathfrak{H}_t,glob}(f) \right\|_{L^p(\mathbb R^n,L^q((0,\infty),\frac{dt}{t},B))} \leq C\|f\|_{L^p(\mathbb R^n,B)}.
$$
Finally, we study $G_{W_t^\mathfrak{H},loc}-G_{W_t,loc}$. We proceed as in \cite{BCFR}. Here we sketch the main steps.

By using the perturbation formula (\cite[(5.25)]{DGMTZ}) we can write
$$
\partial_t[W_t(x-y)-W_t^\mathfrak{H}(x,y)] = \sum_{j=1}^3H_j(x,y,t),\;\;x,y\in \mathbb R^n\;\;\text{and}\;\; t>0,
$$
where
$$
H_1(x,y,t) = \int_0^{\frac{t}{2}}\int_{\mathbb R^n} |z|^2 [\partial_uW_u(x-z)]_{|u=t-s}  W_{s}^{\mathfrak{H}}(z,y) dzds, \;\;x,y\in \mathbb R^n\;\;\text{and}\;\; t>0,
$$
$$
H_1(x,y,t) = \int_0^{\frac{t}{2}}\int_{\mathbb R^n} |z|^2 W_s(x-z)[\partial_uW^{\mathfrak{H}}_u(z,y)]_{|u=t-s}dzds, \;\;x,y\in \mathbb R^n\;\;\text{and}\;\; t>0,
$$
and
$$
H_3(x,y,t) = \int_{\mathbb R^n} |z|^2W_{\frac{t}{2}}(x-z) W^\mathfrak{H}_{\frac{t}{2}}(z,y)dz, \;\;x,y\in \mathbb R^n\;\;\text{and}\;\; t>0.
$$
For every $x \in \mathbb R^n$ and $t>0$, we obtain
\begin{eqnarray}\label{H11}
\int_{\mathbb R^n} e^{-c\frac{|x-y|^2}{t}}|y|^2dy &\leq & \int_{\mathbb R^n} e^{\frac{-c|z|^2}{t}}(|z|^2+|x|^2)dz \leq Ct^{\frac{n}{2}}(t+|x|^2) \nonumber\\
&\leq & C\frac{t^{\frac{n}{2}}}{\rho(x)^2},\;\;x\in \mathbb R^n\;\;\text{and}\;\; 0<t\leq\rho(x)^2.
\end{eqnarray}
Minkowski integral inequality leads to
\begin{eqnarray*}
&&\left\|G_{W^\mathfrak{H}_t,loc}(f)(x,\cdot)- G_{W_t,loc}(f)(x,\cdot)\right\|_{L^q((0,\infty),\frac{dt}{t},B)} \\
&&\hspace{1cm}\leq C\left(\sum_{j=1}^3\int_{B(x, \rho(x))}\|f(y)\|_B\left(\int_0^{\rho(x)^2}|tH_j(x,y,t)|^q\frac{dt}{t}\right)dy\right.\\
&&\hspace{15mm}+\left.\int_{B(x,\rho(x))}\|f(y)\|_B\left(\int_{\rho(x)^2}^\infty |t\partial_tW_t(x-y)-t\partial_tW_t^H(x,y)|^q\frac{dt}{t}\right)^{\frac{1}{q}}\right), \quad x\in\Rn.
\end{eqnarray*}
By using \eqref{H11} we get, for $j=1,2,3$,
$$
\left(\int_0^{\rho(x)^2} |tH_j(x,y,t)|^q \frac{dt}{t}\right)^{\frac{1}{q}} \leq \frac{C}{\sqrt{\rho(x)}\:|x-y|^{n-\frac{1}{2}}}\;, x\in\mathbb R^n,\;\; y\in B(x,\rho(x)),\;\; x \not= y.
$$
Then,
$$
\sum_{j=1}^3\int_{B(x,\rho(x))}\|f(y)\|_B \left(\int_0^{\rho(x)^2} |tH_j(x,y,t)|^q \frac{dt}{t}\right)^{\frac{1}{q}}dy \leq C\mathcal{M}(\|f\|_B),\;\;x\in \mathbb R^n
$$
Also, \eqref{H1} leads to
\begin{eqnarray*}
&&\int_{B(x,\rho(x))} \|f(y)\|_B\left(\int_{\rho(x)^2}^\infty |t\partial_tW_t(x-y)- t\partial_tW_{t}^{H}(x,y)|^q\frac{dt}{t}\right)^{\frac{1}{q}}dy \\
&&\hspace{1cm}\leq \:C\left( \int_{\rho(x)^2}^\infty \frac{dt}{t^{n\frac{q}{2}+1}}\right)^{\frac{1}{q}}\int_{B(x,\rho(x))}\|f(y)\|_Bdy \\
&&\hspace{1cm} \leq\frac{C}{\rho(x)^{n}}\int_{B(x,\rho(x))}\|f(y)\|_Bdy \leq C\mathcal{M}(\|f\|)(x),\;\;x\in \mathbb R^n.
\end{eqnarray*}
By combining the above estimates we obtain
$$
\left\|G_{W^\mathfrak{H}_t,loc}(f)-G_{W_t,loc}(f)\right\|_{L^p(\mathbb R^n,L^q((0,\infty),\frac{dt}{t},B))} \leq C\|f\|_{L^p(\mathbb R^n,B)}.
$$
Thus, the proof of this theorem is complete.
\end{proof}

\begin{thm}\label{thm2.2}
Let $B$ be a Banach space, $2\leq q < \infty$, $1<p<\infty$, and $\alpha >0$. The following assertions are equivalent.
\begin{enumerate}
\item[(a)] There exists a norm $\left\VERT \cdot \right\VERT$ on $B$ that is equivalent to $\|\cdot\|$  and such that $(B,\left\VERT\cdot \right\VERT)$ is $q$-uniformly convex.
\item[(b)] There exists $C>0$ such that
$$
\left\|\mathcal{A}^\alpha_{q,W^\mathfrak{H}_t;B}(f)\right\|_{L^p(\mathbb R^n)} \leq C\|f\|_{L^p(\mathbb R^n,B)},\;\;f \in L^p(\mathbb R^n,B).
$$
\end{enumerate}
\end{thm}
\begin{proof}
Let $N \in \mathbb N$. We define $\Gamma_N(0)=\{(y,t) \in \mathbb R^n \times (\frac{1}{N},\infty):|x-y|< t\}$ and $H_N=L^q(\Gamma_{N}(0); \frac{dydt}{t^{n+1}},B)$. We consider the operator $T_\alpha$ defined as follows: for every $f  \in C_c^\infty(\mathbb R^n)\otimes B$, $T_\alpha f=F$ where, for every $x \in \mathbb R^n$, $[F(x)](y,t)= (s^\alpha\partial_s^\alpha W^\mathfrak{H}_s(f)(x+y))_{|s=t^2}$, $y\in \mathbb R^n$ and $t>0$. By taking in account \eqref{H1} and by proceeding as in the proof of Theorem \ref{thm4} we can see that, for every $f\in C_c^\infty(\mathbb R^n) \otimes B$,
$$
(T_\alpha f)(x) = \int_{\mathbb R^n} K_\alpha(x,z) f(z) dz,\;\; x \in \mathbb R^n,
$$
where $[K_\alpha(x,z)](y,t) = (s^\alpha\partial^\alpha_s W^\mathfrak{H}_s(x+y,z))_{|s=t^2}$, $x,z,y \in \mathbb R^n$ and $t>0$. Here the integral is understood in the $H_N$-B\"{o}chner sense.
Also, we have that
$$
\|K_\alpha(x,y)\|_{L^q(\Gamma_N(0), \frac{dydt}{t^{n+1}})}\leq \frac{C}{|x-y|^n},\;\; x,y \in \mathbb R^n,\;\; x \not= y,
$$
and, for $i=1,\ldots,n$
$$
\|\partial_{x_i}K_\alpha(x,y)\|_{L^q(\Gamma_N(0), \frac{dydt}{t^{n+1}})}\leq \frac{C}{|x-y|^{n+1}},\;\; x,y \in \mathbb R^n,\;\; x \not= y,
$$
where $C>0$ does not depend on $N$.

We are going to see our result.

$(a) \Longrightarrow (b)$. Suppose that $(a)$ holds. According to Theorem \ref{thm2.1} there exists $C>0$ such that
$$
\left\|g^{\alpha}_{q,W^\mathfrak{H}_t;B}(f)\right\| \leq C \|f\|_{L^q(\mathbb R^n,B)},\;\; f \in L^q(\mathbb R^n,B).
$$
Then,
$$
\left\|\mathcal{A}^\alpha_{q,W^\mathfrak{H}_t;B}(f)\right\|_{L^q(\mathbb R^n)}\leq C\left\|g^{\alpha}_{q,W^\mathfrak{H}_t;B}(f)\right\|_{L^q(\mathbb R^n)} \leq C\|f\|_{L^q(\mathbb R^n,B)},\;\; f \in L^q(\mathbb R^n,B).
$$
Hence, for every $N\in \mathbb N$,
$$
\|T_\alpha f \|_{L^q(\mathbb R^n, H_N)} \leq C\|f\|_{L^q(\mathbb R^n,B)},\;\; f \in C_c^\infty(\mathbb R^n)\otimes B,
$$
where $C$ is independent of $N$.

By using now Calder\'on-Zygmund theorem for Banach-valued singular integrals and tacking limits as $N \rightarrow \infty$, we conclude that there exists $C>0$ such that
$$
\|\mathcal{A}_{q,W_t^\mathfrak{H};B}^\alpha (f)\| _{L^p(\mathbb R^n)} \leq C\|f\|_{L^p(\mathbb R^n,B)}, \;\; f \in L^p(\mathbb R^n,B).
$$

$(b) \Longrightarrow (a)$. Suppose that $(b)$ is true. Then, for every $N\in \mathbb N$,
$$
\|T_\alpha f\|_{L^p(\mathbb R^n,H_N)} \leq C\|f\|_{L^p(\mathbb R^n,B)},\;\; f \in C_c^\infty(\mathbb R^n) \otimes B,
$$
where $C>0$ is not depending on $N$.

By using again Calder\'on-Zygmund theorem for Banach valued singular integrals we deduce that, for every $N \in \mathbb N$,
$$
\|T_\alpha f\|_{L^q(\mathbb R^n,H_N)} \leq C\|f\|_{L^q(\mathbb R^{n},B)},\;\;f \in C_c^\infty(\mathbb R^n \otimes B,
$$
where $C>0$ is independent of $N$.

By tacking limit as $N \rightarrow \infty$ we get
$$
\left\|\mathcal{A}_{q,W_t^\mathfrak{H};B}^\alpha (f)\right\| _{L^q(\mathbb R^n)} \leq C\|f\|_{L^q(\mathbb R^n,B)}, \;\; f \in L^q(\mathbb R^n,B).
$$
Then,
\begin{eqnarray*}
\left\|g^\alpha_{q,W^\mathfrak{H}_t;B}(f) \right\|_{L^q(\mathbb R^n)} &\leq & C \left\|\mathcal{A}_{q,W_t^\mathfrak{H},B}^\alpha (f)\right\| _{L^q(\mathbb R^n)}\\
&\leq & C\|f\|_{L^q(\mathbb R^n,B)},\;\; f \in L^q(\mathbb R^n,B).
\end{eqnarray*}

According to Theorem \ref{thm2.1} we deduce that $(a)$ holds.
\end{proof}

For the heat semigroup $\{W_t^\mathfrak{H}\}_{t>0}$ defined by the Hermite operator the subspace $F$ of fixed points in $L^p(\mathbb R^n)$, $1<p<\infty$, reduces to $\{0\}$.

\begin{thm}\label{thm2.3}
Let $B$ be a Banach space , $1<q \leq 2$, $1<p< \infty$, and $\alpha >0$. The following assertions are equivalent.
\begin{enumerate}
\item[(a)] There exists a norm $\left\VERT\cdot\right\VERT$ on $B$ that is equivalent to $\|\cdot\|$ and such that $(B,\left\VERT\cdot\right\VERT)$ is $q$-uniformly smooth.
\item[(b)] There exists $C>0$ such that
$$
\|f\|_{L^p(\mathbb R^n,B)} \leq C\left\|g^\alpha_{q,W^\mathfrak{H}_t;B}(f)\right\|_{L^p(\mathbb R^n)},\;\; f \in L^p(\mathbb R^n) \otimes B.
$$
\end{enumerate}
\end{thm}
\begin{proof}
$(a) \Rightarrow (b)$ In order to prove this we can proceed as in the proof of Theorem \ref{thm6}.

$(b) \Rightarrow (a)$. By taking into account \eqref{H1}, to see this property we can argue as in the proof of Theorem \ref{thm7} by using Theorem \ref{thm2.1}.
\end{proof}

\begin{thm}\label{thm2.4}
Let $B$ be a Banach space, $1<q \leq 2$, $1<p<\infty$, and $\alpha >0$. The following assertions are equivalent.
\begin{enumerate}
\item[(a)] There exists a norm $\left\VERT \cdot\right\VERT$ on $B$ that is equivalent to $\|\cdot\|$ and and such that $(B,\left\VERT \cdot\right\VERT)$ is $q$-uniformly smooth.
\item[(b)] There exists $C>0$ such that
$$
\|f\|_{L^p(\mathbb R^n,B)} \leq C\left\|\mathcal{A}^\alpha_{q,W^\mathfrak{H}_t;B}(f)\right\|_{L^p(\mathbb R^n)}, \;\; f \in L^p(\mathbb R^n) \otimes B.
$$
\end{enumerate}
\end{thm}
\begin{proof}
$(a) \Rightarrow (b)$. Let $f,g \in span\{h_k\}_{k \in \mathbb N^n}$. Here, for every $k \in \mathbb N^n$, $h_k$ denotes the k-th Hermite function.

We have that
$$
W^\mathfrak{H}_t(f) = \sum_{k \in \mathbb N^n} e^{-\lambda_kt} c_k(f)h_k,\;\;t>0,
$$
where for every $k=(k_1,\ldots k_n)\in \mathbb N^n$, $\lambda_k$ denotes the eigenvalue for the Hermite operator associated to $h_k$, i.e., $\lambda_k=2(k_1 + \ldots+k_n)+n$, and $c_k(f)= \int_{\mathbb R^n} f(x) h_k(x)\,dx$. Note that the series is actually a finite sum.

Let $m \in \mathbb N$ such that $m-1 \leq \alpha < m$. We can write
\begin{eqnarray*}
\partial_t^\alpha W^\mathfrak{H}_t(f)(x) & = & \sum_{k \in \mathbb N^n} c_k(f) h_k(x) \frac{1}{\Gamma(m-\alpha)}\int_t^\infty \partial_u^m(e^{-\lambda_ku})(u-t)^{m-\alpha -1} du\\
&=& \sum_{k \in \mathbb N^n} c_k(f) h_k(x) \frac{(-1)^m\lambda^m_k}{\Gamma(m-\alpha)}\int_t^\infty e^{-\lambda_ku}(u-t)^{m-\alpha -1} du\\
&=& \sum_{k \in \mathbb N^n} c_k(f) h_k(x) \frac{(-1)^m\lambda^m_k}{\Gamma(m-\alpha)}\int_{0}^\infty e^{-\lambda_ku}u^{m-\alpha -1} du e^{-\lambda_k t}\\
&=& (-1)^m\sum_{k \in \mathbb N^n} c_k(f) h_k(x) \lambda_k^\alpha e^{-\lambda_kt}, \;\; x \in \mathbb R^n\;\; and \;\; t>0.
\end{eqnarray*}
Orthonormality properties of Hermite functions lead to
\begin{eqnarray*}
&& \int_0^\infty \int_{\mathbb R^n} t^\alpha \partial_t^\alpha W_t^\mathfrak{H}(f)(x) t^\alpha \partial_t^\alpha W_t^\mathfrak{H}(g)(x)\,dx\frac{dt}{t} \\
&& \hspace{1cm}= \sum_{k,m \in \mathbb N^n} c_k(f) c_m(g)(\lambda_k\lambda_m)^\alpha \int_0^\infty \int_{\mathbb R^n} t^{2\alpha -1}h_k(x)h_m(x) e^{-(\lambda_k + \lambda_m)t} dxdt\\
&& \hspace{1cm}=\sum_{k\in \mathbb N^n} c_k(f)c_k(g) \lambda_k^{2\alpha}\int_0^\infty t^{2\alpha-1}e^{-2\lambda_kt} dt \\
&& \hspace{1cm} =\frac{\Gamma(2\alpha)}{2^{2\alpha}} \sum_{k\in \mathbb N^n} c_k(f)c_k(g) = \frac{\Gamma(2\alpha)}{2^{2\alpha}} \int_{\mathbb R^n} f(x)g(x) dx.
\end{eqnarray*}

According to Theorem \ref{thm2.1} with $B=\mathbb C$ we get
\begin{eqnarray*}
\left| \int_0^\infty \int_{\mathbb R^n} t^\alpha \partial_t^\alpha W_t^\mathfrak{H}(f)(x) t^\alpha \partial_t^\alpha W_t^\mathfrak{H}(g)(x)\,dx\frac{dt}{t}\right|
& \leq & \int_{\mathbb R^n} g^\alpha_{2,W^\mathfrak{H}_t}(f) g^\alpha_{2,W_t^\mathfrak{H}}(g)(x)dx \\
&\leq &\|g^\alpha_{2,W^\mathfrak{H}_t;\mathbb C}(f)\|_{L^2(\mathbb R^n)} \| g^\alpha_{2,W^{\mathfrak{H}}_t;\mathbb C}(g)\|_{L^2(\mathbb R^n)} \\
&\leq &\:C\|f\|_{L^2(\mathbb R^n)} \|g\|_{L^2(\mathbb R^n)}, \;\; f,g \in L^2(\mathbb R^n).
\end{eqnarray*}

Since span $\{h_n\}_{n \in \mathbb N}$ is a dense subspace of $L^2(\mathbb R^n)$, we deduce that
$$
\int_0^\infty\int_{\mathbb R^n} t^\alpha\partial_t^\alpha W_t^\mathfrak{H}(f)(x)  t^\alpha\partial_t^\alpha W_t^\mathfrak{H}(g)(x) dx \frac{dt}{t} = \frac{\Gamma(2\alpha)}{2^{2\alpha}}\int_{\mathbb R^n}f(x)g(x)dx,\;\; f,g \in L^{2}(\mathbb R^n)\otimes B.
$$
We now prove that $(a)$ implies $(b)$ by proceeding as in the proof of the corresponding part in Theorem \ref{thm9}

$(b) \Rightarrow (a)$. By using \eqref{H1} we can prove this property following the same ideas than in the second part of the proof of Theorem \ref{thm9}.
\end{proof}

\section{Results involving heat semigroup for Laguerre operators.}

In this section we prove characterizations for uniformly convex and smooth Banach spaces by using $g$-functions and area integrals involving heat semigroups for Laguerre operators (Theorems \ref{thm1.11} and \ref{thm1.12}).
Let $\beta > -\frac{1}{2}$. By $\mathcal{L}_\beta$ we denote the Laguerre operator. We recall that the semigroup $\left\{W_t^{\mathcal{L}_\beta}\right\}_{t>0}$ is not Markovian
\begin{thm}\label{thm3.1}
Let $B$ be a Banach space, $2 \leq q < \infty$, $1<p<\infty$, $\beta>-\frac{1}{2}$ and $\alpha > 0$. The following assertion are equivalent.
\begin{enumerate}
\item[(a)] There exists a norm $\left\VERT \cdot \right\VERT$ on $B$ that is equivalent to $\|\cdot\|$ and such that $(B,\left\VERT\cdot \right\VERT)$ is $q$-uniformly convex.
\item[(b)] There exists $C>0$ such that
$$
\left\| g^\alpha_{q,W_t^{\mathcal{L}_\beta};B}(f)\right\|_{L^p(0,\infty)} \leq C\|f\|_{L^p((0,\infty),B)},\;\; f \in L^p((0,\infty),B).
$$
\end{enumerate}
\end{thm}
\begin{proof}
According to \cite[p. 123]{Le}, for every $k \in \mathbb N$,
\begin{equation}\label{L1}
\sqrt{2\pi z} e^{-z} I_\beta(z) = \sum^k_{l=0} \frac{(-1)^l[\beta,l]}{(2z)^l} + \mathcal{O}\left(\frac{1}{z^{k+1}}\right),\;\;z\in (0,\infty),
\end{equation}
where $[\beta,0] =1$ and
$$
[\beta,l]= \frac{(4\beta^2-1)(4\beta^2-3^2)\ldots(4\beta^2-(2l-1)^2)}{2^{2l}\Gamma(l+1)},\;\; l \in \mathbb N,\;l\geq 1.
$$
By \eqref{L1} with $k=0$ we get
\begin{eqnarray*}
0\leq W_t^{\mathcal{L}_\beta}(x,y) &\leq & C\left(\frac{e^{-t}}{1-e^{-2t}}\right)^{\frac{1}{2}} e^{-\frac{1}{2}(x^2+y^2)\frac{1+e^{-2t}}{1-e^{-2t}}+2xy\frac{e^{-t}}{1-e^{-2t}}}\\
&\leq & \frac{C}{t^{1/2}}e^{-\frac{(x-y)^2(1+e^{-2t})+2xy(1-e^{-t})^2}{1-e^{-2t}}}\\
&\leq & \frac{C}{t^{1/2}}e^{-c\frac{(x-y)^2}{t}},\;\;x,y,t \in (0,\infty).
\end{eqnarray*}
The Laguerre operator $\mathcal{L}_\beta$ is a self adjoint and positive operator in $L^2(0,\infty)$. Then,
according to \cite[3.0.1]{CS}, $-\mathcal{L}_\beta$ generates bounded analytic semigroup in $L^2(0,\infty)$. By \cite[Theorem 6.17]{Ou}, for every $k \in \mathbb N$,
\begin{equation}\label{L2}
\left|\frac{d^k}{dt^k} W_t^{\mathcal L_\beta}(x,t)\right| \leq \frac{C}{t^{k+\frac{1}{2}}} e^{-c\frac{(x-y)^2}{t}},\;\; x,y,t \in (0,\infty).
\end{equation}
By using \eqref{L2} and proceeding as in the proof of Theorem \ref{thm2.1} we can see that the property $(b)$ is equivalent than the following one:

$(b')$ There exists $C>0$ such that
$$
\left\|g^1_{q,W_t^{\mathcal{L}_\beta};B}(f)\right\|_{L^p(\mathbb R^n)} \leq \|f\|_{L^p(\mathbb R^n,B)},\;\; f\in L^p(\mathbb R^n,B).
$$

$(b') \Rightarrow (a)$. By $\{P_t^{\mathcal{L}_\alpha}\}_{t>0}$ we denote the Poisson semigroup associated with $\mathcal{L}_\alpha$. From the subordination formula we deduce that, for every $f \in L^r(0,\infty)$, $1 \leq r \leq \infty$,
$$
g^1_{q,P_t^{\mathcal{L}_\beta};B}(f) \leq g^1_{q,W_t^{\mathcal{L}_\beta};B}(f).
$$
From \cite[Theorem 1]{BFRST} it follows $(a)$ provided that $(b')$ holds.

$(a) \Rightarrow (b')$. Suppose that $(a)$ is satisfied. According to Theorem \ref{thm2.1} there exists $C>0$ such that
$$
\left\|g^1_{q,W_t^{\mathcal{H}};B}(f)\right\|_{L^p(\mathbb R^n)} \leq C\|f\|_{L^p(\mathbb R^n,B)},\;\; f\in L^p(\mathbb R,B).
$$
Then, it is clear that
\begin{equation}\label{L3}
\left\|g^1_{q,W_t^{\mathfrak{H}/2};B}(f)\right\|_{L^p(\mathbb R^n)} \leq C\|f\|_{L^p(\mathbb R^n,B)},\;\; f\in L^p(\mathbb R,B).
\end{equation}
In order to prove $(b')$ we follow the ideas \cite[\S 4]{BCFR}.

We define, for every $f \in L^p(\mathbb R,B)$,
$$
G_{W_t^{\mathfrak{H}/2}}(f)(x,t) = t\partial_tW_{t}^{\mathfrak{H}/2}(f)(x),\;\; x\in \mathbb R \;and\;\; t>0,
$$
and, for every $f \in L^p((0,\infty),B)$,
$$
G_{W_t^{\mathcal{L}_\beta}}(f)(x,t) = t\partial_tW_t^{\mathcal{L}_\beta}(f)(x),\;\; x\in (0,\infty) \;and\;\; t>0.
$$
If $f$ is a measurable function on $(0,\infty)$ we define
$$
\tilde{f}(x)= \left\{\begin{array}{ll}
                      f(x), & x\in (0,\infty)\\
                      0, & x\in(-\infty,0]
                      \end{array}\right.
$$
Let $f\in L^p((0,\infty),B)$. Minkowski integral inequality allows us to write
\begin{eqnarray*}
\left\|G_{W_t^{\mathcal{L}_\beta}}(f)(x,\cdot)\right.& - &\left.G_{W_t^{\mathfrak{H}/2}}(\tilde{f})(x,\cdot)\right\|_{L^q((0,\infty)\frac{dt}{t},B)}\\
&\leq& \int \limits_{(0,\frac{x}{2}) \cup(2x,\infty)} \hspace{-8mm}\|f(y)\|_B\left(\|t\partial_tW_t^{\mathcal{L}_\beta}(x,y)\|_{L^q((0,\infty),\frac{dt}{t})}+\|t\partial_tW^{\mathfrak{H}/2}(x,y)\|_{L^q((0,\infty),\frac{dt}{t})}\right)dy  \\
 &&\hspace{4mm} +\int^{2x}_{\frac{x}{2}} \|f(y)\|_B\|t\partial_tW_t^{\mathfrak{H}/2}(x,y) - t \partial_tW_t^{\mathcal{L}_\beta}(x,y)\|_{L^q((0,\infty),\frac{dt}{t})}dy\\
&=&T_1(f)(x) + T_2(f)(x),\;\; x \in (0,\infty).
\end{eqnarray*}

According to \eqref{H1} and \eqref{L1} we get
\begin{eqnarray*}
\|t\partial_tW_t^{\frac{\mathcal{H}}{2}} (x,y)\|_{L^q((0,\infty),\frac{dt}{t})}&+&\|t\partial_tW_t^{\mathcal{L}_\beta}(x,y)\|_{L^q((0,\infty),\frac{dt}{t})}\\
&\leq & C\left(\int_0^\infty \frac{e^{-c\frac{|x-y|^2}{t}}}{t^{n\frac{q}{2}+1}}dt \right)^{\frac{1}{q}} \leq \frac{C}{|x-y|},\;\;x,y \in (0,\infty),\;\;x\not= y.
\end{eqnarray*}
Then,
$$
T_1(f) \leq C(H_0(\|f\|) + H_\infty(\|f\|),
$$
where $H_0$ and $H_\infty$ denotes the classical Hardy operators defined by
$$
H_0(g)(x)= \frac{1}{x}\int_0^x g(y)dy, \;\;{\rm{and}}\;\; H_\infty(g)(x)= \int_x^\infty \frac{g(y)}{y} dy,\;\; x \in (0,\infty).
$$
As it is wellknown $H_0$ and $H_\infty$ are bounded operators from $L^p(0,\infty)$ into itself.

We obtain
$$
\|T_1(f)\|_{L^p(0,\infty)} \leq C \|f\|_{L^p((0,\infty),B)}.
$$
According to \cite[(35), (36) and (38)]{BCFR} we have that
$$
\left|t\partial_tW_t^{\mathfrak{H}/2}(x,y) - t \partial_tW_t^{\mathcal{L}_\beta}(x,y)\right| \leq C \frac{e^{-c\frac{x^2+y^2}{t}}}{\sqrt{t}},\; t,x,y \in (0,\infty)\;\;{\rm{and}}\;\; \frac{xy}{t} \leq 1,
$$
and
$$
\left|t\partial_tW_t^{\frac{H}{2}}(x,y) - t \partial_tW_t^{\mathcal{L}_\beta}(x,y)\right| \leq C \frac{e^{-c\frac{(x-y)^2}{t}}}{(xyt)^{1/4}},\; t,x,y \in (0,\infty)\;\;{\rm{and}}\;\; \frac{xy}{t} \geq 1,
$$
By using these estimates we get
\begin{eqnarray*}
&& \left\|t\partial_t W_t^{\mathfrak{H}/2}(x,y)- t \partial_tW_t^{\mathcal{L}_\beta}(x,y)\right\|_{L^q((0,\infty),\frac{dt}{t})}\\
&&\hspace{1cm}\leq \left(\int_{0,xy \leq t}^\infty\left(\left|t\partial_tW_t^{\mathfrak{H}/2}(x,y)\right| +\left|t\partial_tW_t^{\mathcal{L}_\beta}(x,y)\right|\right)^q \frac{dt}{t}\right.\\
&&\hspace{15mm} +\left.\int_{0,xy \geq t}^\infty\left|t\partial_tW_t^{\mathfrak{H}/2}(x,y)-t\partial_tW_t^{\mathcal{L}_\beta}(x,y)\right|^q \frac{dt}{t}\right)^{\frac{1}{q}}\\
&&\hspace{1cm}\leq C \left(\int_0^\infty \frac{e^{-c\frac{(x^2+y^2)}{t}}}{t^{q/2+1}}dt + \frac{1}{(xy)^{q/4}}\int_0^\infty \frac{e^{-c\frac{(x-y)^2}{t}}}{t^{q/4+1}}dt\right)^{\frac{1}{q}}\\
&&\hspace{1cm}\leq C \left(\frac{1}{\sqrt{x^2+y^2}}+\frac{1}{(xy)^{1/4}\sqrt{|x-y|}}\right),\;\; x,y\in(0,\infty).
\end{eqnarray*}
Then,
$$
\left\|t\partial_tW^{\mathfrak{H}/2}_t(x,y) - t \partial_tW^{\mathcal{L}_\beta}_t(x,y)\right\|_{L^q((0,\infty),\frac{dt}{t})} \leq C\frac{1}{y}\left(1+\sqrt{\frac{y}{|x-y|}}\right),\;\; 0 <\frac{x}{2}<y<2x.
$$
We define the operator
$$
N(g)(x) = \int_{\frac{x}{2}}^{2x} \frac{1}{y}\left(1+\sqrt{\frac{y}{|x-y|}}\right)g(y) dy,\;\; x \in (0,\infty).
$$
We have that
$$
\int_{\frac{x}{2}}^{2x}\frac{1}{y} \left(1+\sqrt{\frac{y}{|x-y|}}\right)dy = \int_{\frac{1}{2}}^2\left(1+\sqrt{\frac{1}{|1-u|}}\right)du < \infty,\;\; x \in (0,\infty).
$$
Jensen inequality leads to
$$
\|N(g)\|_{L^p(0,\infty)} \leq C\|g\|_{L^p(0,\infty)},\;\; g \in L^p(0,\infty).
$$
It follows that
$$
\|T_2(f)\|_{L^p(0,\infty)} \leq C\|f\|_{L^p((0,\infty),B)}.
$$

We conclude that
\begin{equation}\label{L4}
\left\| G_{W_t^{\mathcal{L}_\beta}}(f) - G_{W_t^{\mathfrak{H}/2}}(\tilde{f}) \right\|_{L^p((0,\infty),L^q((0,\infty),\frac{dt}{t},B))} \leq C\|f\|_{L^p((0,\infty),B)}.
\end{equation}
By combining \eqref{L3} and \eqref{L4} we obtain
$$
\left\|g^1_{q,W_t^{\mathcal{L}_\beta};B}(f)\right\|_{L^p(0,\infty)} \leq C\|f\|_{L^p((0,\infty),B)}.
$$
Thus $(b')$ is established.
\end{proof}
\begin{thm}\label{thm3.2}
Let $B$ be a Banach space, $2 \leq q \leq \infty$, $1<p<\infty$, $ \beta > -\frac{1}{2}$ and $\alpha >0$.
The following assertions are equivalent
\begin{enumerate}
\item[(a)] There exists  a norm $\left\VERT\cdot\right\VERT$ in $B$ that is equivalent to $\|\cdot\|$ and such that $(B,\left\VERT\cdot\right\VERT)$ is $q$-uniformly convex.
\item[(b)] There exists $C>0$ such that
$$
\left\| \mathcal{A}^\alpha_{q,W_t^{\mathcal{L}_\beta};B}f\right\|_{L^p(0,\infty)} \leq C \|f\|_{L^p((0,\infty),B)},\;\; f \in L^p((0,\infty),B).
$$
\end{enumerate}
\end{thm}
\begin{proof}
This theorem can be proved by taking into account \eqref{L2} and by using Calder\'on-Zygmund theorem for Banach valued singular integrals. We can proceed following the procedure developed in the proof of Theorem \ref{thm4} and Theorem \ref{thm2.2} by applying Theorem \ref{thm3.1}
\end{proof}

The subspace of fixed points for the semigroup $\{W_t^{\mathcal{L}_\beta}\}_{t>0}$ in $L^p(0,\infty)$, $1<p<\infty$, reduces to $\{0\}$.

\begin{thm}\label{thm3.3}
Let $B$ be a Banach space, $1<q\leq 2$, $1<p<\infty$, $\beta >-\frac{1}{2}$, and $\alpha >0$. The following assertion are equivalent.
\begin{enumerate}
\item[(a)] There exists a norm $\left\VERT\cdot\right\VERT$ on $B$ that is equivalent to $\|\cdot\|$ and such that $(B,\left\VERT\cdot\right\VERT)$ is $q$-uniformly smooth.
\item[(b)] There exists $C>0$ such that
$$
\|f\|_{L^p((0,\infty),B)} \leq C \left\|g^\alpha_{q,W_t^{\mathcal{L}_\beta};B}(f)\right\|_{L^p(0,\infty)},\;\; f \in L^p(0,\infty)\otimes B.
$$
\end{enumerate}
\end{thm}
\begin{proof}
This result can be proved by using \eqref{L2} following the lines in the proof of Theorem \ref{thm6}.
\end{proof}

\begin{thm}\label{thm3.4}
Let $B$ be a Banach space, $1<q\leq 2$, $1<p<\infty$, $\beta > -\frac{1}{2}$, and $\alpha >0$. The following assertions are equivalent.
\begin{enumerate}
\item[(a)] There exists a norm $\left\VERT\cdot\right\VERT$ on $B$ that is equivalent to $\|\cdot\|$ and such that $(B;\left\VERT\cdot\right\VERT)$ is $q$-uniformly smooth.
\item[(b)] There exists $C>0$ such that
$$
\|f\|_{L^p((0,\infty),B)} \leq C \left\|\mathcal{A}^\alpha_{q,W_t^{\mathcal{L}_\beta};B}(f)\right\|_{L^p(0,\infty)},\;\; f \in L^p(0,\infty)\otimes B.
$$
\end{enumerate}
\end{thm}
\begin{proof}
This result can be proved as Theorem \ref{thm2.4} by using Laguerre functions associated with $\mathcal{L}_\alpha$ instead of Hermite functions
\end{proof}


\begin{thebibliography}{10}

\bibitem{AST}
{\sc I.~Abu-Falahah, P.~R. Stinga, and J.~L. Torrea}, {\em Square functions
  associated to {S}chr\"{o}dinger operators}, Studia Math., 203 (2011),
  pp.~171--194.

\bibitem{AT1}
{\sc I.~{Abu-Falahah} and J.~L. {Torrea}}, {\em {Hermite function expansions
  versus Hermite polynomial expansions.}}, {Glasg. Math. J.}, 48 (2006),
  pp.~203--215.

\bibitem{AHM}
{\sc P.~Auscher, S.~Hofmann, and J.-M. Martell}, {\em Vertical versus conical
  square functions}, Trans. Amer. Math. Soc., 364 (2012), pp.~5469--5489.

\bibitem{BCCR}
{\sc J.~J. {Betancor}, A.~J. {Castro}, J.~{Curbelo}, and
  L.~{Rodr\'{\i}guez-Mesa}}, {\em {Characterization of UMD Banach spaces by
  imaginary powers of Hermite and Laguerre operators.}}, {Complex Anal. Oper.
  Theory}, 7 (2013), pp.~1019--1048.

\bibitem{BCFR}
{\sc J.~J. Betancor, A.~J. Castro, J.~C. Fari\~{n}a, and
  L.~Rodr\'{i}guez-Mesa}, {\em U{MD} {B}anach spaces and square functions
  associated with heat semigroups for {S}chr\"{o}dinger, {H}ermite, and
  {L}aguerre operators}, Math. Nachr., 289 (2016), pp.~410--435.

\bibitem{BCFST}
{\sc J.~J. Betancor, R.~Crescimbeni, J.~C. Fari\~{n}a, P.~R. Stinga, and J.~L.
  Torrea}, {\em A {$T1$} criterion for {H}ermite-{C}alder\'{o}n-{Z}ygmund
  operators on the {$BMO_H(\Bbb R^n)$} space and applications}, Ann. Sc. Norm.
  Super. Pisa Cl. Sci. (5), 12 (2013), pp.~157--187.

\bibitem{BFRST}
{\sc J.~J. {Betancor}, J.~C. {Fari\~na}, L.~{Rodr{\'\i}guez-Mesa},
  A.~{Sanabria}, and J.~L. {Torrea}}, {\em {Lusin type and cotype for Laguerre
  $g$-functions.}}, {Isr. J. Math.}, 182 (2011), pp.~1--30.

\bibitem{BFMT}
{\sc J.~J. Betancor, J.~C. Fari{\~{n}}a, T.~Mart{\'i}nez, and J.~L. Torrea},
  {\em Riesz transform and g-function associated with bessel operators and
  their appropriate banach spaces}, Israel Journal of Mathematics, 157 (2007),
  pp.~259--282.

\bibitem{BMR}
{\sc J.~J. {Betancor}, S.~M. {Molina}, and L.~{Rodr{\'\i}guez-Mesa}}, {\em
  {Area Littlewood-Paley functions associated with Hermite and Laguerre
  operators.}}, {Potential Anal.}, 34 (2011), pp.~345--369.

\bibitem{Bo}
{\sc J.~Bourgain}, {\em Some remarks on {B}anach spaces in which martingale
  difference sequences are unconditional}, Ark. Mat., 21 (1983), pp.~163--168.

\bibitem{Bu1}
{\sc D.~L. {Burkholder}}, {\em {A geometrical characterization of Banach spaces
  in which martingale difference sequences are unconditional.}}, {Ann.
  Probab.}, 9 (1981), pp.~997--1011.

\bibitem{Bu}
{\sc D.~L. Burkholder}, {\em Martingales and {F}ourier analysis in {B}anach
  spaces}, in Probability and analysis ({V}arenna, 1985), vol.~1206 of Lecture
  Notes in Math., Springer, Berlin, 1986, pp.~61--108.

\bibitem{CMS}
{\sc R.~R. Coifman, Y.~Meyer, and E.~M. Stein}, {\em Some new function spaces
  and their applications to harmonic analysis}, J. Funct. Anal., 62 (1985),
  pp.~304--335.

\bibitem{CD}
{\sc T.~Coulhon and X.~T. Duong}, {\em Maximal regularity and kernel bounds:
  observations on a theorem by {H}ieber and {P}r\"{u}ss}, Adv. Differential
  Equations, 5 (2000), pp.~343--368.

\bibitem{CS}
{\sc T.~Coulhon and A.~Sikora}, {\em Gaussian heat kernel upper bounds via the
  {P}hragm\'{e}n-{L}indel\"{o}f theorem}, Proc. Lond. Math. Soc. (3), 96
  (2008), pp.~507--544.

\bibitem{DGMTZ}
{\sc J.~Dziuba{\'{n}}ski, G.~Garrig{\'o}s, T.~Mart{\'i}nez, J.~L. Torrea, and
  J.~Zienkiewicz}, {\em Bmo spaces related to schr{\"o}dinger operators with
  potentials satisfying a reverse h{\"o}lder inequality}, Mathematische
  Zeitschrift, 249 (2005), pp.~329--356.

\bibitem{Gu}
{\sc S.~{Guerre-Delabriere}}, {\em {Some remarks on complex powers of
  (-$\Delta$ ) and UMD spaces.}}, {Ill. J. Math.}, 35 (1991), pp.~401--407.

\bibitem{HN}
{\sc T.~{Hyt{\"o}nen} and A.~{Naor}}, {\em {Heat flow and quantitative
  differentiation}}, ArXiv e-prints,  (2016).

\bibitem{Kw}
{\sc S.~{Kwapien}}, {\em {Isomorphic characterizations of inner product spaces
  by orthogonal series with vector valued coefficients.}}, {Stud. Math.}, 44
  (1972), pp.~583--595.

\bibitem{Le}
{\sc N.~N. Lebedev}, {\em Special functions and their applications}, Dover
  Publications, Inc., New York, 1972.
\newblock Revised edition, translated from the Russian and edited by Richard A.
  Silverman, Unabridged and corrected republication.

\bibitem{LT}
{\sc J.~{Lindenstrauss} and L.~{Tzafriri}}, {\em {Classical Banach spaces. II:
  Function spaces.}}
\newblock {Ergebnisse der Mathematik und ihrer Grenzgebiete. 97.
  Berlin-Heidelberg-New York: Springer-Verlag. X, 243 p. (1979).}, 1979.

\bibitem{MTX}
{\sc T.~Mart\'{i}nez, J.~L. Torrea, and Q.~Xu}, {\em Vector-valued
  {L}ittlewood-{P}aley-{S}tein theory for semigroups}, Adv. Math., 203 (2006),
  pp.~430--475.

\bibitem{NS}
{\sc A.~Nowak and K.~Stempak}, {\em On {$L^p$}-contractivity of {L}aguerre
  semigroups}, Illinois J. Math., 56 (2012), pp.~433--452.

\bibitem{Ou}
{\sc E.~M. Ouhabaz}, {\em Analysis of heat equations on domains}, vol.~31 of
  London Mathematical Society Monographs Series, Princeton University Press,
  Princeton, NJ, 2005.

\bibitem{OX}
{\sc C.~Ouyang and Q.~Xu}, {\em B{MO} functions and {C}arleson measures with
  values in uniformly convex spaces}, Canad. J. Math., 62 (2010), pp.~827--844.

\bibitem{Pi5}
{\sc G.~{Pisier}}, {\em {Martingales with values in uniformly convex spaces.}},
  {Isr. J. Math.}, 20 (1975), pp.~326--350.

\bibitem{Pi2}
\leavevmode\vrule height 2pt depth -1.6pt width 23pt, {\em {Probabilistic
  methods in the geometry of Banach spaces.}}
\newblock {Probability and analysis, Lect. Sess. C.I.M.E., Varenna/Italy 1985,
  Lect. Notes Math. 1206, 167-241 (1986).}, 1986.

\bibitem{Pi3}
\leavevmode\vrule height 2pt depth -1.6pt width 23pt, {\em {Martingales in
  Banach spaces.}}, vol.~155, Cambridge: Cambridge University Press, 2016.

\bibitem{Rub}
{\sc J.~L. {Rubio de Francia}}, {\em {Martingale and integral transforms of
  Banach space valued functions.}}
\newblock {Probability and Banach spaces, Proc. Conf., Zaragoza/Spain 1985,
  Lect. Notes Math. 1221, 195-202 (1986).}, 1986.

\bibitem{RRT}
{\sc J.~L. Rubio~de Francia, F.~J. Ruiz, and J.~L. Torrea}, {\em
  Calder\'{o}n-{Z}ygmund theory for operator-valued kernels}, Adv. in Math., 62
  (1986), pp.~7--48.

\bibitem{Sh}
{\sc Z.~W. Shen}, {\em {$L^p$} estimates for {S}chr\"{o}dinger operators with
  certain potentials}, Ann. Inst. Fourier (Grenoble), 45 (1995), pp.~513--546.

\bibitem{SteinLp}
{\sc E.~M. Stein}, {\em Topics in harmonic analysis related to the
  {L}ittlewood-{P}aley theory}, Annals of Mathematics Studies, No. 63,
  Princeton University Press, Princeton, N.J.; University of Tokyo Press,
  Tokyo, 1970.

\bibitem{ST1}
{\sc K.~{Stempak} and J.~L. {Torrea}}, {\em {BMO results for operators
  associated to Hermite expansions.}}, {Ill. J. Math.}, 49 (2005),
  pp.~1111--1131.

\bibitem{ST2}
{\sc K.~{Stempak} and J.~L. {Torrea}}, {\em {On g-functions for Hermite
  function expansions.}}, {Acta Math. Hung.}, 109 (2005), pp.~99--125.

\bibitem{Sz3}
{\sc G.~{Szeg\"o}}, {\em {Orthogonal polynomials. 4th ed.}}, vol.~23, American
  Mathematical Society (AMS), Providence, RI, 1975.

\bibitem{TZ}
{\sc J.~L. Torrea and C.~Zhang}, {\em Fractional vector-valued
  {L}ittlewood-{P}aley-{S}tein theory for semigroups}, Proc. Roy. Soc.
  Edinburgh Sect. A, 144 (2014), pp.~637--667.

\bibitem{Wr}
{\sc B.~J. {Wr\'obel}}, {\em {On $g$-functions for Laguerre function expansions
  of Hermite type.}}, {Proc. Indian Acad. Sci., Math. Sci.}, 121 (2011),
  pp.~45--75.

\bibitem{Xu98}
{\sc Q.~Xu}, {\em Littlewood-{P}aley theory for functions with values in
  uniformly convex spaces}, J. Reine Angew. Math., 504 (1998), pp.~195--226.

\bibitem{Xu18}
{\sc Q.~{Xu}}, {\em {Vector-valued Littewood-Paley-Stein theory for semigroups
  II}}, arXiv e-prints, to appear in Int. Math. Res. Not.,  (2018).

\end{thebibliography}
\end{document}